\documentclass[12pt]{amsart}
\usepackage[latin9]{inputenc}
\usepackage{amstext}
\usepackage{amsthm}
\usepackage{amssymb}
\usepackage[unicode=true,
 bookmarks=false,
 breaklinks=false,pdfborder={0 0 1},backref=section,colorlinks=false]{hyperref}
\usepackage{verbatim}
\usepackage{amstext}
\usepackage{amsthm}
\usepackage{eurosym}
\usepackage{latexsym}
\usepackage{amsthm}
\usepackage{amsfonts}
\usepackage{ulem}

\makeatletter
\numberwithin{equation}{section}
\numberwithin{figure}{section}
\numberwithin{equation}{section}
\numberwithin{figure}{section}
\numberwithin{equation}{section}
\newtheorem{theorem}{Theorem}

\newtheorem{lemma}[theorem]{Lemma}
\newtheorem{corollary}[theorem]{Corollary}
\newtheorem{proposition}[theorem]{Proposition}

\theoremstyle{definition}
\newtheorem{definition}[theorem]{Definition}
\newtheorem{example}[theorem]{Example}
\newtheorem*{acknowledgements*}{Acknowledgements}
\theoremstyle{remark}
\newtheorem{remark}[theorem]{Remark}

\numberwithin{theorem}{section}   
\thanks{The work of the first author was completed as a part of the implementation of the development program of the Volga Region Scientific and Educational Mathematical Center (agreement
no. 075-02-2023-931.}
\thanks{}
\thanks{}
\makeatother

\begin{document}
\title[$S$-decomposable Banach lattices]{$S$-decomposable Banach lattices,
optimal sequence spaces and interpolation}
\author{Sergey V. Astashkin}
\address{Astashkin: Department of Mathematics, Samara National Research
University, Moskovskoye shosse 34, 443086, Samara, Russian Federation;
Department of Mathematics, Bahcesehir University, 34353, Istanbul, Turkey}
\email{astash56@mail.ru}
\author{Per G. Nilsson}
\address{Nilsson: Stockholm, Sweden}
\email{pgn@plntx.com}
\date{\today }

\begin{abstract}
We investigate connections between upper/lower estimates for Banach lattices
and the notion of relative $s$-decomposability, which has roots in
interpolation theory. To get a characterization of relatively $s$%
-decomposable Banach lattices in terms of the above estimates, we assign to
each Banach lattice $X$ two sequence spaces $X_{U}$ and $X_{L}$ that are
largely determined by the set of $p$, for which $l_{p}$ is finitely lattice
representable in $X$. As an application, we obtain an orbital factorization
of relative $K$-functional estimates for Banach couples $\vec{X}%
=(X_{0},X_{1})$ and $\vec{Y}=(Y_{0},Y_{1})$ through some suitable couples of
weighted $L_{{p}}$-spaces provided if $X_{i},Y_{i}$ are relatively $s$%
-decomposable for $i=0,1$.

Also, we undertake a detailed study of the properties of optimal upper and
lower sequence spaces $X_{U}$ and $X_{L}$, and, in particular, prove that
these spaces are rearrangement invariant. In the Appendix, a description of
the optimal upper sequence space for a separable Orlicz space as a certain
intersection of some special Musielak-Orlicz sequence spaces is given.
\end{abstract}

\maketitle

\section{Introduction}

\label{Intro}

This paper has roots in the classification problem of the interpolation
theory of operators, see for instance Peetre \cite{Pee71}, i.e., the problem
of identification of equivalence classes of Banach couples with the "same"
interpolation structure\footnote{%
For standard definitions and notation used in the interpolation theory, see
e.g. \cite{BeLo76},\cite{BK91}, \cite{BrKeSe88}, \cite{Pee71}.}.
Specifically, there are close connections of the topic of this paper with
the so-called Calderón-Mityagin property of Banach couples, which often
allows to describe effectively the class of all interpolation spaces with
respect to them. Let us recall this notion. $\medskip $

Assuming that $\overrightarrow{X}=(X_{0},X_{1})$ and $\overrightarrow{Y}%
=(Y_{0},Y_{1})$ are Banach couples, we consider the following two properties
of elements $x\in X_{0}+X_{1}$ and $y\in Y_{0}+Y_{1}$:

\begin{equation}
y=Tx\text{ for some bounded linear operator }T:\,X_{i}\rightarrow
Y_{i},\;i=0,1,  \label{kmon1}
\end{equation}%
and 
\begin{equation}
K(t,y;\overrightarrow{Y})\leq C\cdot K(t,x;\overrightarrow{X})\text{ for
some constant }C\text{ and }t>0.  \label{kmon2}
\end{equation}%
Here, $K(t,x;\overrightarrow{X})$ is the Peetre $K$-functional defined for
all $x\in X_{0}+X_{1}$ and $t>0$ by 
\begin{equation*}
K(t,x;\overrightarrow{X}):=\inf \{\Vert x_{0}\Vert _{X_{0}}+t\Vert
x_{1}\Vert _{X_{1}}:\,x=x_{0}+x_{1},\,x_{i}\in X_{i}\}.
\end{equation*}

It is easy to check that condition \eqref{kmon1} implies \eqref{kmon2} with $%
C=\max_{i=0,1}\Vert T\Vert _{{X_{i}}\rightarrow {Y_{i}}}$. If the converse
implication holds (for all $x\in X_{0}+X_{1}$ and $y\in Y_{0}+Y_{1}$), the
couples $\overrightarrow{X}$ and $\overrightarrow{Y}$ are said to have the 
\textit{relative Calderón-Mityagin property} (in brief, \textit{$\mathcal{C-M%
}$ property}). In this case all relative interpolation spaces with respect
to these couples can be described by using $K$-functionals in a precise
quantitative manner, see for instance \cite[Theorem 4.1.11]{BK91}.

$\medskip $

A result of this type related to the finite dimensional couples $\left(
l_{1}^{n},l_{\infty }^{n}\right) $ goes back to Hardy-Littlewood-Polya, see 
\cite[Theorem 5.5]{Sch12}. As the name of the term suggests, the first
example of a $\mathcal{C-M}$ couple in the non-discrete infinite dimensional
setting was obtained independently by Calderón \cite{Cal66} and Mityagin 
\cite{Mit65}. Their result that the Banach couple $(L_{1},L_{\infty })$ over
an arbitrary $\sigma $-finite measure space has the $\mathcal{C-M}$ property
marked the start of searching for other such Banach couples. Early examples
are pairs of weighted $L_{\infty }$-spaces, Peetre \cite{Pee71}, and $L_{1}$%
-spaces, Sedaev-Semenov, \cite{SeSe71}. Over time, this property has been
verified for a large number of Banach couples (and also quasi-Banach
couples, couples of normed abelian groups); cf. \cite[Sec. 4.7.2]{BK91}, 
\cite[p. 2 onwards]{CwNiSc03}, \cite{ACN-23}.\textbf{\ $\medskip $ }$%
\medskip $

In the paper \cite{Cwi84}, Cwikel have found a general condition, which
guarantees that two given couples of Banach lattices have the relative $%
\mathcal{C-M}$ property. Let $X$ and $Y$ be Banach lattices of measurable
functions (possibly having different underlying measure spaces). Then $X,Y$
are called \textit{relatively decomposable}\textbf{\ }if for any sequences $%
\{x_{n}\}_{n=1}^{\infty }\subset X$ and $\{y_{n}\}_{n=1}^{\infty }\subset Y$
of elements with pair-wise disjoint supports such that $\sum_{n=1}^{\infty
}x_{n}\in X$ and $\left\Vert y_{n}\right\Vert _{Y}\leq \left\Vert
x_{n}\right\Vert _{X}$, $n\in N$, we have $\sum_{n=1}^{\infty }y_{n}\in Y$
and 
\begin{equation}
\left\Vert \sum_{n=1}^{\infty }y_{n}\right\Vert _{Y}\leq D\left\Vert
\sum_{n=1}^{\infty }x_{n}\right\Vert _{X},  \label{reldec}
\end{equation}%
for some constant $D$ independent of $\{x_{n}\}_{n=1}^{\infty }$ and $%
\{y_{n}\}_{n=1}^{\infty }$. As is proved in \cite{Cwi84} (see also \cite[%
Theorem 4.4.29]{BK91}), couples of Banach lattices $\overrightarrow{X}%
=(X_{0},X_{1})$ and $\overrightarrow{Y}=(Y_{0},Y_{1})$ have the relative
Calderón-Mityagin property whenever both pairs $X_{0}$, $Y_{0}$ and $X_{1}$, 
$Y_{1}$ are relatively decomposable.$\medskip $

Let us mention that the latter notion is closely connected with $L_{p}$%
-spaces. In particular, by an appropriate form of the
Kakutani-Bohnenblust-Tzafriri representation theorem, any decomposable
lattice $X$ of measurable functions on $\Omega $ (i.e., when $X=Y$), which
is also $\sigma $-order continuous and has the Fatou property, coincides,
for some $p\in \lbrack 1,\infty )$, with an $L_{p}$-space of functions
supported on some measurable subset $\Omega ^{\prime }$ of $\Omega $ for
some suitable measure defined on $\Omega ^{\prime }$ (see Proposition 1.4 in 
\cite[p. 58]{CwNi84}). Observe also that there is a simple sufficient
condition for couples of Banach lattices to\textbf{\ }be relatively
decomposable. Namely, if $Y$ satisfies an upper $p$-estimate and $X$ a lower 
$p$-estimate, then the Banach lattices $X$ and $Y$ are relative decomposable.%
$\medskip $

The notion of relative decomposibility plays a central role in the paper 
\cite{CwNiSc03}, where it was proved that all weighted couples modelled on
two given couples $\overrightarrow{X}=(X_{0},X_{1})$ and $\overrightarrow{Y}
=(Y_{0},Y_{1})$ of Banach lattices of measurable functions over $\sigma $
-finite measure spaces possess the relative $\mathcal{C-M}$ property if and
only if both pairs $X_{0},Y_{0}$ and $X_{1},Y_{1}$ are relatively
decomposable. Therefore, in the case when $\overrightarrow{X}= 
\overrightarrow{Y}$ and $X_{0},X_{1}$ are $\sigma $-order continuous and
have the Fatou property, the latter property of a couple $\overrightarrow{X}$
implies that both $X_{0}$ and $X_{1}$ are $L_{p}$-spaces, which is a strong
converse to the well-known result of Sparr \cite{Spa78}, asserting that each
weighted $L_{p}$-couple has the relative $\mathcal{C-M}$ property. $\medskip$

Of course, not all pairs of Banach couples possess the relative $\mathcal{C-M%
}$ property and this motivated Cwikel to consider a weakened version of
that. Specifically, already in the papers \cite{Cwi76} and \cite{Cwi81},
condition \eqref{kmon2} is replaced with the inequality 
\begin{equation}
K(t,y;\overrightarrow{Y})\leq w(t)K(t,x;\overrightarrow{X}),\;t>0,\;\;%
\mbox{with}\;\int\nolimits_{0}^{\infty }w(u)^{s}du/u<\infty ,  \label{wCM}
\end{equation}%
for some (fixed) function $w(u)\ge 0$ and $s\in \lbrack 1,\infty ]$ (if $%
s=\infty $, after the usual modification of the condition imposed on $w$, we
come certainly to the definition of relative $\mathcal{C-M}$ property). In
turn, this led to the introduction of the following more general concept of
\thinspace relatively $s$-decomposable pairs of Banach lattices.$\medskip $

Let $1\leq s\leq \infty $. A pair of Banach lattices $X$, $Y$ is said to be 
\textit{relatively }$s$-\textit{decomposable}\textbf{\ }whenever for all
sequences $\{x_{n}\}_{n=1}^{\infty }\subset X$, $\{y_{n}\}_{n=1}^{\infty
}\subset Y$ of pair-wise disjoint elements such that $\sum_{n=1}^{\infty
}x_{n}\in X$ and $\left\Vert y_{n}\right\Vert _{Y}\leq \left\Vert
x_{n}\right\Vert _{X}$, $n\in N$, we have $\sum_{n=1}^{\infty }y_{n}\in Y$
and 
\begin{equation}
\left\Vert \sum_{n=1}^{\infty }\lambda _{n}y_{n}\right\Vert _{Y}\leq D\left(
\sum_{n=1}^{\infty }\left\vert \lambda _{n}\right\vert ^{s}\right)
^{1/s}\left\Vert \sum_{n=1}^{\infty }x_{n}\right\Vert _{X},  \label{frr}
\end{equation}%
for some constant $D$ and every sequence $\{\lambda _{n}\}_{n=1}^{\infty
}\in l_{s}$ (again with the usual modification in the case $s=\infty $ that
gives \eqref{reldec}, i.e., the relative decomposibility). According to \cite%
{Cwi84} (see also\textbf{\ }\cite[Remark 4.4.33]{BK91}), if $\overrightarrow{%
X}=(X_{0},X_{1})$ and $\overrightarrow{Y}=(Y_{0},Y_{1})$ are two couples of
Banach lattices such that the pairs $X_{0}$, $Y_{0}$ and $X_{1}$, $Y_{1}$
are relatively $s$-decomposable, for every $x\in X_{0}+X_{1}$ and $y\in
Y_{0}+Y_{1}$ satisfying condition \eqref{wCM} we have $y=Tx$ for some linear
operator $T:X_{i}\rightarrow Y_{i}$, $i=0,1$. $\medskip $

One of the main results of the paper \cite{CwNiSc03} is a characterization
of the relative decomposability in the setting of Banach lattices of
measurable functions, implying that the above-mentioned trivial sufficient
condition expressed in terms of upper and lower estimates is also necessary.
In the case of relative $s$-decomposable pairs of Banach lattices $X$, $Y$
there is also a simple sufficient condition, formulated in terms of upper
estimates for $Y$ and lower estimates for $X.$ The main aim \textbf{\ }of
this paper is to prove that, in a more general setting of abstract Banach
lattices, this trivial sufficient condition is also necessary (as in the
case of relative decomposability, i.e., when $s=\infty $). $\medskip $

A pivotal role in the proof of our main result is played by the notions of
optimal upper and lower sequence spaces, introduced in this paper.
Specifically, we associate to every Banach lattice $X$ two sequence spaces $%
X_{U}$ and $X_{L}$, which rather precisely reflect lattice properties of $X$%
, in particular, encoding the optimal upper and lower estimate information,
respectively. Section \ref{Sec-optimal-spaces} is devoted to a detailed
study of the properties of these spaces, which, as we believe, can be useful
tools also when considering other issues related to Banach lattices. As a
result, in Section \ref{Proof_Main_theorem} we present the proof of Theorem %
\ref{Th_main}, which gives a solution of the above problem. On the way, we
obtain also other results related to optimal sequence spaces. We show\textbf{%
\ t}hat if Banach lattices $X$ and $Y$ are relatively $s$-decomposable, then
the space of multiplicators from $X_{L}$ into $Y_{U}$ with respect to
coordinate-wise multiplication includes the space $l_{s}$ (see Proposition %
\ref{Prop_rel_decomp_mult}). Another important ingredient in the proof of
Theorem \ref{Th_main} is the relationship between the construction of
optimal sequence spaces $X_{U}$ and $X_{L}$ and the finite lattice
representability of $l_{r}$-spaces in $X$.$\medskip $

Section \ref{Inter} contains some applications of Theorem \ref{Th_main} to
the interpolation theory. In particular, in Theorem \ref{factor3} we prove
that Banach lattices $X$ and $Y$ being relatively $s$-decomposable admit an
orbital factorization of relative $K$-functional estimates through some
suitable couples of weighted $L_{{p}}$-spaces. In the next section we
present the full proof of rearrangement invariance of the optimal upper and
lower sequence spaces (see Theorem \ref{Th_XL_XU_Prop}). $\medskip $

Finally, in the Appendix we identify the optimal upper sequence space for a
separable Orlicz space $L_{M}$ on $[0,1]$. Namely, in Theorem \ref{Th. Oricz}%
, we prove that the space $(L_{M})_{U}$ can be described as a certain
intersection of some special Musielak-Orlicz sequence spaces. $\medskip $

It is worth noting that, in contrast to \cite{Cwi84} and \cite{BK91}, in
this paper we use a weaker version of \textit{finite} relative $s$%
-decomposibility, which involves estimates \eqref{frr} only for finite
sequences. In a certain sense, this approach seems to be more natural, since
our main result reveals the relationship between this property and
upper/lower estimates of Banach lattices involved, whose definition contains
only finite sequences of elements of these lattices as well. Observe that
the above versions of the definition of relative $s$-decomposibility
coincide if $X$, $Y$ are Banach lattices of measurable functions such that $Y
$ has the Fatou property\footnote{%
A Banach lattice $X$ of measurable functions on a $\sigma $-finite measure
space $\left( T,\Sigma ,\mu \right) $ is said to have the \textit{Fatou
property} if the conditions $x_n\in X$, $n=1,2,\dots$, $\sup_{n=1,2,\dots}%
\|x_n\|_X<\infty$ and $x_n\to x$ a.e. on $T$ imply that $x\in X$ and $%
\|x\|_X\le \liminf_{n\to\infty}\|x_n\|_X$.} (in this case only, the concept
of relative $s$-decomposibility is applied in \cite{Cwi84} (see also \cite%
{BK91}) to the study of interpolation properties of Banach couples).

Some remarks around the history of this paper. Already in \cite[Theorem 1.3, p.~98]{CwNiSc03} the classification of decomposable pairs of Banach lattices on
measure spaces was addressed. In fact, the corresponding results for the
general case were also announced there  (see \cite[p.~100]{CwNiSc03}). This
problem was presented to the second author by Michael Cwikel back in 2003.
Some preliminary results were presented by the second author during Jaak
Peetre's "Summer Seminar" in Lund, 2003. Time flies, and resulted the in the
present paper 20 years later. The authors thank Michael Cwikel for his
insight and contributions to this paper.

\vskip0.3cm

\section{\label{Sec-definitions}Some preliminaries and statements of the
main results.}

We will assume that the reader is familiar with the definition of Banach
lattices, their basic properties and some basic terminology (see, for
instance, \cite{LT79},\cite{MeNi91},\cite{Sch12}). In particular, two
elements $x,y$ from a Banach lattice $X$ are said to be \textit{disjoint }
if they satisfy $\left\vert x\right\vert \wedge \left\vert y\right\vert =0$.
For a given Banach space (lattice), we set $B_{X}:=\left\{ x\in
X:\,\left\Vert x\right\Vert _{X}\leq 1\right\} $ and $S_{X}:=\left\{ x\in
X:\,\left\Vert x\right\Vert _{X}=1\right\} $. $\medskip $

Recall that a lattice $X$ is said to be \textit{$\sigma $-order complete} if
every order bounded sequence in $X$ has a least upper bound (see e.g. \cite[%
Definition~1.a.3]{LT79}). A Banach lattice has a \textit{$\sigma $-order
continuous norm} if every positive, increasing, norm bounded sequence
converges in norm (see e.g. \cite[Definition~5.12]{Sch12}, \cite[%
Definition~1.a.6]{LT79}).$\medskip $

Let $l_{p}\left( I\right) $ denote the Banach space of all sequences,
indexed by the set $I,$ which are absolutely $p$-summable if $1\leq p<\infty 
$ (resp. bounded if $p=\infty $). In the case $I=\mathbb{N}$ we simply write 
$l_{p}$. As usual, by $c_{0}=c_{0}(\mathbb{N})$ will be denoted the space of
all sequences tending to zero as $n\rightarrow \infty $. For definiteness,
all Banach spaces and lattices considered in this paper are assumed to be
real. $\medskip $

Let $F_{1}$ and $F_{2}$ be two positive functions (quasinorms). We write $%
F_{1}\preceq F_{2}$ if we have $F_{1}\leq CF_{2}$ for some positive constant 
$C$ that does not depend on the arguments of $F_{1}$ and $F_{2}$. In the
case when both $F_{1}\preceq F_{2}$ and $F_{2}\preceq F_{1}$ we write $%
F_{1}\asymp F_{2}$. For a finite set $E\subset \mathbb{N}$ we denote by $|E|$
cardinality of $E$. Finally, if $F:\,\Omega \rightarrow \mathbb{R}$ is a
function (resp. $a=\left\{ a_{i}\right\} _{i=1}^{\infty }$ is a sequence of
real numbers), then $\mathrm{supp}\,F:=\{\omega \in \Omega :\,F(\omega )\neq
0\}$ (resp. $\mathrm{supp}\,a:=\{i\in \mathbb{N}:\,a_{i}\neq 0\}$).

\subsection{Relative decomposability of Banach lattices}

The following definition generalizes the first of two definitions in \cite[%
p. 44]{Cwi84} (see also \cite[Definition 4.4.26, p. 597]{BK91}) and has
roots in the interpolation theory of operators (see e.g. \cite[Theorem~2]%
{Cwi84}, \cite[Theorem~4.4.29, p. 598]{BK91} and Section \ref{Intro}).
However, as was already mentioned in Section \ref{Intro}, in contrast to 
\cite{Cwi84} (and also to \cite{CwNiSc03}, where the case $s=\infty $ is
covered), we will use a weaker version of this notion, which involves only
finite sums of elements of given lattices.

\begin{definition}
\label{Def-Rel-Decomposable}Let $1\leq s\leq \infty .$ Banach lattices $X $
and $Y$ are said to be \textit{(finitely) relative $s$-decomposable } if
there exists a constant $D>0$ such that for each $n\in \mathbb{N}$ and for
all sequences of pair-wise disjoint non-zero elements $\left\{ x_{i}\right\}
_{i=1}^{n}\subset X$ and $\left\{ y_{i}\right\} _{i=1}^{n}\subset Y$ it holds%
\begin{equation*}
\left\Vert \sum_{i=1}^{n}y_{i}\right\Vert _{Y}\leq D\left(
\sum_{i=1}^{n}\left( \left\Vert y_{i}\right\Vert _{Y}/\left\Vert
x_{i}\right\Vert _{X}\right) ^{s}\right) ^{1/s}\left\Vert
\sum_{i=1}^{n}x_{i}\right\Vert _{X}
\end{equation*}%
(with the usual modification in the case $s=\infty$). Let $D_{s}=D_{s}\left(
X,Y\right) $ denote the infimum of constants $D$ satisfying the above
condition and we refer to $D_{s}$ as the \textit{relative $s$%
-decomposibility constant }of $X$ and $Y.$
\end{definition}

In what follows, we will suppress the word "finitely", although this is a
change of terminology as compared to the previous use of this term, both in 
\cite{Cwi84} and \cite[Definition 2.2.16]{BK91} (see also a related
discussion in Section \ref{Intro}).

Definition \ref{Def-Rel-Decomposable} can be also stated equivalently as
follows. Let $\left\{ x_{i}\right\} _{i=1}^{n}\subseteq S_{X},\left\{
y_{i}\right\} _{i=1}^{n}\subseteq S_{Y}$ be two sequences of pair-wise
disjoint elements. Then for all sequences $\left\{ a_{i}\right\} _{i=1}^{n}$
and $\left\{ b_{i}\right\} _{i=1}^{n}$ of scalars it holds 
\begin{equation*}
\left\Vert \sum_{i=1}^{n}a_{i}b_{i}y_{i}\right\Vert _{Y}\leq D\left(
\sum_{i=1}^{n}\left\vert a_{i}\right\vert ^{s}\right) ^{1/s}\left\Vert
\sum_{i=1}^{n}b_{i}x_{i}\right\Vert _{X}.
\end{equation*}

Note that in the case $s=\infty $ Definition \ref{Def-Rel-Decomposable}
reduces to the following: if $\left\{ x_{i}\right\} _{i=1}^{n}\subseteq
X,\left\{ y_{i}\right\} _{i=1}^{n}\subseteq Y$ are two sequences of
pair-wise disjoint elements with $\left\Vert y_{i}\right\Vert _{Y}\leq
\left\Vert x_{i}\right\Vert _{X}$, $i=1,\dots ,n$, then%
\begin{equation*}
\left\Vert \sum_{i=1}^{n}y_{i}\right\Vert _{Y}\leq D\left\Vert
\sum_{i=1}^{n}x_{i}\right\Vert _{X}
\end{equation*}%
Following \cite{CwNiSc03}, we will say in this case that Banach lattices $X$
and $Y$ are \textit{relatively decomposable} (see Section \ref{Intro}). $%
\medskip $

In particular, by Hölder inequality, we have

\begin{example}
\label{ex1} If $1\leq p,q,s\leq \infty $ then 
\begin{equation*}
D_{s}\left( l_{q},l_{p}\right) <\infty \iff \frac{1}{p}\leq \frac{1}{q}+%
\frac{1}{s}
\end{equation*}%
and $D_{s}\left( l_{q},l_{p}\right) =1$ whenever this constant is finite.
\end{example}

\subsection{Upper and lower estimates for disjoint elements in Banach
lattices and the Grobler-Dodds indices}

\label{estimates}

Let us start with recalling the notions of lower and upper estimates in
Banach lattices; see \cite[Definition 1.f.4]{LT79}. A Banach lattice $X$ is
said to \textit{satisfy an upper (resp. a lower) $p$-estimate}, where $p\in %
\left[ 1,\infty \right] $, if for some constant $M$ and all finite sequences
of pair-wise disjoint elements $\left\{ x_{i}\right\} _{i=1}^{n}\subseteq X$
it holds

\begin{equation*}
\left\Vert \sum_{i=1}^{n} x_{i} \right\Vert _{X}\leq M\left(
\sum_{i=1}^{n}\left\Vert x_{i}\right\Vert _{X}^{p}\right) ^{1/p}
\end{equation*}%
(resp. 
\begin{equation*}
\left( \sum_{i=1}^{n}\left\Vert x_{i}\right\Vert _{X}^{p}\right) ^{1/p}\leq
M\left\Vert \sum_{i=1}^{n} x_{i} \right\Vert).
\end{equation*}

The infimum of all $M$, satisfying the above inequality, is denoted by $M^{%
\left[ p\right] }\left( X\right) $ and $M_{\left[ p\right] }\left( X\right) $%
, respectively. Note that every Banach lattice admits (trivially) an upper $%
1 $-estimate and a lower $\infty $-estimate.$\medskip $

The fact that a Banach lattice satisfies an upper or a lower $p$-estimate
can be equivalently expressed in terms of relative decomposability. More
explicitly, since $M_{\left[ p\right] }\left( l_{p}\right) =M^{\left[ p%
\right] }\left(l_{p}\right)=1$, $1\leq p\leq \infty $, one can immediately
check the following (see also \cite{CwNiSc03}):

\begin{proposition}
\label{Prop_rel_dec_lp sp} Let $X$ be a Banach lattice and $1\leq p\leq
\infty.$ We have:

$\left( i\right)$ $X$ satisfies an upper $p$-estimate if and only if $l_{p}$
and $X$ are relatively decomposable;

$\left( ii\right)$ $X$ satisfies a lower $p$-estimate if and only if $X$ and 
$l_{p}$ are relatively decomposable.

Moreover, we have 
\begin{equation*}
M^{\left[ p\right] }\left( X\right) =D_{\infty }\left( l_{p},X\right) ,M_{%
\left[ p\right] }\left( X\right) =D_{\infty }\left( X,l_{p}\right) .
\end{equation*}
\end{proposition}

A direct application of the above definitions (see also Example \ref{ex1})
gives the following useful result.

\begin{proposition}
\label{Prop_estimates_decomp} Assume that $X$ and $Y$ are Banach lattices
such that $X$ satisfies a lower $q$-estimate and $Y$ satisfies an upper $p$%
-estimate, where $1\leq q,p\leq \infty $. If ${1}/{p}\leq {1}/{q}+{1}/{s}$,
then $X,Y$ are relatively $s$-decomposable and the following estimate holds: 
\begin{equation*}
D_{s}\left( X,Y\right) \leq M_{\left[ q\right] }\left( X\right) M^{\left[ p%
\right] }\left( Y\right) .
\end{equation*}
\end{proposition}

From the latter proposition and the trivial fact that every Banach lattice
satisfies an upper $1$-estimate and a lower $\infty $-estimate it follows
that any pair of Banach lattices is relatively $1$-decomposable. Hence,
given a pair of Banach lattices $X$ and $Y$, the set of all $s\in \left[
1,\infty \right] $ such that they are $s$-decomposable is always non-empty
and is of the form either~$\left[ 1,s_{\max }\right] $ or $\left[ 1,s_{\max
}\right) $, where 
\begin{equation*}
s_{\max }=s_{\max }(X,Y):=\sup \{s\in \left[ 1,\infty \right] :\,X,Y%
\mbox{are $s$-decomposable}\}.
\end{equation*}%
Taking $X=l_{p}$, $Y=l_{1}$ and $s_{\max }$ so that $1/s_{\max }+1/p=1$, we
obtain an example of the first type because of $X$ and $Y$ are $s$%
-decomposable if and only if $s\in \left[ 1,s_{\max }\right] $ (see Example %
\ref{ex1}). To get an example of the second type, let $Y$ be a Banach
lattice that satisfies, for a given $p_{\max }\in \left[ 1,\infty \right] $,
an upper $p$-estimate for all $p<p_{\max }$ but not for $p=p_{\max }.$ Then,
taking $l_{\infty }$ for $X$, we see that $X$ and $Y$ are $s$-decomposable
if and only if $s<p_{\max }$ (see Proposition \ref{Prop_rel_dec_lp sp}(i)). $%
\medskip $

Recall that the Grobler-Dodds indices $\delta(X)$ and $\sigma (X)$ of a
Banach lattice $X$ are defined by 
\begin{equation*}
\delta(X):=\sup \{p\geq 1:\,X\;\text{satisfies an upper $p$-estimate}\}
\end{equation*}%
and 
\begin{equation*}
\sigma (X):=\inf \{q\geq 1:\,X\;\text{satisfies a lower $q$-estimate}\}.
\end{equation*}
For every infinite-dimensional Banach lattice $X$ we have $1\leq
\delta(X)\leq\sigma (X)\leq \infty $. Moreover, the following duality
relations hold: 
\begin{equation*}
\frac{1}{\delta(X)}+\frac{1}{\sigma (X^{\ast })}=1\;\;\mbox{and}\;\;\frac{1}{%
\sigma (X)}+\frac{1}{\delta(X^{\ast })}=1.
\end{equation*}

\begin{definition}
\label{Def_ap} Let $1\le p\le\infty$. We say that $l_p$ is \textit{finitely
lattice representable} in a Banach lattice $X$ whenever for every $n\in%
\mathbb{N}$ and each $\varepsilon >0$ there exist pair-wise disjoint
elements $x_{i}\in X$, $i=1,2,\dots,n$, such that for any sequence $\left\{
a_{i}\right\} _{i=1}^{n}$ of scalars we have

\begin{equation}  \label{Fin distr}
\left( \sum_{i=1}^{n}\left\vert a_{i}\right\vert ^{p}\right) ^{1/p}\leq
\left\Vert \sum_{i=1}^{n}a_{i}x_{i}\right\Vert _{X}\leq \left( 1+\varepsilon
\right) \left( \sum_{i=1}^{n}\left\vert a_{i}\right\vert ^{p}\right) ^{1/p}.
\end{equation}

Similarly, $l_p$ is said to be \textit{crudely finitely lattice representable%
} in $X$ whenever instead of \eqref{Fin distr} it holds 
\begin{equation*}
C^{-1}\left( \sum_{i=1}^{n}\left\vert a_{i}\right\vert ^{p}\right)
^{1/p}\leq \left\Vert \sum_{i=1}^{n}a_{i}x_{i}\right\Vert _{X}\leq C\left(
\sum_{i=1}^{n}\left\vert a_{i}\right\vert ^{p}\right) ^{1/p},
\end{equation*}%
where $C$ is a constant independent of $n\in\mathbb{N}$ and $\left\{
a_{i}\right\} _{i=1}^{n}$.
\end{definition}

For the following result see, for instance, \cite[Theorem 1.f.12.ii]{LT79}.

\begin{proposition}
\label{Prop_dia} Let $X$ be a Banach lattice. Then $X$ admits a lower $p$%
-estimate for some $p<\infty$ if and only if has $l_\infty$ fails to be
finitely lattice representable in $X$.
\end{proposition}

Moreover, in view of \cite[Theorem 1.a.5, 1.a.7]{LT79} and \cite[p.~288]%
{JMST}, it follows

\begin{proposition}
\label{Prop_not_A-infinity} If a Banach lattice $X$ is not $\sigma$-
complete or has a not $\sigma $-order continuous norm, then $l_\infty$ is
finitely lattice representable in $X$.
\end{proposition}

\subsection{The main result and its consequences}

Now we are ready to state the main result of this paper, which gives a
characterization of relatively $s$-decomposable Banach lattices in terms of
their upper and lower estimates. This is an extension of results of the
paper \cite{CwNiSc03}, where the case $s=\infty $ was covered in a more
restrictive setting of Banach lattices of measurable functions. Note that
the case when $\delta\left(Y\right)\le \sigma \left( X\right)$ is more
interesting, because then the lattices $X$ and $Y$ potentially may be not
relatively decomposable. As was mentioned in Section \ref{Intro}, a
non-trivial part of the next theorem can be also treated as the converse to
Proposition \ref{Prop_estimates_decomp}.

\begin{theorem}
\label{Th_main} Suppose $X$ and $Y$ are infinite dimensional Banach
lattices. If $\delta\left(Y\right)\le \sigma \left( X\right)$ the following
conditions are equivalent:

$\left( i\right) $ $X$ and $Y$ are relatively $s$-decomposable;

$\left( ii\right)$ There exist $p,q$, with $1/p=1/q+1/s$, such that $X$
satisfies a lower $q$-estimate and $Y$ an upper $p$-estimate;

$\left( iii\right)$ There exist $p,q$ with $1/p=1/q+1/s$ such that $X,l_{q} $
and $l_{p},Y$ are relatively decomposable.

In addition, if 
\begin{equation*}
F_{s}\left( X,Y\right) :=\inf \left\{ M_{\left[ q\right] }\left( X\right) M^{%
\left[ p\right] }\left( Y\right) :\frac{1}{s}=\frac{1}{p}-\frac{1}{q},1\leq
p\leq q\leq \infty \right\} ,
\end{equation*}%
it holds 
\begin{equation*}
D_{s}\left( X,Y\right) \leq F_{s}\left( X,Y\right) \leq D_{s}\left(
X,Y\right) ^{2}.
\end{equation*}

Moreover, $\sigma \left( X\right)\le \delta\left(Y\right)$ if and only if $%
s_{\max }(X,Y)=\infty$.
\end{theorem}

The proof of this theorem is presented in Section \ref{Proof_Main_theorem}
below. $\medskip $

$\medskip $ Now, by using standard arguments (see e.g. \cite[Proposition
1.f.6]{LT79}), one can prove that Definition \ref{Def-Rel-Decomposable} is
equivalent to the following assertion for general sequences of elements in
Banach lattices. 

\begin{corollary}
Infinite dimensional Banach lattices $X$ and $Y$ are relatively $s$%
-decomposable if and only if there exists a constant $D_{s}>0$ such that for
each $n\in\mathbb{N}$ and all sequences $\left\{ x_{i}\right\}
_{i=1}^{n}\subseteq S_{X}$, $\left\{ y_{i}\right\} _{i=1}^{n}\subseteq S_{Y}$
and $\left\{ a_{i}\right\} _{i=1}^{n}\subseteq \mathbb{R}$ we have 
\begin{equation*}
\left\Vert \vee _{i=1}^{n}\left\vert a_{i}y_{i}\right\vert \right\Vert
_{Y}\leq D_{s}\left( \sum_{i=1}^{n}\left\vert a_{i}\right\vert ^{s}\right)
^{1/s}\left\Vert \vee _{i=1}^{n}\left\vert x_{i}\right\vert \right\Vert _{X}
\end{equation*}
\end{corollary}

\begin{corollary}
Suppose infinite dimensional Banach lattices $X$ and $Y$ are relatively $s$%
-decomposable for some $s\in [1,\infty]$. Then, there exist equivalent norms
on $X$ and $Y$ such that $X$ and $Y$ are relatively $s$-decomposable Banach
lattices with constant one.
\end{corollary}

\begin{proof}
By \cite[Lemma $2.8.8$]{MeNi91}, every Banach lattice which satisfies a
lower $p$-estimate/an upper $q$-estimate admits an equivalent Banach lattice
norm such that the corresponding lower/upper estimate constant is equal to
one. Consequently, after a suitable renorming $X$ and $Y$, in the notation
of Theorem \ref{Th_main} we have $F_{s}\left(X,Y\right) =1$, which implies
that $D_{s}\left( X,Y\right)=1. $
\end{proof}

\subsection{Rearrangement invariant spaces}

\label{RI}

For a detailed theory of rearrangement invariant spaces we refer to the
monographs \cite{LT79,KPS}. $\medskip $

{Let $I=[0,1]$ or $(0,\infty)$ and let $m$ be the Lebesgue measure on $I$.
Given a measurable function $x(t)$ on $I$ we define its distribution
function by 
\begin{equation*}
n_x(\tau):=m\{t\in I:\,|x(t)|>\tau\},\;\; \tau>0.
\end{equation*}
}

Measurable functions $x(t)$ and $y(t)$ on $I$ are called \textit{%
equimeasurable} if $n_{x}(\tau )=n_{y}(\tau )$ for all $\tau >0$. In
particular, each function $x(t)$ on $I$ is equimeasurable with its
non-increasing left-continuous rearrangement $x^{\ast }(t)$ of $|x(t)|$,
which defines by 
\begin{equation*}
x^{\ast }(t):=\inf \{\tau >0:\,n_{x}(\tau )<t\},\;\;t\in I.
\end{equation*}

\begin{definition}
A Banach function space $X$ on $I$ is said to be \textit{rearrangement
invariant} (in short, r.i.) (or \textit{symmetric}) if the conditions $x\in
X $ and $n_y(\tau)\le n_x(\tau)$ for all $\tau>0$ imply that $y\in X$ and $%
\|y\|_{X} \le \|x\|_{X}$.
\end{definition}

Let $X$ be a r.i. space. If $I=[0,1]$ (resp. $I=(0,\infty )$) we have $%
L_{\infty }\hookrightarrow X\hookrightarrow L_{1}$ (resp. $L_{\infty }\cap
L_{1}\hookrightarrow X\hookrightarrow L_{\infty }+L_{1}$). $\medskip $

$\medskip $The \textit{K{ö}the dual space} $X^{\prime }$ consists of all
measurable functions $y$ such that 
\begin{equation*}
\Vert y\Vert _{X^{\prime }}:=\sup \{\int_{I}|x(t)y(t)|\,dt:\,x\in X,\,\Vert
x\Vert _{X}\leq 1\}<\infty .
\end{equation*}%
Then, $X^{\prime }$ equipped with the norm $\Vert \cdot \Vert _{X^{\prime }}$
is a r.i. space. Moreover, $X\subset X^{\prime \prime }$, and the isometric
equality $X=X^{\prime \prime }$ holds if and only if the norm in $X$ has the 
\textit{Fatou property}, that is, if the conditions $0\leq x_{n}\uparrow x$
a.e. on $I$ and $\sup_{n\in {\mathbb{N}}}\Vert x_{n}\Vert <\infty $ imply $%
x\in X$ and $\Vert x_{n}\Vert \uparrow \Vert x\Vert $.$\medskip $

The \textit{fundamental function} $\phi _{X}$ of a r.i.\ space $X$ is
defined by $\phi _{X}(t):=\Vert \chi _{A}\Vert _{X}$, where $\chi _{A}$ is
the characteristic function of a measurable set $A\subset I$ with $m(A)=t$.
The function $\phi _{X}$ is \textit{quasi-concave} (i.e., $\phi _{X}(0)=0$, $%
\phi _{X}$ increases and $\phi _{X}(t)/t$ decreases). $\medskip $

Most important examples of r.i. spaces are the $L_{p}$-spaces, $1\leq p\leq
\infty $, and their natural generalization, the Orlicz spaces (for their
detailed theory we refer to the monographs \cite{KR61,RR,Mal}).$\medskip $

Let $M$ be an Orlicz function, that is, an increasing convex continuous
function on $[0,\infty )$ such that $M(0)=0$ and $\lim_{t\rightarrow \infty
}M(t)=\infty $. In what follows, we will assume also that $M(1)=1$. Denote
by $L_{M}:=L_{M}(I)$ the \textit{Orlicz space} endowed with the Luxemburg
norm 
\begin{equation*}
\Vert f\Vert _{L_{M}}:=\inf \left\{ \lambda >0\colon \int_{I}M\Big(\frac{%
|f(t)|}{\lambda }\Big)\,dt\leq 1\right\} .
\end{equation*}%
In particular, if $M(u)=u^{p}$, $1\leq p<\infty $, we obtain $L_{p}$.$%
\medskip $

Note that the definition of an Orlicz function space $L_{M}[0,1]$ depends
(up to equivalence of norms) only on the behaviour of the function $M(t)$
for large values of argument $t$. An easy calculation (see also formula
(9.23) in \cite{KR61} on page 79 of the English version) shows that 
\begin{equation}
\varphi _{L_{M}}(t)=\frac{1}{M^{-1}(1/t)},\;\;0<t\leq 1,  \label{fundamfunc}
\end{equation}%
where $M^{-1}$ is the inverse for $M$. $\medskip $

If $M$ is an Orlicz function, then the \textit{Young conjugate} function $%
\tilde{M}$ is defined by 
\begin{equation*}
\tilde{M}(u):=\sup_{t>0}(ut-M(t)),\;\;u\geq 0.
\end{equation*}%
Moreover, $\tilde{M}$ is also an Orlicz function and the Young conjugate for 
$\tilde{M}$ is $M$. $\medskip $

Every Orlicz space $L_{M}(I)$ has the Fatou property; $L_{M}[0,1]$ (resp. $%
L_{M}(0,\infty )$) is separable if and only if the function $M$ satisfies
the \textit{$\Delta _{2}^{\infty }$-condition} (resp. \textit{$\Delta _{2}$%
-condition}), i.e., $\sup_{u\geq 1}{M(2u)}/{M(u)}<\infty $ (resp. $\sup_{u>0}%
{M(2u)}/{M(u)}<\infty $). In this case we have $L_{M}(I)^{\ast
}=L_{M}(I)^{\prime }=L_{\tilde{M}}(I)$.$\medskip $

Let $1<q<\infty $, $1\leq r<\infty $. The Lorentz space $L_{q,r}=L_{q,r}(I)$
consists of all measurable functions $x$ such that 
\begin{equation*}
\Vert x\Vert _{q,r}:=\Big(\frac{r}{q}\int_{I}(t^{1/q}x^{\ast }(t))^{r}\frac{%
dt}{t}\Big)^{1/r}<\infty .
\end{equation*}%
The functional $x\mapsto \Vert x\Vert _{q,r}$ is not subadditive, but it is
equivalent to the norm $x\mapsto \Vert x^{\ast \ast }\Vert _{q,r}$, where $%
x^{\ast \ast }(t):=\frac{1}{t}\int_{0}^{t}x^{\ast }(s)\,ds$, $t>0$.
Moreover, $L_{q,r_{1}}\hookrightarrow L_{q,r_{2}}$, $1\leq r_{1}\leq
r_{2}<\infty $ and $L_{q,q}=L_{q}$ isometrically.$\medskip $

Rearrangement invariant (r.i.) sequence spaces are defined quite similarly.
In particular, the \textit{fundamental function} of a r.i.\ sequence space $%
X $ is defined by $\phi _{X}(n):=\Vert \sum_{k=1}^{n}e_{k}\Vert _{X}$, $%
n=1,2,\dots $. In what follows, $e_{k}$ are the canonical unit vectors,
i.e., $e_{k}=(e_{k}^{i})_{i=1}^{\infty }$, $e_{k}^{i}=0$ for $i\neq k$ and $%
e_{k}^{k}=1$, $k,i=1,2,\dots $.$\medskip $

Recall that an \textit{Orlicz sequence space} $\ell _{\psi }$, where $\psi $
is an Orlicz function, consists of all sequences $(a_{k})_{k=1}^{\infty }$
such that 
\begin{equation*}
\Vert (a_{k})\Vert _{\ell _{\psi }}:=\inf \left\{ u>0:\sum_{k=1}^{\infty
}\psi \Big(\frac{|a_{k}|}{u}\Big)\leq 1\right\} <\infty .
\end{equation*}%
Clearly, if $\psi (t)=t^{p}$, $p\geq 1$, then $\ell _{\psi }=\ell ^{p}$
isometrically. $\medskip $

The fundamental function of an Orlicz sequence space $\ell_{\psi}$ can be
calculated by the formula: $\phi_{\ell_{\psi}}(n)=\frac{1}{\psi^{-1}(1/n)}$, 
$n=1,2,\dots$ Furthermore, an Orlicz sequence space $\ell_{\psi}$ is
separable if and only if $\psi$ satisfies the $\Delta_{2}^{0}$-condition ($%
\psi\in \Delta_{2}^{0}$), that is, 
\begin{equation*}
\sup_{0<u\le 1}{\psi(2u)}/{\psi(u)}<\infty.
\end{equation*}
In this case we have $\ell_{\psi}^*=\ell_{\psi}^{\prime }=\ell_{\tilde{\psi}%
} $, with the Young conjugate function $\tilde{\psi}$ for $\psi$.

Observe that the definition of an Orlicz sequence space $\ell_{\psi}$
depends (up to equivalence of norms) only on the behaviour of $\psi$ near
zero.

\section{\label{Sec-optimal-spaces}Optimal Upper and Lower Sequence Lattices}

In this section we introduce and study some specialized notions which will
play an important role in the proof of our main Theorem \ref{Th_main}. They
are a special kind of sequence spaces which are generated via some
appropriate sequences of norms, defined on $\mathbb{R}^n$, $n\in\mathbb{N}$.

\subsection{Definitions and general properties}

\begin{definition}
Let $X$ be a Banach lattice. For each integer $n,$ let $\mathfrak{B}%
_{n}\left( X\right) $ denote the set of all sequences $\left\{ x_{i}\right\}
_{i=1}^{n}\subseteq S_{X}$ of elements with pair-wise disjoint support.
\end{definition}

\begin{lemma}
\label{Lemma_bn}If $X$ is a Banach lattice of dimension at least $n,$ then
the set $\mathfrak{B}_{n}\left( X\right) $ is non-empty.
\end{lemma}

The proof of this lemma will be provided in Section \ref{R.invariance}. $%
\medskip $

$\medskip $Let now $X$ be an infinite dimensional Banach lattice. Based on $%
X $ we associate two auxiliary constructions, which yield two sequence
spaces $X_{U}$ and $X_{L}$ that satisfy the following norm one continuous
embeddings: 
\begin{equation}
l_{1}\overset{1}{\hookrightarrow }X_{U}\overset{1}{\hookrightarrow }X_{L}%
\overset{1}{\hookrightarrow }l_{\infty }.  \label{embeddings}
\end{equation}%
We will call $X_{U}$ and $X_{L}$ the \textit{optimal upper} and respectively 
\textit{optimal lower sequence spaces} generated by $X.$ Note that the
construction, which leads to the space $X_{L}$, is close to the one
developed in the paper \cite{Jun02} and related to the optimal cotype and
summing properties of a Banach space. $\medskip $

To construct $X_{U}$ we define first, for each fixed integer $n$, the
following norm on $\mathbb{R}^{n}$ by 
\begin{equation*}
\left\Vert \left\{ a_{i}\right\} _{i=1}^{n}\right\Vert _{X_{U}\left(
n\right) }:=\sup \left\{ \left\Vert \sum_{i=1}^{n}a_{i}x_{i}\right\Vert
_{X}:\left\{ x_{i}\right\} _{i=1}^{n}\in \mathfrak{B}_{n}\left( X\right)
\right\} .
\end{equation*}%
Let $X_{U}$ be the space of all real-valued sequences $a=\left\{
a_{i}\right\} _{i=1}^{\infty }$, for which%
\begin{equation*}
\left\Vert a\right\Vert _{X_{U}}:=\sup_{n}\left\Vert \left\{ a_{i}\right\}
_{i=1}^{n}\right\Vert _{X_{U}\left( n\right) }<\infty .
\end{equation*}%
Since 
\begin{equation*}
\left\Vert \left\{ a_{i}\right\} _{i=1}^{n}\right\Vert _{X_{U}\left(
n\right) }\leq \sum_{i=1}^{n}\left\vert a_{i}\right\vert ,\;\;n\in \mathbb{N}%
,
\end{equation*}%
it follows the left-hand side embedding in \eqref{embeddings}. 

The first step in the definition of the space $X_{L}$ is the introduction of
the functionals $\Phi _{n}$, defined for $a=\left\{ a_{i}\right\}
_{i=1}^{n}\in \mathbb{R}^{n}$, $n\in \mathbb{N}$, by 
\begin{equation*}
\Phi _{n}\left( a\right) :=\inf \left\{ \left\Vert
\sum_{i=1}^{n}a_{i}x_{i}\right\Vert _{X}:\,\left\{ x_{i}\right\}
_{i=1}^{n}\in \mathfrak{B}_{n}\left( X\right) \right\} .
\end{equation*}%
Next, we set 
\begin{equation*}
\left\Vert a\right\Vert _{X_{L}\left( n\right) }:=\inf \left\{ \sum_{k\in
F}\Phi _{n}\left( a^{k}\right) :\,F\subseteq \mathbb{N},\left\vert
F\right\vert <\infty ,a^{k}\in \mathbb{R}^{n},a=\sum_{k\in F}a^{k}\right\} .
\end{equation*}%
Note that $\sup_{1\leq i\leq n}\left\vert a_{i}\right\vert \leq \Phi
_{n}\left( a\right) $ and hence $\left\Vert a\right\Vert _{l_{\infty
}^{n}}\leq \left\Vert a\right\Vert _{X_{L}\left( n\right) }$, which implies
that the mapping $a\mapsto \left\Vert a\right\Vert _{X_{L}\left( n\right) }$
defines a norm on $\mathbb{R}^{n}$. Finally, we define $X_{L}$ to be the
space of all real-valued sequences $a=\left\{ a_{i}\right\} _{i=1}^{\infty }$%
, for which 
\begin{equation*}
\left\Vert a\right\Vert _{X_{L}}:=\sup_{n}\left\Vert \left\{ a_{i}\right\}
_{i=1}^{n}\right\Vert _{X_{L}\left( n\right) }<\infty .
\end{equation*}%
Clearly, these definitions imply the second and third embeddings in %
\eqref{embeddings}. 

The proof of the following important properties of the spaces $X_{U}$ and $%
X_{L}$ we provide in Section \ref{R.invariance} below. 

\begin{theorem}
\label{Th_XL_XU_Prop} Let $X$ be an infinite dimensional Banach lattice.
Then $X_{L}$ is a r.i. sequence space and $X_{U}$ is a Banach sequence
lattice. Moreover, if $l_\infty$ is not finitely representable in $X $, $%
X_{U}$ is a r.i. sequence space as well.
\end{theorem}

Denote by $X_{L}^{0}$ (resp. $X_{U}^{0}$) the closed linear span of all
finitely supported sequences in $X_{L}$ (resp. $X_{U}$). Recall that $e_{m}$%
, $m=1,2,\dots$, are the unit basis vectors in spaces of real-valued
sequences.

\begin{corollary}
If $X$ is an infinite dimensional Banach lattice, then $X_{L}^{0}$ is a
Banach space in which the vectors $e_{m}$, $m=1,2,\dots$, form a symmetric
normed basis. If $l_\infty$ is not finitely representable in $X$, the same
conclusion applies also to $X_{U}^{0}$.

In particular, if $\mathrm{supp}\,a\subset \{1,2,\dots,n\}$ for some $n\in%
\mathbb{N}$, we have $\|a\|_{X_L}=\|a\|_{X_L(n)}$ and $\|a\|_{X_U}=\|a%
\|_{X_U(n)}$.
\end{corollary}

\begin{example}
We claim that $\left( c_{0}\right) _{U}=\left( c_{0}\right) _{L}=l_{\infty }$
and hence $\left( c_{0}\right) _{U}^{0}=\left( c_{0}\right)_{L}^0=c_{0}.$

Indeed, first by \eqref{embeddings}, we have $\left( c_{0}\right) _{U}%
\overset{1}{\hookrightarrow }\left( c_{0}\right) _{L}\overset{1}{%
\hookrightarrow }l_{\infty }.$ For the converse, fix an integer $n$ and let $%
\left\{ x_{i}\right\} _{i=1}^{n}\subseteq c_{0}$ be a positive unit norm
sequence with pair-wise disjoint support. Clearly, for all $a_{i}\in \mathbb{%
R}$, $i=1,2,\dots ,n$, 
\begin{equation*}
\left\vert \sum_{i=1}^{n}a_{i}x_{i}\right\vert \leq \sup_{1\leq n\leq
n}\left\vert a_{i}\right\vert ,
\end{equation*}%
and hence $l_{\infty }^{n}\overset{1}{\hookrightarrow }\left( c_{0}\right)
_{U}\left( n\right) $. In consequence, $l_{\infty }\overset{1}{%
\hookrightarrow }\left( c_{0}\right) _{U}.$ Thus, in view of %
\eqref{embeddings}, everything is done.
\end{example}

The above example shows that the spaces $X_{U}$ and $X_{L}$ do not need to
have $\sigma $-order continuous norm even if $X$ has so. At the same time,
the construction of the optimal upper and lower sequence spaces ensures that
they \textit{always} have the following somewhat weaker property.$\medskip $

Recall that a Banach function lattice $X$ on a measure space $(T,\mu )$ is
called \textit{order semi-continuous} if the conditions $x_{n}\in X$, $%
n=1,2,\dots $, $x\in X$ and $x_{n}\chi _{B}\rightarrow x\chi _{B}$ $\mu $%
-a.e. for each set $B\subset T$ such that $\mu (B)<\infty $ imply that $%
\Vert x\Vert _{X}\leq \liminf_{n\rightarrow \infty }\Vert x_{n}\Vert _{X}$.$%
\medskip $

In particular, Banach sequence lattice $E$ (in this case $T=\mathbb{N}$ with
the counting measure $\mu$) is order semi-continuous if $\Vert a\Vert
_{E}\leq \liminf_{n\rightarrow \infty }\Vert a^{n}\Vert _{E}$ whenever a
sequence $\{a^n\}_{n=1}^\infty\subset E$ converges coordinate-wise to $a\in
E $.

\begin{lemma}
\label{Lemma semi-continuous} $X_U$ and $X_L$ are order semi-continuous
Banach sequence lattices for each Banach lattice $X$.
\end{lemma}

\begin{proof}
We prove this result only for $X_U$, because for $X_L$ this can be done in
the same way.

Assume that a sequence $\{a^n\}_{n=1}^\infty\subset X_U$ converges
coordinate-wise to an element $a\in X_U$. Let $a^n=\{a^n_i\}_{i=1}^\infty$, $%
a=\{a_i\}_{i=1}^\infty$. For arbitrary $\varepsilon>0$ select $m\in\mathbb{N}
$ so that 
\begin{equation*}
\left\Vert a\right\Vert _{X_{U}}\le (1+\varepsilon)\left\Vert
\left\{a_{i}\right\} _{i=1}^{m}\right\Vert_{X_{U}\left( m\right) }.
\end{equation*}
Then, for all sufficiently large $n\in\mathbb{N}$ we have 
\begin{equation*}
|a^n_i|\ge (1-\varepsilon)|a_i|,\;\;i=1,2,\dots,m,
\end{equation*}
whence 
\begin{equation*}
\left\Vert \left\{a^n_{i}\right\} _{i=1}^{m}\right\Vert_{X_{U}\left(
m\right) }\ge (1-\varepsilon)\left\Vert \left\{a_{i}\right\}
_{i=1}^{m}\right\Vert_{X_{U}\left( m\right) }.
\end{equation*}
Combining the above inequalities, we get 
\begin{equation*}
\left\Vert a\right\Vert _{X_{U}}\le \frac{1+\varepsilon}{1-\varepsilon}%
\left\Vert \left\{a^n_{i}\right\} _{i=1}^{m}\right\Vert_{X_{U}\left(
m\right) }
\end{equation*}
if $n\in\mathbb{N}$ is sufficiently large. This implies that 
\begin{equation*}
\left\Vert a\right\Vert _{X_{U}}\le \frac{1+\varepsilon}{1-\varepsilon}%
\liminf_{n\rightarrow \infty }\Vert a^{n}\Vert _{X_{U}}.
\end{equation*}
Since $\varepsilon>0$ is arbitrary, the desired result for $X_U$ is proved.
\end{proof}

For any sequence $a=\left\{ a_{i}\right\}_{i=1}^{\infty }$ we have $%
a_{i}=\left\langle a,e_{i}\right\rangle,$ where $\left\langle
\cdot,\cdot\right\rangle$ is the usual inner product. In what follows, the
properties of the optimal sequence spaces from the next proposition, will
play a crucial role.

\begin{proposition}
\label{Prop_main} Let $X$ be a Banach lattice.

$\left( i\right)$. For any $m\in\mathbb{N}$ and every pairwise disjoint
sequences $u_{k}\in X_U$, $k=1,2,\dots,m$, we have

\begin{equation*}
\left\Vert \sum_{k=1}^{m}u_{k}\right\Vert _{X_{U}}\leq \left\Vert
\sum_{k=1}^{m}\left\Vert u_{k}\right\Vert _{X_{U}}e_{k}\right\Vert _{X_{U}};
\end{equation*}

$\left( ii\right)$. For any $m\in\mathbb{N}$ and every pairwise disjoint
sequences $u_{k}\in X_L$, $k=1,2,\dots,m$,

\begin{equation*}
\left\Vert \sum_{k=1}^{m}\left\Vert u_{k}\right\Vert
_{X_{L}}e_{k}\right\Vert _{X_{L}}\leq \left\Vert
\sum_{k=1}^{m}u_{k}\right\Vert _{X_{L}}.
\end{equation*}
\end{proposition}

\begin{proof}
First, we prove $\left( i\right) $ assuming additionally that the elements $%
u_{k}$, $k=1,2,\dots ,m$, have finite support. Then, denoting $%
u:=\sum_{k=1}^{m}u_{k}$ and $n:=|\mathrm{supp}\,u|$, we get 
\begin{equation*}
u_{k}=\sum_{i\in \mathrm{supp}\,u_{k}}\left\langle u_{k},e_{i}\right\rangle
e_{i},\;\;k=1,2,\dots ,m.
\end{equation*}

Take $\left\{ x_{i}\right\} _{i=1}^{n}\in \mathfrak{B}_{n}\left( X\right) $
and put 
\begin{equation*}
z_{k}:=\sum_{i\in \mathrm{supp}\,u_k}\left\langle u_{k},e_{i}\right\rangle
x_{i},\;\;k=1,2,\dots,m.
\end{equation*}
Without any loss of generality, we may assume that $u_{k}\ge 0$ and $%
u_{k}\neq 0$ for all $k$. Hence, $z_{k}\neq 0$ for all $k=1,2,\dots,m$ and 
\begin{equation}  \label{est z}
\left\Vert z_{k}\right\Vert _{X}\leq \left\Vert \sum_{i=1}^{n}\left\langle
u_{k},e_{i}\right\rangle e_{i}\right\Vert _{X_U\left( n\right)}=\left\Vert
u_{k}\right\Vert _{X_{U}}.
\end{equation}%
Moreover, since $u_{k}$, $k=1,2,\dots,m$, are pairwise disjoint sequences
and $x_i$, $i=1,2,\dots,n$, are pairwise disjoint elements from $X$, we
infer that $z_{i}\wedge z_{j}=0$ if $i\neq j$. Thus, $\left\{
z_{k}/\left\Vert z_{k}\right\Vert _{X}:\,1 \leq k\leq m\right\} \in 
\mathfrak{B}_{m}\left( X\right)$. Consequently, from \eqref{est z} it
follows 
\begin{eqnarray*}
\left\Vert \sum_{i=1}^{n}\left\langle u,e_{i}\right\rangle
x_{i}\right\Vert_{X} &=&\left\Vert\sum_{k=1}^{m}\sum_{i=1}^{n}\left\langle
u_{k},e_{i}\right\rangle x_{i}\right\Vert
_{X}=\left\Vert\sum_{k=1}^{m}\sum_{i\in \mathrm{supp}\,u_k}\left\langle
u_{k},e_{i}\right\rangle x_{i}\right\Vert _{X} \\
&=& \left\Vert\sum_{k=1}^{m}z_k\right\Vert _{X}
=\left\Vert\sum_{k=1}^{m}\left\Vert z_{k}\right\Vert _{X}\frac{z_{k}}{%
\left\Vert z_{k}\right\Vert _{X}}\right\Vert _{X} \\
&\leq &\left\Vert \sum_{k=1}^{m}\left\Vert z_{k}\right\Vert
_{X}e_{k}\right\Vert _{X_{U}\left( m\right) }\leq \left\Vert
\sum_{k=1}^{m}\left\Vert u_{k}\right\Vert _{X_{U}}e_{k}\right\Vert _{X_{U}}.
\end{eqnarray*}%
Hence, as a sequence $\left\{ x_{i}\right\} _{i=1}^{n}\in \mathfrak{B}%
_{n}\left( X\right) $ is arbitrary, we conclude that 
\begin{equation*}
\left\Vert u\right\Vert _{X_{U}}=\left\Vert \sum_{i=1}^{n}\left\langle
u,e_{i}\right\rangle e_i \right\Vert _{X_{U}\left( n\right) }\leq \left\Vert
\sum_{k=1}^{m}\left\Vert u_{k}\right\Vert _{X_{U}}e_{k}\right\Vert _{X_{U}},
\end{equation*}
and for finitely supported sequences assertion $\left( i\right) $ is proved.

Let now $u_{k}\in X_{U}$, $k=1,2,\dots ,m$, be arbitrary pairwise disjoint
non-negative elements. Denote by $u_{k}^{(n)}$ the truncations of $u_{k}$ to
the set $\left\{ 1,\dots ,n\right\} $, that is, 
\begin{equation}
u_{k}^{(n)}:=\sum_{1\leq j\leq n,j\in \mathrm{supp}\,u_{k}}a_{j}e_{j},\;%
\;k=1,\dots ,m,  \label{aux seq}
\end{equation}%
Since $u_{k}^{(n)}$, $k=1,2,\dots ,m$, are pairwise disjoint sequences with
finite support, by the first part of the proof, we have 
\begin{equation*}
\Big\|\sum_{k=1}^{m}u_{k}^{(n)}\Big\|_{X_{U}}\leq \Big\|\sum_{k=1}^{m}\Vert
u_{k}^{(n)}\Vert _{X_{U}}e_{k}\Big\|_{X_{U}}\leq \Big\|\sum_{k=1}^{m}\Vert
u_{k}\Vert _{X_{U}}e_{k}\Big\|_{X_{U}}.
\end{equation*}%
Therefore, taking into account that the sequence $\sum_{k=1}^{m}u_{k}^{(n)}$
tends coordinate-wise to $\sum_{k=1}^{m}u_{k}$ as $n\rightarrow \infty $, by
Lemma \ref{Lemma semi-continuous}, we obtain 
\begin{equation*}
\Big\|\sum_{k=1}^{m}u_{k}\Big\|_{X_{U}}\leq \liminf_{n\rightarrow \infty }%
\Big\|\sum_{k=1}^{m}u_{k}^{(n)}\Big\|_{X_{U}}\leq \Big\|\sum_{k=1}^{m}\Vert
u_{k}\Vert _{X_{U}}e_{k}\Big\|_{X_{U}},
\end{equation*}%
which implies $\left( i\right) $ in the general case.

Proceeding with the proof of $\left( ii\right)$, we again consider first the
case when the elements $u_{k}$, $k=1,2,\dots,m$, have finite support. Let $u$%
, $n\in\mathbb{N}$, $\left\{ x_{j}\right\} _{i=1}^{n}$ and $z_k$, $%
k=1,2,\dots,m$, be defined in the same way as in the beginning of the proof
of $\left( i\right)$. Assuming as above that $u_{k}\ge 0$ and $u_{k}\ne 0$, $%
k=1,2,\dots,m$, we get $z_{k}\ne 0$, $k=1,2,\dots,m$, and $n\ge m$.
Consequently, $\left\{ z_{k}/\left\Vert z_{k}\right\Vert_{X}\right\}
_{k=1}^{m}\in \mathfrak{B}_{m}\left( X\right) .$ Moreover, by the definition
of the norm in $X_L$, we have 
\begin{equation*}
\left\Vert u_{k}\right\Vert _{X_{L}}\leq \Phi _{n}\left( u_{k}\right)
\leq\left\Vert z_{k}\right\Vert _{X},\;\;k=1,2,\dots,m.
\end{equation*}

Hence, by Theorem \ref{Th_XL_XU_Prop}, it follows 
\begin{eqnarray*}
\left\Vert \sum_{k=1}^{m}\left\Vert u_{k}\right\Vert
_{X_{L}}e_{k}\right\Vert _{X_{L}\left( m\right) } &\leq &\left\Vert
\sum_{k=1}^{m}\left\Vert z_{k}\right\Vert _{X}e_{k}\right\Vert _{X_{L}\left(
m\right) }\leq \left\Vert \sum_{k=1}^{m}\left\Vert z_{k}\right\Vert _{X}%
\frac{z_{k}}{\left\Vert z_{k}\right\Vert _{X}}\right\Vert _{X} \\
&=&\left\Vert \sum_{k=1}^{m}z_{k}\right\Vert _{X}=\left\Vert
\sum_{k=1}^{m}\sum_{j\in \sigma _{k}}\left\langle u_{k},e_{j}\right\rangle
x_{j}\right\Vert _{X}=\left\Vert \sum_{j=1}^{m}\left\langle
u,e_{j}\right\rangle x_{j}\right\Vert _{X}.
\end{eqnarray*}%
Passing to the infimum over all $\left\{ x_{j}\right\} _{i=1}^{n}\in 
\mathfrak{B}_{n}\left( X\right) $, we obtain 
\begin{equation}
\left\Vert \sum_{k=1}^{m}\left\Vert u_{k}\right\Vert
_{X_{L}}e_{k}\right\Vert _{X_{L}\left( m\right) }\leq \Phi _{n}\left(
u\right) .  \label{est for XL}
\end{equation}

Next, let $u=\sum_{l\in F}v_{l}$ for some finite set $F$ of positive
integers. Clearly, we may assume that the supports of $v_{l}$ are contained
in that of $u$ and hence in the set $\cup _{k=1}^{m}{\mathrm{supp}\,u_{k}}$.
Then, if 
\begin{equation*}
v_{l,k}:=\sum_{i\in \mathrm{supp}\,u_{k}}\left\langle
v_{l},e_{i}\right\rangle e_{i},\;\;k=1,2,\dots ,m,
\end{equation*}%
we have $v_{l}=\sum_{k=1}^{m}v_{l,k}$, $l\in F$, and $u_{k}=\sum_{l\in
F}v_{l,k}$, $k=1,2,\dots ,m$. Furthermore, since $u_{k}^{l}$, $k=1,2,\dots
,m $, are pairwise disjoint and have finite support, applying \eqref{est for
XL} for $v_{l}$, we infer 
\begin{equation*}
\left\Vert \sum_{k=1}^{m}\left\Vert v_{l,k}\right\Vert
_{X_{L}}e_{k}\right\Vert _{X_{L}\left( m\right) }\leq \Phi _{n}\left(
v_{l}\right) ,\;\;l\in F.
\end{equation*}%
Hence, by the triangle inequality, 
\begin{eqnarray*}
\left\Vert \sum_{k=1}^{m}\left\Vert u_{k}\right\Vert
_{X_{L}}e_{k}\right\Vert _{X_{L}\left( m\right) } &\leq &\left\Vert
\sum_{k=1}^{m}\sum_{l\in F}\left\Vert v_{l,k}\right\Vert
_{X_{L}}e_{k}\right\Vert _{X_{L}\left( m\right) } \\
&\leq &\sum_{l\in F}\left\Vert \sum_{k=1}^{m}\left\Vert v_{l,k}\right\Vert
_{X_{L}}e_{k}\right\Vert _{X_{L}\left( m\right) } \\
&\leq &\sum_{l\in F}\Phi _{n}\left( v_{l}\right) .
\end{eqnarray*}%
Since the above representation of $u$ is arbitrary, from Theorem \ref%
{Th_XL_XU_Prop} it follows 
\begin{equation*}
\left\Vert \sum_{k=1}^{m}\left\Vert u_{k}\right\Vert
_{X_{L}}e_{k}\right\Vert _{X_{L}}=\left\Vert \sum_{k=1}^{m}\left\Vert
u_{k}\right\Vert _{X_{L}}e_{k}\right\Vert _{X_{L}\left( m\right) }\leq \Vert
u\Vert _{X_{L}\left( n\right) }=\Vert u\Vert _{X_{L}}.
\end{equation*}%
Thus, for sequences with finite support $\left( ii\right) $ is proved.

To extend the assertion $\left( ii\right) $ to the general case, assume that 
$u_{k}\in X_{L}$, $k=1,\dots ,m$, are pairwise disjoint and non-negative.
Let $n\in \mathbb{N}$ be arbitrary and $u_{k}^{(n)}$ be the truncations
defined by formula \eqref{aux seq}. Since $u_{k}^{(n)}$, $k=1,\dots ,m$, are
finitely supported, as was already proved, it holds 
\begin{equation*}
\Big\|\sum_{k=1}^{m}u_{k}\Big\|_{X_{L}}\geq \Big\|\sum_{k=1}^{m}u_{k}^{(n)}%
\Big\|_{X_{L}}\geq \Big\|\sum_{k=1}^{m}\Vert u_{k}^{(n)}\Vert _{X_{L}}e_{k}%
\Big\|_{X_{L}}.
\end{equation*}%
Observe that from Lemma \ref{Lemma semi-continuous} it follows $%
\lim_{n\rightarrow \infty }\Vert u_{k}^{(n)}\Vert _{X_{L}}=\Vert u_{k}\Vert
_{X_{L}}$ for each $k=1,\dots ,m$. In consequence, we have 
\begin{equation*}
\lim_{n\rightarrow \infty }\Big\|\sum_{k=1}^{m}\Vert u_{k}^{(n)}\Vert
_{X_{L}}e_{k}\Big\|_{X_{L}}=\Big\|\sum_{k=1}^{m}\Vert u_{k}\Vert
_{X_{L}}e_{k}\Big\|_{X_{L}}.
\end{equation*}%
Combining this together with the preceding estimate, we infer that 
\begin{equation*}
\Big\|\sum_{k=1}^{m}u_{k}\Big\|_{X_{L}}\geq \Big\|\sum_{k=1}^{m}\Vert
u_{k}\Vert _{X_{L}}e_{k}\Big\|_{X_{L}},
\end{equation*}%
and so the proof is completed.
\end{proof}

By Theorem \ref{Th_XL_XU_Prop}, both spaces $X_{U}$ and $X_{L}$ are Banach
lattices and hence $X_{U}$- and $X_{L}$-constructions can be applied also to
them. However, this process terminates already on the second step, because
of the following result.

\begin{proposition}
\label{termination} (a) For every Banach sequence lattice $E$ we have $E%
\overset{1}{\hookrightarrow} E_L$. If additionally $E$ is order
semi-continuous, then $E_U\overset{1}{\hookrightarrow} E$.

(b) For every Banach lattice $X$ we have $(X_L)_L=X_L$ and $(X_U)_U=X_U$
isometrically.
\end{proposition}

\begin{proof}
(a) We show first that $E\overset{1}{\hookrightarrow }E_{L}$. Let $n\in 
\mathbb{N}$ be arbitrary. Since $\left\{ e_{i}\right\} _{i=1}^{n}\in 
\mathfrak{B}_{n}\left( E\right) $, for every $a=(a_{i})_{i=1}^{\infty }\in E$%
, we can write 
\begin{equation*}
\Vert a\Vert _{E}\geq \Big\|\sum_{i=1}^{n}a_{i}e_{i}\Big\|_{E}\geq \Phi
_{n}((a_{i})_{i=1}^{n})\geq \Vert (a_{i})_{i=1}^{n}\Vert _{E_{L}(n)}.
\end{equation*}%
Consequently, $\Vert a\Vert _{E}\geq \Vert a\Vert _{E_{L}}$, i.e., $E\overset%
{1}{\hookrightarrow }E_{L}$.

Assume now that $E$ is order semi-continuous. Then, for every $%
a=(a_i)_{i=1}^\infty\in E_U$ and $n\in\mathbb{N}$ we have 
\begin{equation*}
\|(a_i)_{i=1}^n\|_{E_U(n)}\ge \Big\|\sum_{i=1}^n a_ie_i\Big\|_{E},
\end{equation*}
whence 
\begin{equation*}
\|a\|_{E_U}\ge\Big\|\sum_{i=1}^n a_ie_i\Big\|_{E},\;\;n\in\mathbb{N}.
\end{equation*}
Therefore, since $E$ is order semi-continuous, we get $\|a\|_{E_U}\ge
\|a\|_{E}$ for each $a\in E_U$. Thus, the proof of (a) is completed.

(b) If $X$ is an arbitrary Banach lattice, then by Lemma \ref{Lemma
semi-continuous}, $X_L$ is an order semi-continuous Banach sequence lattice.
Hence, from the already proved part (a) it follows that $X_L\overset{1}{%
\hookrightarrow} (X_L)_L$. It remains to check that $(X_L)_L\overset{1}{%
\hookrightarrow} X_L$.

Suppose $n\in\mathbb{N}$ and $b=(b_i)_{i=1}^n\in \mathbb{R}$ is arbitrary.
For every $\varepsilon>0$ there is a sequence $\left\{ u_{i}\right\}
_{i=1}^{n}\in \mathfrak{B}_{n}\left( X_L\right)$ such that 
\begin{equation*}
\Phi_n(b)\ge (1-\varepsilon)\Big\|\sum_{i=1}^n b_iu_i\Big\|_{X_L}.
\end{equation*}
Hence, by Proposition \ref{Prop_main}(ii), we obtain 
\begin{equation*}
\Phi_n(b)\ge (1-\varepsilon)\Big\|\sum_{i=1}^n b_i\|u_i\|_{X_L}e_i\Big\|%
_{X_L}=(1-\varepsilon)\Big\|\sum_{i=1}^n b_ie_i\Big\|_{X_L}.
\end{equation*}

Now, let $a=(a_i)_{i=1}^\infty\in (X_L)_L$ and $n\in\mathbb{N}$. Let $%
\sum_{i=1}^n a_ie_i=\sum_{k\in F} b^k$, where $F\subset \mathbb{N}$ is a
finite set and $b^k=(b_i^k)_{i=1}^n$, $k\in F$. Then, from the preceding
estimate and the triangle inequality it follows that 
\begin{equation*}
\sum_{k\in F}\Phi_n((b_i^k)_{i=1}^n)\ge (1-\varepsilon)\sum_{k\in F}\Big\|%
\sum_{i=1}^n b_i^k e_i\Big\|_{X_L}\ge (1-\varepsilon)\Big\|\sum_{i=1}^n a_i
e_i\Big\|_{X_L},
\end{equation*}
Hence, taking the infimum over all the above representations of $%
(a_i)_{i=1}^n$, implies 
\begin{equation*}
\|a\|_{(X_L)_L}\ge\|(a_i)_{i=1}^n\|_{(X_L)_L(n)}\ge
(1-\varepsilon)\|(a_i)_{i=1}^n\|_{X_L(n)}.
\end{equation*}
Since $\varepsilon>0$ is arbitrary and $X_L$ is order semi-continuous, one
can easily get now that $\|a\|_{(X_L)_L}\ge\|a\|_{X_L}$, and so the proof of
the equality $(X_L)_L=X_L$ is completed.

The proof of the fact that $(X_U)_U=X_U$ is very similar and simpler. Again,
in view of Lemma \ref{Lemma semi-continuous} and the part (a) of this
proposition, it suffices to show that $X_U\overset{1}{\hookrightarrow}
(X_U)_U$. Indeed, if $a=(a_i)_{i=1}^\infty\in X_U$ and $n\in\mathbb{N}$,
then for each sequence $\left\{ u_{i}\right\} _{i=1}^{n}\in \mathfrak{B}%
_{n}\left( X_U\right)$, by Proposition \ref{Prop_main}(i), we have 
\begin{equation*}
\Big\|\sum_{i=1}^n a_iu_i\Big\|_{X_U}\le\Big\|\sum_{i=1}^n
a_i\|u_i\|_{X_U}e_i\Big\|_{X_U}=\Big\|\sum_{i=1}^n a_ie_i\Big\|_{X_U}\le
\|a\|_{X_U}.
\end{equation*}
Therefore, 
\begin{equation*}
\|(a_i)_{i=1}^n\|_{(X_U)_U(n)}\le \|a\|_{X_U}\;\;\mbox{for all}\;n\in\mathbb{%
N},
\end{equation*}
whence $\|a\|_{(X_U)_U}\le \|a\|_{X_U}$. Thus, the proof of the proposition
is completed.
\end{proof}

\subsection{Optimal sequence spaces and upper/lower estimates of Banach
lattices}

As we will see in this section, properties of the spaces $X_{U}$ and $X_{L}$
are largely determined by the optimal upper and lower estimate information
related to the given Banach lattice $X$. The connections, revealed in the
next proposition, will play an important role in the proof of our main
Theorem \ref{Th_main}. Recall that $\delta \left( X\right) $ and $\sigma
\left( X\right) $ are the Grobler-Dodds indices of a Banach lattice $X$ (see
Section \ref{estimates}).

\begin{proposition}
\label{coincidence with lp} Let $X$ be a Banach lattice. Then,

(i) $X_{U}\overset{1}{\hookrightarrow }l_{\delta(X)}$ and $%
l_{p}\hookrightarrow X_{U}$ for every $p<\delta(X)$;

(ii) $l_{\sigma (X)}\overset{1}{\hookrightarrow }X_{L}$ and $%
X_{L}\hookrightarrow l_{q}$ for every $q>\sigma (X)$;

(iii) $X_U=l_p$ if and only if $p=\delta(X)$ and $X$ admits an upper $%
\delta(X)$-estimate;

(iv) $X_{L}=l_{q}$ if and only if $q=\sigma (X)$ and $X$ admits a lower $%
\sigma (X)$-estimate.
\end{proposition}

\begin{proof}
(i) Let $p<\delta (X)$ and $a=(a_{k})_{k=1}^{\infty}\in l_p$. Since $X$
admits an upper $p$-estimate, then for every $n\in \mathbb{N}$ and $\left\{
x_{k}\right\} _{k=1}^{n}\in \mathfrak{B}_{n}\left( X\right)$ we have 
\begin{equation*}
\left\Vert \sum_{k=1}^{n}a_{k}x_{k}\right\Vert _{X}\leq C_{p}\left(
\sum_{k=1}^{n}|a_{k}|^{p}\right) ^{1/p},
\end{equation*}%
where $C_{p}$ depends only on $p$. Consequently, by the definition of $X_{U}$%
, we get 
\begin{equation*}
\Vert a\Vert _{X_{U}}=\sup_{n=1,2,\dots }\Vert (a_{k})_{k=1}^{n}\Vert
_{X_{U}(n)}\leq C_{p}\Vert a\Vert _{l_{p}}.
\end{equation*}

Next, suppose $a=(a_{k})_{k=1}^{\infty}\in X_U$. By Schep's result (see \cite{She92}), $l_{\delta(X)}$ is finitely lattice representable in $X$, which
implies that for every $\varepsilon >0$ and $n\in \mathbb{N}$ there exists a
sequence $\left\{ y_{k}\right\} _{k=1}^{n}\in \mathfrak{B}_{n}\left(
X\right) $ such that for any $a_{k}\in \mathbb{R}$, $k=1,2,\dots ,n$, we
have 
\begin{equation*}
\left\Vert \sum_{k=1}^{n}a_{k}y_{k}\right\Vert _{X}\geq (1-\varepsilon
)\Vert (a_{k})_{k=1}^{n}\Vert _{l_{\delta(X)}}.
\end{equation*}%
Hence, 
\begin{equation*}
\Vert (a_{k})_{k=1}^{n}\Vert _{l_{\delta(X)}}\leq \frac{1}{1-\varepsilon }%
\left\Vert \sum_{k=1}^{n}a_{k}y_{k}\right\Vert _{X}\leq \frac{1}{%
1-\varepsilon }\Vert (a_{k})_{k=1}^{n}\Vert _{X_{U}(n)}\leq \frac{1}{%
1-\varepsilon }\Vert a\Vert _{X_{U}}.
\end{equation*}%
Since $n\in \mathbb{N}$ and $\varepsilon >0$ are arbitrary, we have $\Vert
a\Vert _{l_{\delta(X)}}\leq \Vert a\Vert _{X_{U}}$, which completes the
proof of (i).

(ii) Let $q>\sigma (X)$, $n\in \mathbb{N}$ and $a=(a_{k})_{k=1}^{n}\in 
\mathbb{R}^n$. Since $X$ admits a lower $q$-estimate, then for every $%
\left\{ x_{k}\right\} _{k=1}^{n}\in \mathfrak{B}_{n}\left( X\right) $ we
have 
\begin{equation*}
\Vert a\Vert _{l_{q}^{n}}=\left( \sum_{k=1}^{n}|a_{k}|^{q}\right) ^{1/q}\leq
C_{q}\left\Vert \sum_{k=1}^{n}a_{k}x_{k}\right\Vert _{X},
\end{equation*}%
where $C_{q}$ depends only on $q$. Passing to the infimum over all sequences 
$\left\{ x_{k}\right\} _{k=1}^{n}\in \mathfrak{B}_{n}\left( X\right) $, we
come to the inequality 
\begin{equation*}
\Vert a\Vert _{l_{q}^{n}}\leq C_{q}\Phi _{n}(a).
\end{equation*}%
Next, if $a=\sum_{l\in F}b^{l}$, where $F\subset \mathbb{N}$ is finite and $%
b^{l}\in \mathbb{R}^{n}$, we have 
\begin{equation*}
\Vert a\Vert _{l_{q}^{n}}\leq \sum_{l\in F}\Vert b^{l}\Vert _{l_{q}^{n}}\leq
C_{q}\sum_{l\in F}\Phi _{n}(b^{l}).
\end{equation*}%
Passing to the infimum over all above representations of $a$, we obtain 
\begin{equation*}
\Vert a\Vert _{l_{q}^{n}}\leq C_{q}\Vert a\Vert _{X_{L}(n)}.
\end{equation*}%
Since this holds for every $n\in \mathbb{N}$ and $a=(a_{k})_{k=1}^{n}\in 
\mathbb{R}^{n}$, by using Lemma \ref{Lemma semi-continuous}, we conclude
that $X_{L}\hookrightarrow l_{q}$.

Further, again appealing to \cite{She92}, we have that $l_{\sigma (X)}$ is
finitely lattice representable in $X$. Therefore, for every $\varepsilon >0$
and $n\in \mathbb{N}$ there exists a sequence $\left\{ y_{k}\right\}
_{k=1}^{n}\in \mathfrak{B}_{n}\left( X\right) $ such that for any $%
a=(a_{k})_{k=1}^n\in \mathbb{R}^n$ it holds 
\begin{equation*}
\left\Vert \sum_{k=1}^{n}a_{k}y_{k}\right\Vert _{X}\leq (1+\varepsilon
)\Vert a\Vert _{l_{\sigma (X)}}.
\end{equation*}%
In consequence, by the definition of the $X_{L}(n)$-norm, 
\begin{equation*}
\Vert a\Vert _{X_{L}(n)}\leq \Phi_{n}(a)\leq (1+\varepsilon )\Vert a\Vert
_{l_{\sigma (X)}},
\end{equation*}%
and hence again for each $a\in l_{\sigma(X)}$ and any $\varepsilon >0$ it
follows that 
\begin{equation*}
\Vert a\Vert _{X_{L}}\leq (1+\varepsilon )\Vert a\Vert _{l_{\sigma (X)}}.
\end{equation*}%
Application of Lemma \ref{Lemma semi-continuous} again completes the proof.

(iii) If $X$ admits an upper $\delta(X)$-estimate, the same argument as in
the proof of (i) implies that $l_{\delta(X)}\hookrightarrow X_{U}$.
Combining this with the first embedding in (i), we get that $%
X_{U}=l_{\delta(X)}$.

Conversely, let $X_U=l_{p}$ for some $p\ge 1$. Then, from (i) it follows
immediately that $p$ should be equal to $\delta(X)$. It remains to show that 
$X$ admits an upper $\delta(X)$-estimate.

Suppose that $x_{k}\in X$, $k=1,2,\dots ,n$, are arbitrary pair-wise
disjoint elements. Then, by the definition of the $X_{U}(n)$-norm and the
fact that $X_U=l_{\delta(X)}$, we have 
\begin{equation*}
\left\Vert \sum_{k=1}^{n}x_{k}\right\Vert _{X}\leq \left\Vert
\sum_{k=1}^{n}\|x_{k}\|_Xe_k\right\Vert_{X_{U}(n)}\leq C\left
(\sum_{k=1}^{n}\|x_{k}\|_X^{\delta(X)}\right)^{1/\delta(X)},
\end{equation*}%
where $C$ does not depend on $n$ and $x_{k}$. This means that $X$ admits an
upper $\delta(X)$-estimate.

(iv) If $X$ admits a lower $\sigma(X)$-estimate, then from (ii) it follows
immediately then $X_L=l_{\sigma(X)}$.

Conversely, if $X_{L}=l_{q}$, then, by (ii), $q=\sigma (X)$. Consequently,
for every pair-wise disjoint $x_{k}\in X$, $k=1,2,\dots ,n$, by the
definition of the $X_{L}(n)$-norm, it follows 
\begin{equation*}
\left (\sum_{k=1}^{n}\|x_{k}\|_X^{\sigma (X)}\right)^{1/\sigma (X)} \leq
C\left\Vert \sum_{k=1}^{n}\|x_{k}\|_Xe_k\right\Vert_{X_{L}(n)}\leq
C\left\Vert\sum_{k=1}^{n}x_{k}\right\Vert _{X}.
\end{equation*}%
Therefore, $X$ admits a lower $\sigma (X)$-estimate, and the proof is
complete.
\end{proof}

In some cases, an application of the last proposition allows to find
immediately the optimal sequence spaces.

\begin{example}
\label{Lorentz} Let $1<p<\infty $, $1\leq q<\infty $ and let $%
L_{p,q}=L_{p,q}(I)$ be the Lorentz space, where $I=[0,1]$ or $(0,\infty )$
(see Section \ref{RI}). It is well known that $\delta(L_{p,q})=\min (p,q)$, $%
\sigma (L_{p,q})=\max (p,q)$, and moreover, that $L_{p,q}$ admits an upper $%
\delta(L_{p,q})$-estimate and a lower $\sigma (L_{p,q})$-estimate (see e.g. 
\cite[Theorem~3]{D01}). Consequently, by Proposition \ref{coincidence with
lp}, $(L_{p,q})_{U}=l_{\min (p,q)}$ and $(L_{p,q})_{L}=l_{\max (p,q)}$.
\end{example}

\begin{corollary}
\label{cor2} Let $X$ be a Banach lattice. Then, $\delta(X_{U})=\delta(X)$
and $\sigma (X_{L})=\sigma (X)$.
\end{corollary}

\begin{proof}
First, we claim that $X_{U}$ admits an upper $p$-estimate for every $%
p<\delta(X)$. Indeed, if $p<\delta(X)$ and $u_i\in X_U$, $i=1,\dots,n$, are
disjoint, by Propositions \ref{Prop_main} and \ref{coincidence with lp}(i),
we have 
\begin{equation*}
\Big\|\sum_{i=1}^n u_i\Big\|_{X_U}\le \Big\|\sum_{i=1}^n \|u_i\|_{X_U}e_i%
\Big\|_{{X_U}}\le C_p\Big(\sum_{i=1}^n \|u_i\|_{{X_U}}^{p}\Big)^{1/p}.
\end{equation*}
Since $p<\delta(X)$ is arbitrary, this inequality implies that $\delta(X)\le
\delta(X_{U})$. It remains to prove the opposite inequality.

Suppose that $X_{U}$ admits an upper $p$-estimate with $p>\delta(X)$. Then,
there is a constant $C>0$ such that for every $n\in \mathbb{N}$ and any $%
a_{k}\in \mathbb{R}$, $k=1,2,\dots ,n$, we have 
\begin{equation*}
\Big\|\sum_{k=1}^{n}a_{k}e_{k}\Big\|_{X_{U}}\leq C\Vert
(a_{k})_{k=1}^{n}\Vert _{l_{p}}.
\end{equation*}%
On the other hand, by Proposition \ref{coincidence with lp}(i), 
\begin{equation*}
\Big\|\sum_{k=1}^{n}a_{k}e_{k}\Big\|_{X_{U}}\geq \Vert
(a_{k})_{k=1}^{n}\Vert _{l_{\delta(X)}}
\end{equation*}%
for all $n\in \mathbb{N}$ and $a_{k}\in \mathbb{R}$, $k=1,2,\dots ,n$. Since 
$p>\delta(X)$, combining these inequalities, we come to a contradiction.
Thus, $\delta(X_{U})=\delta(X)$, as required.

Similarly, by using Propositions \ref{Prop_main} and \ref{coincidence with
lp}(ii), one can easily check that the space $X_L$ has a lower $p$-estimate
for every $p>\sigma(X)$. Therefore, $\sigma (X_{L})\le\sigma (X)$, and the
equality $\sigma (X_{L})=\sigma (X)$ will be proved, once we check that $%
X_{L}$ does not admit a lower $p$-estimate with any $p<\sigma(X)$. To the
contrary, assume that for some $p<\sigma (X)$ there is a constant $C>0$ such
that for every $n\in \mathbb{N}$ and any $a_{k}\in \mathbb{R}$, $k=1,2,\dots
,n$, 
\begin{equation*}
\Vert (a_{k})_{k=1}^{n}\Vert _{l_{p}}\leq C\left\Vert
\sum_{k=1}^{n}a_{k}e_{k}\right\Vert _{X_{L}}.
\end{equation*}%
On the other hand, from Proposition \ref{coincidence with lp}(ii) it follows
that 
\begin{equation*}
\left\Vert \sum_{k=1}^{n}a_{k}e_{k}\right\Vert _{X_{L}}\leq \Vert
(a_{k})_{k=1}^{n}\Vert _{l_{\sigma (X)}}
\end{equation*}%
for all $n\in \mathbb{N}$ and $a_{k}\in \mathbb{R}$, $k=1,2,\dots ,n$. Since
the latter estimates imply a contradiction, everything is done.
\end{proof}

\section{\label{Proof_Main_theorem}Proof of Theorem \protect\ref{Th_main}}

We start with some auxiliary assertions. $\medskip $

Our first result shows that relative $s$-decomposability of Banach lattices $%
X$ and $Y$ implies that each sequence from the space $l_{s}$ can be treated
as a multiplicator, bounded from $X_{L}$ into $Y_{U}.$

\begin{proposition}
\label{Prop_rel_decomp_mult}Let $X$ and $Y$ be relatively $s$-decomposable
Banach lattices. Then, we have 
\begin{equation*}
X_{L}\cdot l_{s}\hookrightarrow Y_{U},
\end{equation*}%
i.e., the conditions $a=\left\{ a_{i}\right\} _{i=1}^{\infty }\in X_{L}$, $%
b=\left\{b_{i}\right\} _{i=1}^{\infty }\in l_{s}$ imply $ab:=\left\{
a_{i}b_{i}\right\}_{i=1}^{\infty }\in Y_{U}$ and 
\begin{equation*}
\left\Vert ab\right\Vert _{Y_{U}}\leq D_{s}\left( X,Y\right) \left\Vert
b\right\Vert _{l_{s}}\left\Vert a\right\Vert _{X_{L}}.
\end{equation*}
\end{proposition}

\begin{proof}
Let $n$ be any positive integer and $D>D_{s}\left( X,Y\right) .$ For
arbitrary $a=\left\{ a_{i}\right\} _{i=1}^{\infty }\in X_{L}$, $%
b=\left\{b_{i}\right\} _{i=1}^{\infty }\in l_{s}$ we put $%
a^{(n)}:=\sum_{i=1}^{n}a_{i}e_{i},b^{(n)}:=\sum_{i=1}^{n}b_{i}e_{i}.$ Since $%
X$ and $Y $ are relatively $s$-decomposable, for any $\left\{ x_{i}\right\}
\in \mathfrak{B}_{n}\left( X\right)$ and $\left\{ y_{i}\right\}_{i=1}^{n}\in 
\mathfrak{B}_{n}\left( Y\right)$ it holds 
\begin{equation*}
\left\Vert \sum_{i=1}^{n}a_{i}b_{i}y_{i}\right\Vert _{Y}\leq D\left(
\sum_{i=1}^{n}\left\vert b_{i}\right\vert ^{s}\right) ^{1/s}\left\Vert
\sum_{i=1}^{n}a_{i}x_{i}\right\Vert _{X}.
\end{equation*}%
By taking the supremum over all sequences $\left\{ y_{i}\right\} _{i=1}^{n}$
and the infimum over all sequences $\left\{ x_{i}\right\} _{i=1}^{n}$ we
infer%
\begin{equation*}
\left\Vert a^{(n)}b^{(n)}\right\Vert _{Y_{U}\left( n\right) }\leq
D\left\Vert b\right\Vert _{l_{s}}\Phi _{n}\left( a^{(n)}\right)
\end{equation*}%
Next, write $a^{(n)}=\sum_{k\in F}a^{k}$ for some finite set $F\subset%
\mathbb{N}$ and $a^{k}=\{a^k_i\}_{i=1}^n$, $k\in F$. Then $%
a^{(n)}b^{(n)}=\sum_{k\in F}b^{(n)}a^{k}$ and the preceding estimate implies 
\begin{equation*}
\left\Vert a^{(n)}b^{(n)}\right\Vert _{Y_{U}\left( n\right) }\leq \sum_{k\in
F}\left\Vert b^{(n)}a^{k}\right\Vert _{Y_{U}\left( n\right) }\leq
D\left\Vert b\right\Vert _{l_{s}}\left( \sum_{k\in F}\Phi _{n}\left(
a^{k}\right) \right).
\end{equation*}%
After taking the infimum over all such decompositions of $a^{(n)}$ we obtain 
\begin{equation*}
\left\Vert a^{(n)}b^{(n)}\right\Vert _{Y_{U}\left( n\right) }\leq
D\left\Vert b\right\Vert _{l_{s}}\left\Vert a^{(n)}\right\Vert _{X_{L}\left(
n\right) },\;\; n\in\mathbb{N},
\end{equation*}%
which implies the claimed result (see also Lemma \ref{Lemma semi-continuous}%
).
\end{proof}

\begin{proposition}
\label{Lemma-1} Suppose that $X$ and $Y$ are relatively $s$-decomposable
Banach lattices for some $1\leq s\leq \infty $, $l_{p}$ is finitely lattice
representable in $Y_{U}$, where $p\leq s$.

Then, $X$ satisfies a lower $q$-estimate for every $q$ such that ${1}/{p}\ge{%
1}/{q}+1/{s}$, and $M_{\left[ q\right]}\left( X\right) \leq D_{s}\left(
X,Y\right)$.
\end{proposition}

\begin{proof}
Let $n$ be any positive integer and $\varepsilon >0$ be arbitrary. By the
assumption, we can find pair-wise disjoint elements $u_{i}\in Y_{U}$, $%
i=1,\dots,n$, satisfying 
\begin{equation}  \label{est1}
\left\Vert b\right\Vert _{l_{p}^{n}}\leq \left\Vert
\sum_{i=1}^{n}b_{i}u_{i}\right\Vert _{Y_{U}}\leq \left( 1+\varepsilon
\right) \left\Vert b\right\Vert _{l_{p}^{n}}
\end{equation}%
for all sequences $b=\left\{ b_{i}\right\} _{i=1}^{n}$ of scalars.

Let $\left\{ x_{i}\right\} _{i=1}^{n}\in \mathfrak{B}_{n}\left( X\right) $
and $D>D_{s}\left( X,Y\right) .$ Then for any sequence $\left\{
a_{i}\right\} _{i=1}^{n}$ of scalars, by using \eqref{est1}, Proposition \ref%
{Prop_main}(i) and relative $s$-decomposability of $X$ and $Y$, we have 
\begin{eqnarray*}
\left( \sum_{i=1}^{n}\left\vert a_{i}b_{i}\right\vert ^{p}\right) ^{1/p}
&\leq &\left\Vert \sum_{i=1}^{n}a_{i}b_{i}u_{i}\right\Vert _{Y_{U}}\leq
\left\Vert \sum_{i=1}^{n}a_{i}b_{i}\left\Vert u_{i}\right\Vert
_{Y_{U}}e_{i}\right\Vert _{Y_{U}} \\
&=&\sup \left\{ \left\Vert \sum_{i=1}^{n}a_{i}b_{i}\left\Vert
u_{i}\right\Vert _{Y_{U}}y_{i}\right\Vert _{Y}:\left\{ y_{i}\right\}
_{i=1}^{n}\in \mathfrak{B}_{n}\left( Y\right) \right\} \\
&\leq &D\left( \sum_{i=1}^{n}\left\vert b_{i}\right\vert ^{s}\right)
^{1/s}\sup_{1\leq i\leq n}\left\Vert u_{i}\right\Vert _{Y_{U}}\left\Vert
\sum_{i=1}^{n}a_{i}x_{i}\right\Vert _{X}
\end{eqnarray*}%
Consequently, since from \eqref{est1} it follows $\left\Vert
u_{i}\right\Vert _{Y_{U}}\leq 1+\varepsilon $, we get%
\begin{equation*}
\left( \sum_{i=1}^{n}\left\vert a_{i}b_{i}\right\vert ^{p}\right) ^{1/p}\leq
\left( 1+\varepsilon \right) D\left( \sum_{i=1}^{n}\left\vert
b_{i}\right\vert ^{s}\right) ^{1/s}\left\Vert
\sum_{i=1}^{n}a_{i}x_{i}\right\Vert _{X}.
\end{equation*}%
Hence, by the reverse Hölder inequality, 
\begin{equation*}
\left( \sum_{i=1}^{n}\left\vert a_{i}\right\vert ^{q}\right) ^{1/q}\leq
\left( 1+\varepsilon \right) D\left\Vert \sum_{i=1}^{n}a_{i}x_{i}\right\Vert
_{X}
\end{equation*}%
whenever ${1}/{q}\leq {1}/{p}-1/{s}$. Thus, $X$ satisfies a lower $q$%
-estimate and $M_{\left[ q\right] }(X)\leq D_{s}\left( X,Y\right) $.
\end{proof}

\begin{proposition}
\label{Lemma-2} Suppose that $X$ and $Y$ are relatively $s$-decomposable
Banach lattices for some $1\le s\le\infty$, $l_{q}$ is finitely lattice
representable in $X_{L}$ and $1/q+1/s\le 1$.

Then, $Y$ satisfies an upper $p$-estimate for every $p$ such that ${1}/{p}\ge%
{1}/{q}+1/{s}$, and $M^{\left[ p\right]}\left( Y\right) \leq D_{s}\left(
X,Y\right)$.
\end{proposition}

\begin{proof}
Let $n$ be a positive integer and $\varepsilon >0$ be arbitrary. By the
assumption, we can select pair-wise disjoint elements $u_{i}\in X_{L}$, $%
i=1,\dots,n$, such that 
\begin{equation}  \label{est2}
\left\Vert b\right\Vert _{l_{q}^{n}}\leq \left\Vert
\sum_{i=1}^{n}b_{i}u_{i}\right\Vert _{X_{L}}\leq \left( 1+\varepsilon
\right) \left\Vert b\right\Vert _{l_{q}^{n}}
\end{equation}%
for all sequences $b=\left\{ b_{i}\right\} _{i=1}^{n}$ of scalars.

Suppose $\left\{ y_{i}\right\} _{i=1}^{n}\in \mathfrak{B}_{n}\left( Y\right) 
$ and $D>D_{s}\left( X,Y\right) .$ For each $b=\left\{ b_{i}\right\}
_{i=1}^{n}$ we write $b=\sum_{k\in F}b^{k}$, where $F\subset \mathbb{N}$ is
a finite set and $b^{k}=\{b_{i}^{k}\}_{i=1}^{n}$ are arbitrary. Then, for
every sequences $\left\{ x_{i}^{k}\right\} _{i=1}^{n}\in \mathfrak{B}%
_{n}\left( X\right) $, $k\in F$, and any sequence $a=\left\{ a_{i}\right\}
_{i=1}^{n}$ of scalars, by the triangle inequality and relative $s$%
-decomposability of $X$ and $Y$, we have 
\begin{equation*}
\left\Vert \sum_{i=1}^{n}a_{i}b_{i}y_{i}\right\Vert _{Y}\leq \sum_{k\in
F}\left\Vert \sum_{i=1}^{n}a_{i}b_{i}^{k}y_{i}\right\Vert _{Y}\leq D\left(
\sum_{i=1}^{n}\left\vert a_{i}\right\vert ^{s}\right) ^{1/s}\sum_{k\in
F}\left\Vert \sum_{i=1}^{n}b_{i}^{k}x_{i}^{k}\right\Vert _{X}.
\end{equation*}%
Therefore, taking the infimum over all sequences $\left\{ x_{i}^{k}\right\}
_{i=1}^{n}\in \mathfrak{B}_{n}\left( X\right) $ for each $k\in F$ implies
that 
\begin{equation*}
\left\Vert \sum_{i=1}^{n}a_{i}b_{i}y_{i}\right\Vert _{Y}\leq D\left\Vert
a\right\Vert _{l_{s}^{n}}\sum_{k\in F}\Phi _{n}\left( b^{k}\right) ,
\end{equation*}%
and hence 
\begin{equation*}
\left\Vert \sum_{i=1}^{n}a_{i}b_{i}y_{i}\right\Vert _{Y}\leq D\left\Vert
a\right\Vert _{l_{s}^{n}}\left\Vert b\right\Vert _{X_{L}}.
\end{equation*}%
Thus, applying Proposition \ref{Prop_main}(ii) and inequalities \eqref{est2}%
, we obtain 
\begin{eqnarray*}
\left\Vert \sum_{i=1}^{n}a_{i}b_{i}y_{i}\right\Vert _{Y} &\leq &D\left\Vert
a\right\Vert _{l_{s}^{n}}\left\Vert \sum_{i=1}^{n}b_{i}e_{i}\right\Vert
_{X_{L}}\leq D\left\Vert a\right\Vert _{l_{s}^{n}}\left\Vert
\sum_{i=1}^{n}b_{i}\frac{u_{i}}{\left\Vert u_{i}\right\Vert _{X_{L}}}%
\right\Vert _{X_{L}} \\
&\leq &D\left\Vert a\right\Vert _{l_{s}^{n}}\left\Vert \sum_{i=1}^{n}b_{i}{%
u_{i}}\right\Vert _{X_{L}}\leq D\left( 1+\varepsilon \right) \left\Vert
a\right\Vert _{l_{s}^{n}}\left\Vert b\right\Vert _{l_{q}^{n}}.
\end{eqnarray*}%
By the reverse Hölder inequality, this implies that 
\begin{equation*}
\left\Vert \sum_{i=1}^{n}c_{i}y_{i}\right\Vert _{Y}\leq D\left(
1+\varepsilon \right) \left( \sum_{i=1}^{n}\left\vert c_{i}\right\vert
^{p}\right) ^{1/p},
\end{equation*}%
whenever $\frac{1}{p}\geq \frac{1}{q}+\frac{1}{s}$. As a result, $Y$
satisfies an upper $p$-estimate and $M^{\left[ p\right] }(Y)\leq D_{s}(X,Y)$.
\end{proof}

\begin{proposition}
\label{Lemma-3} Suppose Banach lattices $X$ and $Y$ satisfy the following
conditions:

$\left( a\right)$ $X$, $Y$ are relatively s-decomposable for some $1\leq
s\le\infty $;

$\left( b\right)$ $l_{p}$ is finitely lattice representable in $Y_{U}$ ;

$\left( c\right)$ $l_{q}$ is finitely lattice representable in $X_{L}$.

Then, it holds 
\begin{equation*}
\frac{1}{p}\leq \frac{1}{q}+\frac{1}{s}
\end{equation*}
\end{proposition}

\begin{proof}
\bigskip Let $n$ be a positive integer and $\varepsilon >0$ be arbitrary. By
assumption $\left( b\right) $, there exist pair-wise disjoint elements $%
y_{i}\in Y_{U}$, $i=1,\dots ,n$, such that for all scalar sequences $%
b=\left\{ b_{i}\right\} _{i=1}^{n}$ 
\begin{equation*}
\left\Vert b\right\Vert _{l_{p}^{n}}\leq \left\Vert
\sum_{i=1}^{n}b_{i}y_{i}\right\Vert _{Y_{U}(n)}\leq \left( 1+\varepsilon
\right) \left\Vert b\right\Vert _{l_{p}^{n}}.
\end{equation*}%
In the same manner, using $\left( c\right) $, we can select pair-wise
disjoint $x_{i}\in X_{L}$, $i=1,\dots ,n$, such that for all scalar
sequences $a=\left\{ a_{i}\right\} _{i=1}^{n}$ 
\begin{equation*}
\left\Vert a\right\Vert _{l_{q}^{n}}\leq \left\Vert
\sum_{i=1}^{n}a_{i}x_{i}\right\Vert _{X_{L}(n)}\leq \left( 1+\varepsilon
\right) \left\Vert a\right\Vert _{l_{q}^{n}}.
\end{equation*}%
Applying these inequalities and Propositions \ref{Prop_main}(i), \ref%
{Prop_rel_decomp_mult} and \ref{Prop_main}(ii), we obtain 
\begin{eqnarray*}
\left\Vert ab\right\Vert _{l_{p}^{n}} &\leq &\left\Vert
\sum_{i=1}^{n}a_{i}b_{i}y_{i}\right\Vert _{Y_{U}(n)}\leq \left\Vert
\sum_{i=1}^{n}a_{i}b_{i}\left\Vert y_{i}\right\Vert _{Y_{U}}e_{i}\right\Vert
_{Y_{U}(n)} \\
&\leq &\left( 1+\varepsilon \right) D_{s}\left( X,Y\right) \left\Vert
b\right\Vert _{l_{s}^{n}}\left\Vert \sum_{i=1}^{n}a_{i}e_{i}\right\Vert
_{X_{L}(n)} \\
&\leq &\left( 1+\varepsilon \right) D_{s}\left( X,Y\right) \left\Vert
b\right\Vert _{l_{s}^{n}}\left\Vert \sum_{i=1}^{n}a_{i}x_{i}\right\Vert
_{X_{L}(n)} \\
&\leq &\left( 1+\varepsilon \right) ^{2}D_{s}\left( X,Y\right) \left\Vert
b\right\Vert _{l_{s}^{n}}\left\Vert a\right\Vert _{l_{q}^{n}}.
\end{eqnarray*}%
Since $n\in \mathbb{N}$, $b=\left\{ b_{i}\right\} _{i=1}^{n}$ and $a=\left\{
a_{i}\right\} _{i=1}^{n}$ are arbitrary, the claim follows (see also Example %
\ref{ex1}).
\end{proof}

Recall that 
\begin{equation*}
s_{\max }=s_{\max }(X,Y):=\sup \{s\in \left[ 1,\infty \right] :\,X,Y%
\mbox{are $s$-decomposable}\}.
\end{equation*}

\begin{proposition}
\label{Lemma-3a} Let $X$ and $Y$ be Banach lattices such that $\delta(Y)\leq
\sigma (X)$. Then, we have 
\begin{equation}
\frac{1}{\delta\left(Y\right) }=\frac{1}{\sigma \left( X\right) }+\frac{1}{%
s_{\max }}.  \label{equ1}
\end{equation}
\end{proposition}

\begin{proof}
Assume first that $s_{\max }>1$. Then, if $1\leq s<s_{\max }$, $X$ and $Y$
are relatively $s$-decomposable. Moreover, by Schep's result \cite{She92}, $%
l_{s\left( Y_{U}\right) }$ and $l_{\sigma \left( X_{L}\right) }$ are
finitely representable in $Y_{U}$ and $X_{L}$, respectively. Hence, all the
conditions of Proposition \ref{Lemma-3} are fulfilled and we conclude 
\begin{equation*}
\frac{1}{\delta\left( Y_{U}\right) }\leq \frac{1}{\sigma \left( X_{L}\right) 
}+\frac{1}{s}.
\end{equation*}%
On the other hand, by Corollary \ref{cor2}, $s\left( Y_{U}\right) =s\left(
Y\right) $ and $\sigma \left( X_{L}\right) =\sigma \left( X\right) $.
Consequently, we have 
\begin{equation*}
\frac{1}{\delta\left(Y\right) }\leq \frac{1}{\sigma \left( X\right) }+\frac{1%
}{s}.
\end{equation*}%
Since this holds for all $s<s_{\max }$, it follows 
\begin{equation}
\frac{1}{\delta\left(Y\right) }\leq \frac{1}{\sigma \left( X\right) }+\frac{1%
}{s_{\max }}.  \label{equ2}
\end{equation}%
Observe that the same arguments work also in the case when $s_{\max }=1$,
because every Banach couples $X$ and $Y$ are relatively $1$-decomposable.
Therefore, we again get inequality \eqref{equ2}.

For the opposite inequality, assume first that $s_{\max }=\infty $. Then, %
\eqref{equ2} implies that $\sigma \left( X\right) \leq \delta\left(Y\right) $%
. Combining this inequality with the assumption, we conclude that $\sigma
\left( X\right) =\delta\left(Y\right) $, and hence \eqref{equ2} becomes %
\eqref{equ1}.

Let now $s_{\max }<\infty $. Assume that \eqref{equ1} fails, i.e., 
\begin{equation*}
\frac{1}{\delta \left( Y\right) }<\frac{1}{\sigma \left( X\right) }+\frac{1}{%
s_{\max }}.
\end{equation*}%
If $\delta \left( Y\right) >1$ and $\sigma \left( X\right) <\infty $, we can
find $1\leq p<\delta \left( Y\right) $, $q>\sigma \left( X\right) $ and $%
s>s_{\max }$ such that $1/p=1/q+1/s$. Since $X$ satisfies a lower $q$%
-estimate and $Y$ an upper $p$-estimate, from Proposition \ref%
{Prop_estimates_decomp} it follows that $X$ and $Y$ are relatively $s$%
-decomposable, which is impossible, since $s>s_{\max }$. Thus, in this case %
\eqref{equ1} is proved.

If $\delta \left( Y\right) =1$ or $\sigma \left( X\right) =\infty $, the
proof follows by the same lines in view of the fact that each Banach lattice
admits an upper $1$-estimate and a lower $\infty $-estimate.
\end{proof}

\begin{proof}[Proof of Theorem \protect\ref{Th_main}]
We start with the case when $\delta (Y)\leq \sigma (X)$.

$\left( i\right) \Longrightarrow \left( ii\right) $. Assume first that $%
s_{\max }>1$, $\delta (Y)>1$ and $\sigma (X)<\infty $.

Let $1\leq s<s_{\max }$. Then, by Proposition \ref{Lemma-3a}, we have 
\begin{equation*}
\frac{1}{\delta \left( Y\right) }<\frac{1}{\sigma \left( X\right) }+\frac{1}{%
s}.
\end{equation*}%
Consequently, for some $1\leq p_{1}<\delta \left( Y\right) $ and $%
q_{1}>\sigma \left( X\right) $ we obtain 
\begin{equation*}
\frac{1}{p_{1}}=\frac{1}{\sigma \left( X\right) }+\frac{1}{s}\;\;\mbox{and}%
\;\;\frac{1}{\delta \left( Y\right) }=\frac{1}{q_{1}}+\frac{1}{s}.
\end{equation*}%
Since $\delta \left( Y_{U}\right) =\delta \left( Y\right) $ and $\sigma
\left( X_{L}\right) =\sigma \left( X\right) $ (see Corollary \ref{cor2}), by 
\cite{She92}, $l_{\delta \left( Y\right) }$ (resp. $l_{\sigma \left(
X\right) }$) is finitely lattice representable in $Y_{U}$ (resp. in $X_{L}$%
). Therefore, according to Propositions \ref{Lemma-1} and \ref{Lemma-2}, $X$
satisfies a lower $q_{1}$-estimate, $Y$ satisfies an upper $p_{1}$-estimate
and $M_{\left[ q_{1}\right] }\left( X\right) \leq D_{s}\left( X,Y\right) $, $%
M^{\left[ p_{1}\right] }\left( Y\right) \leq D_{s}\left( X,Y\right) $. Next,
if $1/p=1/q+1/s$, where $p<\delta \left( Y\right) $ and $q>\sigma \left(
X\right) $, we have $p<p_{1}$ and $q>q_{1}$. Hence, $X$ satisfies a lower $q$%
-estimate, $Y$ satisfies an upper $p$-estimate and 
\begin{equation*}
M_{\left[ q\right] }\left( X\right) M^{\left[ p\right] }\left( Y\right) \leq
M_{\left[ q_{1}\right] }\left( X\right) M^{\left[ p_{1}\right] }\left(
Y\right) \leq D_{s}\left( X,Y\right) ^{2}.
\end{equation*}

Suppose now that $X$ and $Y$ are relatively $s_{\max }$-decomposable. Since $%
l_{\delta \left( Y\right) }$ is finitely lattice representable in $Y_{U}$,
by Propositions \ref{Lemma-1} and \ref{Lemma-3a}, $X$ satisfies a lower $%
\sigma \left( X\right) $-estimate and $M_{\left[ \sigma \left( X\right) %
\right] }\left( X\right) \leq D_{s}\left( X,Y\right) $. In the same manner,
applying this time Proposition \ref{Lemma-2}, we infer that $Y$ satisfies an
upper $\delta \left( Y\right) $-estimate and $M^{\left[ \delta \left(
Y\right) \right] }\left( Y\right) \leq D_{s}\left( X,Y\right) $. Combining
this together with equality \eqref{equ1}, we come to the desired result.

If $s_{\max }=1$, or $\delta\left(Y\right)=1$, or $\sigma \left(
X\right)=\infty$, we can use the same arguments, taking into account that
every Banach lattices $X$ and $Y$ are relatively $1$-decomposable and each
Banach couple satisfies an upper $1$-estimate and a lower $\infty$-estimate.

$\left( ii\right)\Longrightarrow \left( i\right)$. This implication together
with the inequality 
\begin{equation*}
D_{s}\left( X,Y\right) \leq M_{\left[ q\right] }\left( X\right) M^{\left[ p%
\right] }\left( Y\right)
\end{equation*}
is an immediate consequence of Proposition \ref{Prop_estimates_decomp}.

To complete the proof of the equivalence of $\left( i\right)$, $\left(
ii\right)$ and $\left( iii\right)$ it remains now to refer to Proposition %
\ref{Prop_rel_dec_lp sp}.

Finally, let us prove the equivalence of the conditions $\sigma(X)\le
\delta(Y)$ and $s_{max}=\infty$.

If $\sigma \left( X\right)<\delta\left(Y\right)$, then $X$ satisfies a lower 
$p$-estimate and $Y$ an upper $p$-estimate for $p\in (\sigma \left(
X\right),\delta\left(Y\right))$. Therefore, by Proposition \ref%
{Prop_estimates_decomp}, $X$, $Y$ are relatively decomposable. Hence, $%
s_{max}=\infty$. If $\sigma \left( X\right)=\delta\left(Y\right)$, the same
result follows from Proposition \ref{Lemma-3a}.

On the contrary, assume that $s_{max}=\infty$. Then, $X$ and $Y$ are
relatively $s$-decomposable for each $s<\infty$. Therefore, since $%
\delta\left(Y_U\right)=\delta\left(Y\right)$ and $\sigma \left(
X_L\right)=\sigma\left( X\right)$, by Proposition \ref{Lemma-3}, we infer 
\begin{equation*}
\frac{1}{\delta\left(Y\right) }\le \frac{1}{\sigma \left( X\right) }+\frac{1%
}{s}.
\end{equation*}%
Tending $s\to\infty$, we get the required inequality, and so the proof is
completed.
\end{proof}

Recall that the main result of the paper \cite{CwNiSc03} (see Theorem 1.3)
reads that Banach function lattices $X$, $Y$ are relatively decomposable (or 
$\infty$-decomposable) if and only if there exists $p\ge 1$ such that $X$
satisfies a lower $p$-estimate and $Y$ an upper $p$-estimate. As an
immediate consequence of Theorem \ref{Th_main} and its proof we obtain the
following extension of this result to general Banach lattices.

\begin{corollary}
\label{main cor} Banach lattices $X$, $Y$ are relatively decomposable if and
only if there exists $p\ge 1$ such that $X$ satisfies a lower $p$-estimate
and $Y$ an upper $p$-estimate.
\end{corollary}

\begin{remark}
\label{main rem} In contrast to \cite{CwNiSc03}, our definition of relative
decomposability (see Definition \ref{Def-Rel-Decomposable}) deals only with
finite sums. Thanks to that, we need not to impose on lattices $X$ and $Y$
any extra condition. In particular, if $X$ and $Y$ are Banach lattices of
measurable functions on a $\sigma$-finite measure space we omit the
assumption from \cite[Theorem~1.3]{CwNiSc03} that $Y$ has the Fatou property.
\end{remark}

From Proposition \ref{coincidence with lp} and the proof of Theorem \ref%
{Th_main} we also deduce the following result.

\begin{corollary}
\label{main cor2} If Banach lattices $X$, $Y$ are relatively $s_{max}$%
-decomposable, then $X$ admits a lower $\sigma(X)$-estimate and $Y$ admits
an upper $\delta(Y)$-estimate (equivalently, $X_L=l_{\sigma(X)}$ and $%
Y_U=l_{\delta(Y)}$).
\end{corollary}

\vskip0.5cm

\section{Applications to interpolation theory: Calderón-Mityagin couples of
type $s$.}

\label{Inter}

In this section, we freely use notation and results from interpolation
theory as in \cite{BK91}, \cite{BeLo76}, \cite{BrKeSe88}.$\medskip \medskip $

Let $\left( X,\Sigma ,\mu \right) $ be a $\sigma $-finite measure space. A $%
\Sigma $-measurable function $\omega $ is called a \textit{weight} if $%
\omega $ is non-negative $\mu $-a.e. on $X$. Let $1\leq p\leq \infty $ and
let $L_{p}\left( \omega ,\mu \right) $ be the Banach space of all
(equivalence classes of) $\Sigma $-measurable functions $f$ with $f\omega
\in L_{p}\left( \mu \right) .$ Given $1\leq p_{0},p_{1}\leq \infty $ put $%
\overrightarrow{p}=\left( p_{0},p_{1}\right) .$ A Banach couple $%
\overrightarrow{U}=(U_{0},U_{1})$ of Banach lattices is called a \textit{$L_{%
\overrightarrow{p}}$-couple} if $U_{i}=L_{p_{i}}\left( \omega _{i},\mu
\right) $, $i=0,1$, for some measure space $\left( X,\Sigma ,\mu \right) $
and some weights $\omega _{0},\omega _{1}$ with respect to this measure
space.$\medskip $

Let $1\leq s_{0},s_{1}\leq \infty $ and $\overrightarrow{X},\overrightarrow{Y%
}$ be two Banach couples of Banach lattices such that $X_{i},Y_{i}$ are
relatively $s_{i}$-decomposable for $i=0,1$. Then, by Theorem \ref{Th_main},
there exist $1\leq p_{0},p_{1},q_{0},q_{1}\leq \infty $ with $%
1/p_{i}=1/q_{i}+1/s_{i}$, $i=0,1$, such that for every $L_{\overrightarrow{q}%
}$-couple $\overrightarrow{U}=(U_{0},U_{1})$ and $L_{\overrightarrow{p}}$%
-couple $\medskip \medskip \overrightarrow{V}=(V_{0},V_{1})$ both $%
X_{i},U_{i}$ and $V_{i},Y_{i}$ are relative decomposable for $i=0,1$. $%
\medskip $

Combining the last observation with the results of \cite{CwNiSc03}, we see
that each of the pairs of the couples $\overrightarrow{X},\overrightarrow{U}$
and $\overrightarrow{V},\overrightarrow{Y}$ have the relative Calderó%
n-Mityagin property ($\mathcal{C-M}$ property). Hence, the $s$%
-decomposability relation of couples of Banach lattices has some
transitivity property, which is manifested in factorization of this relation
through the canonical $s$-decomposability of suitable $L_{\overrightarrow{q}}
$- and $L_{\overrightarrow{p}}$-couples. More precisely, we get the
following result.

\begin{theorem}
\label{factor1} Let $\overrightarrow{X},\overrightarrow{Y}$ be two couples
of Banach lattices over a $\sigma $-finite measure space. If the spaces $%
X_{i},Y_{i}$ are relative $s_{i}$-decomposable for $i=0,1$, where $1\leq
s_{i}\leq \infty ,$ then there exist pairs $\overrightarrow{p}$, $%
\overrightarrow{q}$ of parameters such that, for every $L_{\overrightarrow{q}%
} $-couple $\overrightarrow{U}=(U_0,U_1)$ and every $L_{\overrightarrow{p}}$%
-couple $\overrightarrow{V}=(V_0,V_1),$ pairs of the couples $%
\overrightarrow{X},\overrightarrow{U}$ and $\overrightarrow{V},%
\overrightarrow{Y}$ have the relative $\mathcal{C-M}$ property and the
spaces $U_{i},V_{i}$ are relatively $s_{i}$-decomposable, $i=0,1$.
\end{theorem}

There are many pairs of Banach couples $\overrightarrow{X}$ and $%
\overrightarrow{Y}$, which fail to have relative $\mathcal{C-M}$ property.
In \cite{Cwi84}, Cwikel introduced the following weaker condition that may
be satisfied by such a pair of Banach couples. $\medskip $

Let $\overrightarrow{X}=(X_{0},X_{1})$ and $\overrightarrow{Y}=(Y_{0},Y_{1})$
be two Banach couples. Given $1\leq s\leq \infty ,$ define the relation $%
R_{s}$ for $(x,y)\in (X_0+X_1)\times (Y_0+Y_1)$ by 
\begin{equation*}
xR_{s}y\iff \exists w\in L_{s}\left( (0,\infty ),dt/t\right) \;\mbox{with}%
\;K( t,y;\overrightarrow{Y}) \leq w(t)\cdot K( t,x;\overrightarrow{X})
,\;t>0.
\end{equation*}%
We say that the Banach couples $\overrightarrow{X},\overrightarrow{Y}$ are
of \textit{relative $\mathcal{C-M}$ type} $s$ whenever the relation $xR_{s}y$
implies that $y=Tx$ for some linear operator $T:\,\overrightarrow{X}%
\rightarrow \overrightarrow{Y}$ (i.e., $T:\,X_0+X_1\rightarrow Y_0+Y_1$, and 
$T$ is bounded from $X_i$ into $Y_i$, $i=0,1$). $\medskip $

Since each $K$-functional is a concave nondecreasing function in $t$, we can
assume that the function $w$ in this definition is continuous or constant on
each dyadic interval. From this observation it follows easily that if $%
\overrightarrow{X},\overrightarrow{Y}$ are of relative $\mathcal{C-M}$ type $%
s_{1}$ and $1\leq s_{2}\leq s_{1}$, then these couples are also of relative $%
\mathcal{C-M}$ type $s_{2}$. Furthermore, it is known \cite[Theorem 1]{Cwi76}
that arbitrary couples $\overrightarrow{X},\overrightarrow{Y}$ are of
relative $\mathcal{C-M}$ type $1$. Hence, the set of real numbers $s$ in $%
\left[ 1,\infty \right] $ such that $\overrightarrow{X},\overrightarrow{Y}$
are of relative $\mathcal{C-M}$ type $s$ is an interval which includes $1.$
In \cite{Cwi76} and \cite{Cwi84} one can find examples of Banach couples,
for which this interval is $[1,q]$, $1\leq q<\infty $, or $[1,\infty )$ (of
course, it is $\left[ 1,\infty \right] $ iff $\overrightarrow{X},%
\overrightarrow{Y}$ have the relative $\mathcal{C-M}$ property). $\medskip $

Further, in \cite{Cwi84}, Cwikel proved that, if the couples $%
\overrightarrow{X},\overrightarrow{Y}$ are mutually closed and $X_{i},Y_{i}$%
, $i=0,1$, are relatively $s$-decomposable for some $1\leq s\leq \infty $,
then these couples are of relative $\mathcal{C-M}$ type $s$ (see also \cite[%
p. 606]{BK91}). Let us show that, under some conditions, this implies the
orbital factorization of relative $K$-functional estimates for such couples
through suitable $L_{\overrightarrow{p}}$- and $L_{\overrightarrow{q}}$%
-couples. $\medskip $

Given Banach couples $\overrightarrow{X},\overrightarrow{Y}$ the couple $%
\overrightarrow{Y}$ is called \textit{$\overrightarrow{X}$-abundant}, if for
each element $x\in X_{0}+X_{1}$ there exists $y\in Y_{0}+Y_{1}$ such that 
\begin{equation*}
K( t,x;\overrightarrow{X}) \asymp K( t,y;\overrightarrow{Y})
\end{equation*}%
with constants independent of $x\in X_{0}+X_{1}$ and $t>0$ (see e.g. \cite[
Definition 4.4.8]{BK91}). For instance, if a couple $\overrightarrow{X}$ is
regular (i.e., $X_0\cap X_1$ is dense in $X_{0}$ and $X_{1}),$ then the $L_{%
\overrightarrow{p}}$-couples $\left( l_{p_{0}}\left( \mathbb{Z},\left(
1\right) _{n\in \mathbb{Z}}\right) ,l_{p_{1}}\left( \mathbb{Z},\left(
2^{-n}\right) _{n\in \mathbb{Z}}\right) \right) $ and $\left(
L_{p_{0}}\left( \mathbb{R}_{+},dt/t\right) ,L_{p_{1}}\left( \mathbb{R}%
_{+},dt/t\right) \right) $ are $\overrightarrow{X} $-abundant for each pair $%
\overrightarrow{p}=\left( p_{0},p_{1}\right) $ \cite[Theorem 4.5.7]{BK91}.
With this notation, we have the following version of Theorem \ref{factor1}.

\begin{theorem}
\label{factor3} Let $\overrightarrow{X}=(X_{0},X_{1})$ and $\overrightarrow{Y%
}=(Y_{0},Y_{1})$ be two Banach lattice couples over a $\sigma $-finite
measure space such that $X_{i},Y_{i}$, $i=0,1$, are relatively $s$%
-decomposable for some $1\leq s\leq \infty $. Then, there are pairs $%
\overrightarrow{p}=(p_{0},p_{1})$ and $\overrightarrow{q}=(q_{0},q_{1})$ of
parameters such that for every $L_{\overrightarrow{q}}$-couple $%
\overrightarrow{U}$, which is $\overrightarrow{X}$-abundant, and every $L_{%
\overrightarrow{p}}$-couple $\overrightarrow{V}$, which is $\overrightarrow{Y%
}$-abundant, we have the following: If $x\in X_{0}+X_{1},y\in Y_{0}+Y_{1}$
satisfy the relation $xR_{s}y$, then there exist linear operators $T_{0}:%
\overrightarrow{X}\rightarrow \overrightarrow{U}$, $T_{1}:\overrightarrow{U}%
\rightarrow \overrightarrow{V}$, $T_{2}:\overrightarrow{V}\rightarrow 
\overrightarrow{Y}$ such that $y=T_{2}T_{1}T_{0}x$.
\end{theorem}

\begin{proof}
Applying first Theorem \ref{factor1}, we find parameters $1\leq
p_{i},q_{i}\leq \infty $, $1/p_{i}=1/q_{i}+1/s$, $i=0,1$, such that, if $%
\overrightarrow{p}=(p_{0},p_{1})$, $\overrightarrow{q}=(q_{0},q_{1})$, then
for every $L_{\overrightarrow{q}}$-couple $\overrightarrow{U}=(U_{0},U_{1})$
and every $L_{\overrightarrow{p}}$-couple $\overrightarrow{V}=(V_{0},V_{1}),$
pairs of the couples $\overrightarrow{X},\overrightarrow{U}$ and $%
\overrightarrow{V},\overrightarrow{Y}$ have relative $\mathcal{C-M}$
property and the spaces $U_{i},V_{i}$, $i=0,1$, are relatively $s$%
-decomposable. Next, assuming that $x\in X_{0}+X_{1},y\in Y_{0}+Y_{1}$
satisfy $xR_{s}y$, by using the abundance assumption, we can select $u\in
U_{0}+U_{1}$ and $v\in V_{0}+V_{1}$ such that 
\begin{eqnarray*}
K( t,x;\overrightarrow{X}) &\asymp &K( t,u;\overrightarrow{U}) \\
K( t,y;\overrightarrow{Y}) &\asymp &K( t,v;\overrightarrow{V})
\end{eqnarray*}%
with constants independent of $x\in X_{0}+X_{1}$, $y\in Y_{0}+Y_{1}$ and $%
t>0 $. Since the couples $\overrightarrow{X},\overrightarrow{U}$ and $%
\overrightarrow{V},\overrightarrow{Y}$ have relative $\mathcal{C-M}$
property, we can find linear operators $T_{0}:\overrightarrow{X}\rightarrow 
\overrightarrow{U}$ and $T_{2}:\overrightarrow{V}\rightarrow \overrightarrow{%
Y}$ satisfying $u=T_{0}x$ and $y=T_{2}v.$ Moreover, as was above-mentioned
(see \cite{Cwi84}), the couples $\overrightarrow{U}$ and $\overrightarrow{V}$
are of relative $\mathcal{C-M}$ type $s$. Hence, from the relation $xR_{s}y$
it follows the existence of a linear operator $T_{1}:\overrightarrow{U}%
\rightarrow \overrightarrow{V}$ such that $v=T_{1}u.$
\end{proof}

Assume now that $\overrightarrow{X}=(X_{0},X_{1})$ and $\overrightarrow{Y}%
=(Y_{0},Y_{1})$ are two Banach lattice couples such that $X_{i},Y_{i}$ are
relatively $\infty $-decomposable for $i=0,1$. Then, the results of \cite%
{CwNiSc03} imply that the couples $\overrightarrow{X}$ and $\overrightarrow{Y%
}$ have relative $\mathcal{C-M}$ property. Arguing in the same way as in the
proof of Theorem \ref{factor3}, one can easily deduce the following
factorization result.

\begin{theorem}
\label{factor2} Let $\overrightarrow{X}=(X_{0},X_{1})$ and $\overrightarrow{Y%
}=(Y_{0},Y_{1})$ be two Banach lattice couples over a $\sigma $-finite
measure space such that $X_{0},Y_{0}$ and $X_{1},Y_{1}$ are relatively
decomposable. Then, there is a pair $\overrightarrow{p}=(p_{0},p_{1})$ of
parameters such that for every $L_{\overrightarrow{p}}$ couples $%
\overrightarrow{U}$ and $\overrightarrow{V}$ such that $\overrightarrow{U}$
is $\overrightarrow{X}$-abundant and $\overrightarrow{V}$ is $%
\overrightarrow{Y}$-abundant we have the following: If $x\in
X_{0}+X_{1},y\in Y_{0}+Y_{1}$ satisfy 
\begin{equation}
K( t,y;\overrightarrow{Y}) \leq K( t,x;\overrightarrow{X}) ,\;\;t>0,
\end{equation}%
then there exist linear operators $T_{0}:\overrightarrow{X}\rightarrow 
\overrightarrow{U}$, $T_{1}:\overrightarrow{U}\rightarrow \overrightarrow{V}$%
, $T_{2}:\overrightarrow{V}\rightarrow \overrightarrow{Y}$ with $%
y=T_{2}T_{1}T_{0}x.$
\end{theorem}

\vskip0.5cm

\section{\label{R.invariance} The proof of Theorem \protect\ref%
{Th_XL_XU_Prop}.}

This proof will be broken down into a number of lemmas and propositions. The
main step is Proposition \ref{Prop_main_Prop_xl_xu} showing under which
conditions a scale of norms on $\mathbb{R}^{n}$, $n\in\mathbb{N}$, generates
a rearrangement invariant Banach sequence lattice. The rest of this section
is to secure that these conditions are valid for both $X_{L}$- and $X_{U}$%
-constructions. $\medskip $

Recall that a functional $\Psi $ (in particular, a norm $\left\Vert \cdot
\right\Vert $) defined on $\mathbb{R}^{n}$ is called \textit{lattice monotone%
} or \textit{lattice norm} if for any elements $a=\left\{ a_{i}\right\}
_{i=1}^{n},b=\left\{ b_{i}\right\} _{i=1}^{n}\in \mathbb{R}^{n}$ such that $%
\left\vert a_{i}\right\vert \leq \left\vert b_{i}\right\vert ,1\leq i\leq n$%
, it holds $\Psi \left( a\right) \leq \Psi \left( b\right) .$ This
functional is said to be \textit{symmetric }if for any permutation $\sigma $
of the set $\left\{ 1,\dots ,n\right\} $ we have $\Psi \left( \sigma
a\right) =\Psi \left( a\right) $ where $\sigma a=\left\{ a_{\sigma \left(
i\right) }\right\} _{i=1}^{n}.$ We introduce also the operators%
\begin{eqnarray*}
I_{n} &:&\mathbb{R}^{n}\rightarrow \mathbb{R}^{n-1},\left\{ a_{i}\right\}
_{i=1}^{n}\mapsto \left\{ a_{i}\right\} _{i=1}^{n-1} \\
Tr_{n} &:&\mathbb{R}^{n}\rightarrow \mathbb{R}^{n},\left\{ a_{i}\right\}
_{i=1}^{n}\mapsto \left\{ 
\begin{array}{c}
a_{i}:i\neq n \\ 
0:i=n%
\end{array}%
\right\}
\end{eqnarray*}

As above, for any sequence $a=\left\{ a_{i}\right\} _{i=1}^{\infty }$ of
real numbers and each integer $k$, by $a^{\left( k\right) }$ we will denote
the truncated sequence $a^{\left( k\right) }$ defined by $a^{\left( k\right)
}=\{a_{i}^{(k)}\}_{i=1}^{\infty }$, with $a_{i}^{(k)}=a_{i}$ if $1\leq i\leq
k$ and $a_{i}^{(k)}=0$ if $i>k$.

\begin{proposition}
\label{Prop_main_Prop_xl_xu} Let $\left\Vert \cdot \right\Vert _{n}$ be
symmetric lattice norms on $\mathbb{R}^{n}$, $n\in \mathbb{N}$. Assume that
the restrictions $I_{n}$ are contractive with respect to these norms. Denote
by $Y$ the space of all sequences $a=\left\{ a_{i}\right\} _{i=1}^{\infty }$%
, for which the norm 
\begin{equation*}
\left\Vert a\right\Vert _{Y}:=\sup_{n\geq 1}\left\Vert \left\{ a_{i}\right\}
_{i=1}^{n}\right\Vert _{n}
\end{equation*}%
is finite. If the space $Y$ is embedded into $c_{0},$ then $Y$ is a r.i.
Banach sequence lattice.
\end{proposition}

\begin{proof}
First, one can easily check that the conditions $a=\left\{
a_{i}\right\}_{i=1}^{\infty }\in Y$ and $\left\vert a_{i}\right\vert \leq
\left\vert b_{i}\right\vert$, $i=1,2,\dots$, imply that $b=\left\{
b_{i}\right\}_{i=1}^{\infty }\in Y$ and $\|b\|_Y\le \|a\|_Y$. Consequently, $%
a\mapsto\left\Vert a\right\Vert _{Y}$ is a lattice norm on $Y$.

To prove the rearrangement invariance of $Y,$ assume that $a=\left\{
a_{i}\right\}_{i=1}^{\infty }\in Y$ and a sequence $b=\left\{
b_{i}\right\}_{i=1}^{\infty }$ is equi-measurable with $a.$ This means that
the sets $\left\{ i:\left\vert a_{i}\right\vert >t\right\} $ and $\left\{
i:\left\vert b_{i}\right\vert >t\right\} $ have the same cardinality for
every $t>0$. Since $Y$ is embedded into $c_{0}$ these sets are finite and
hence the sets $A_{t}:=\left\{ i:\left\vert a_{i}\right\vert=t\right\} $ and 
$B_{t}:=\left\{ i:\left\vert b_{i}\right\vert =t\right\} $ also have the
same cardinality for each $t>0.$ Put $t_{k}:=\left\vert b_{k}\right\vert $, $%
B_{k}:=B_{t_{k}}$, $A_{k}:=A_{t_{k}}$, $k\in\mathbb{N}$.

Let $n\in \mathbb{N}$ be arbitrary. Take $u_{n}\in \mathbb{N}$ such that $%
\cup _{k=1}^{n}A_{k}\subseteq \left\{ 1,2,\dots ,u_{n}\right\} .$ Then, we
have $\left\{ \left\vert b_{k}\right\vert \right\} _{k=1}^{n}\subseteq
\left\{ \left\vert a_{k}\right\vert \right\} _{k=1}^{u_{n}}$. Indeed, if $%
1\leq k\leq n,$ then by construction $\left\vert b_{k}\right\vert =t_{k}\in
B_{k}$ and hence there exists $j\in A_{k}$ with $\left\vert a_{j}\right\vert
=\left\vert b_{k}\right\vert .$ Since $A_{k}\subseteq \left\{ 1,2,\dots
,u_{n}\right\} $, the conclusion follows.

Next, there is a permutation $\sigma $ of the set $\left\{1,2,\dots,u_{n}%
\right\} $ with $\left( \sigma \left\vert a\right\vert \right)
_{k}=\left\vert b_{k}\right\vert$, $1\leq k\leq n.$ By the assumptions of
the lemma, this implies the estimate 
\begin{equation*}
\left\Vert \left\{ b_{k}\right\} _{k=1}^{n}\right\Vert _{n}=\left\Vert
\left\{\left\vert b_{k}\right\vert \right\} _{k=1}^{n}\right\Vert
_{n}=\left\Vert \left\{ \sigma \left\vert a\right\vert \right\}
_{k=1}^{n}\right\Vert _{n}\le \left\Vert\left\{ \sigma \left\vert
a\right\vert \right\} _{k=1}^{u_n}\right\Vert_{u_n}.
\end{equation*}%
Hence, 
\begin{equation*}
\left\Vert \left\{ b_{k}\right\} _{k=1}^{n}\right\Vert _{n}\leq \left\Vert
a\right\Vert _{Y},\;\;n\in\mathbb{N},
\end{equation*}%
and so $b\in Y$ and $\left\Vert b\right\Vert _{Y}\leq \left\Vert
a\right\Vert _{Y}.$ Similarly, $\left\Vert a\right\Vert _{Y}\leq \left\Vert
b\right\Vert _{Y},$ and thus $\left\Vert a\right\Vert _{Y}=\left\Vert
b\right\Vert _{Y}.$

By construction, $Y$ is a normed linear space of sequences. To prove
completeness of $Y$, take $\left\{ a^{n}\right\} _{n=1}^{\infty }\subseteq Y$%
, $a^{n}=\left\{ a_{i}^{n}\right\} _{i=1}^{\infty }$, with $%
\sum_{n=1}^{\infty }\left\Vert a^{n}\right\Vert _{Y}=C<\infty.$ Since $Y$ is
embedded in $c_{0}$, there exists $a\in c_{0}$ with $a=\sum_{n=1}^{\infty
}a^{n}.$ Also, for each integer $k$ we have 
\begin{equation*}
\sum_{n=1}^{\infty }\Vert \{a_{i}^{n}\}_{i=1}^{k}\Vert _{k}\leq C.
\end{equation*}%
Hence, by completeness of the space $\mathbb{R}^{k}$ with respect to the
norm $\left\Vert \cdot \right\Vert _{k}$ and uniqueness of a representation
of vectors by using the canonical unit basis, we get $a^{\left( k\right)
}=\sum_{n=1}^{\infty }(a^{n})^{(k)}$ and $\left\Vert a^{\left( k\right)
}\right\Vert _{k}\leq C$ for all $k\in \mathbb{N}$. Consequently, $a\in Y$
and $\left\Vert a\right\Vert _{Y}\leq C.$ The proposition is proved.
\end{proof}

Next, we proceed with the postponed proof of Lemma \ref{Lemma_bn} on the
nonemptiness of the sets $\mathfrak{B}_{n}\left( X\right) $, $n\in\mathbb{N}$%
. We will use the notation $X_{+}$ for the positive cone $\left\{ x\in
X:\,x\geq 0\right\} $ of a Banach lattice $X.$

\begin{proof}[Proof of Lemma \protect\ref{Lemma_bn}]
If $l_{\infty }$ is finite lattice representable in $X$, then the desired
result follows immediately from Definition \ref{Def_ap}. Therefore, we can
assume that $l_{\infty }$ fails to be finite lattice representable in $X$,
and hence, by Proposition \ref{Prop_not_A-infinity}, $X$ is both $\sigma $%
-complete and $\sigma $-order continuous. This implies that for each $x\in
X_{+}$ we can define the contractive projection $P_{x}:X\rightarrow X$ by $%
P_{x}\left( y\right) =\vee _{n\geq 1}\left( nx\wedge y\right) ,y\in X_{+},$
and then extend it by linearity to the whole of $X$ (see e.g. \cite{LT79}).

Suppose that $\left\{ x_{i}\right\} _{i=1}^{m}$, where $m\in \mathbb{N}$, is
a maximal sequence of normalized positive pair-wise disjoint elements in a
Banach lattice $X$. We claim that $X$ has dimension not bigger than $m.$ We
will divide the proof of this fact into several parts.

$\left( i\right) $ Each element $x\in \left\{ x_{i}\right\} _{i=1}^{m}$ is
an atom.

Assume that $x=y+z$ for some $y,z$ with $\left\vert y\right\vert \wedge
\left\vert z\right\vert =0.$ Since $x>0$, we have $x=\left\vert y\right\vert
+\left\vert z\right\vert $, and thus $0\leq \left\vert y\right\vert
,\left\vert z\right\vert \leq x.$ Hence, by maximality, $\left\vert
y\right\vert =x$ or $\left\vert z\right\vert =x$, i.e., $x$ is an atom.

$\left( ii\right)$ For every $x\in \left\{ x_{i}\right\} _{i=1}^{m}$ the
projection $P_{x}$ has one dimensional range.

$\ $Recall that (see \cite[p. 10]{LT79}) 
\begin{equation}
\mathrm{Im}P_{x}=\left\{ z\in X:\, x\wedge y=0\; \mbox{for some}\;y\in X_+\;
\Longrightarrow \left\vert z\right\vert \wedge y=0\right\}.  \label{Im_Px}
\end{equation}

Putting $z=P_{x}\left( y\right) $, where $y\in Y_{+}$, we have $z\ge 0$.
Without loss of generality, assume that $z>0.$ From \eqref{Im_Px} it follows
that $z\wedge x_{i}=0$ whenever $x_{i}\neq x.$ If $z\wedge x=0$ we get a
contradiction, because the set $\left\{ x_{i}\right\}_{i=1}^{m}$ was
selected to be maximal. Hence, $0<z\wedge x\leq x$ and, since $x$ is an
atom, we conclude that $z\wedge x=\lambda x$ for some $\lambda>0.$ Observe
that the set $(\left\{ x_{i}\right\} _{i=1}^{m}\smallsetminus \left\{
x\right\})\cup \left\{ z/\|z\|_X\right\}$ is also a maximal set of
normalized positive pair-wise disjoint elements in $X$. Consequently, from $%
\left(i\right) $ it follows that $z$ is an atom. Since $\lambda x=x\wedge
z\leq z$, this implies that $\lambda x=\mu z$ for some scalar $\mu>0.$
Hence, $P_{x}$ has one dimensional range, generated by the vector $x.$

$\left( iii\right)$ $X$ is the linear span of the sequence $\left\{
x_{i}\right\} _{i=1}^{m}.$

Put $x=\vee _{i=1}^{m}x_{i}$ and take $y\in X_{+}.$ Then, if $%
z:=P_{x}\left(y\right) $, we have $x\wedge \left( y-z\right) =0.$ From the
inequalities $0\leq x_{i}\leq x$ and $0\leq z\leq y$ it follows that $%
x_{i}\wedge \left( y-z\right) =0$ and hence, by maximality, we have $%
y=z=P_{x}\left( y\right) .$ Since $x\wedge y=\vee _{i=1}^{m}\left(
x_{i}\wedge y\right) $, we have for each integer $n$%
\begin{equation*}
nx\wedge y=\vee _{i=1}^{m}\left( nx_{i}\wedge y\right)
=\sum_{i=1}^{m}nx_{i}\wedge y\leq \sum_{i=1}^{m}P_{x_{i}}\left( y\right),
\end{equation*}%
which implies that 
\begin{equation*}
y=P_{x}\left( y\right) \leq \sum_{i=1}^{m}P_{x_{i}}\left( y\right).
\end{equation*}%
By the decomposition property, we may write $y=\sum_{i=1}^{m}y_{i}$, where $%
0\leq y_{i}\leq P_{x_{i}}\left( y\right) $, and hence $y_{i}\in \mathrm{Im}%
P_{x_{i}}.$ Therefore, by $\left( ii\right)$, $y_{i}=\lambda _{i}x_{i}$ for
some scalars $\lambda _{i}$ and thus $y=\sum_{i=1}^{m}\lambda _{i}x_{i}.$ As
a result, the claim is proven and so the lemma follows.
\end{proof}

\begin{lemma}
\label{Lemma_bn_extension} Let $X$ be an infinite dimensional Banach lattice
such that $l_{\infty }$ is not finitely lattice representable in $X$. Then,
for every sequence $\left\{ x_{i}\right\} _{i=1}^{n}\in \mathfrak{B}%
_{n}\left( X\right) $ and $\varepsilon >0$ there exists a sequence $\left\{
u_{i}\right\} _{i=1}^{n+1}\in \mathfrak{B}_{n+1}\left( X\right) $ such that
either

$\left( i\right):$ $u_{i}=x_{i}$, $i=1,\dots,n$,

or

$\left( ii\right):$ there exists $k$ with $1\leq k\leq n$ and a bijection $%
\psi :\left\{ 1,..,n-1\right\} \rightarrow \left\{ 1,...,n\right\}
\smallsetminus \left\{ k\right\} $ such that $u_{i}=x_{\psi \left( i\right)
},1\leq i\leq n-1$ and $\alpha u_{n}+\beta u_{n+1}=x_{k}$ for some positive
scalars $\alpha$ and $\beta.$ Moreover, we have 
\begin{equation*}
\left\Vert u_{n}-x_{k}\right\Vert _{X}\leq \varepsilon.
\end{equation*}
\end{lemma}

\begin{proof}
Given $\left\{ x_{i}\right\} _{i=1}^{n}\in \mathfrak{B}_{n}\left( X\right) $
we put $x=\vee _{i=1}^{n}x_{i}=\sum_{i=1}^{n}x_{i}.$ Take $y\in X_{+}$ and
set $z=P_{x}\left( y\right) .$ Then $x\wedge \left(y-z\right) =0.$ If there
exists $y$ such that $y\neq P_{x}\left( y\right) $, we define $%
u_{n+1}:=\lambda \left( y-P_{x}\left( y\right) \right) $, where $\lambda $
is selected so that $\left\Vert u_{n+1}\right\Vert =1.$ Then, setting $%
u_i=x_i$, $i=1,\dots,n$, we see that the case $\left(i\right) $ holds.

Therefore, we may assume that $P_{x}\left( y\right) =y$ for each $y\in
X_{+}. $ Hence, as in the proof of Lemma \ref{Lemma_bn}, it follows that $X$
is the direct sum of the bands $P_{x_{i}}\left( X\right) $, $i=1,2,\dots ,n$%
, and hence at least one of them, say, $P_{x_{k}}\left( X\right) $, is
infinite dimensional. Since $x_{k}$ can not be an atom, we may write $%
x_{k}=u+v$, where $u,v\in X_{+}$ and $u\wedge v=0.$ Further, at least one of
the subspaces $P_{u}\left( X\right) $ or $P_{v}\left( X\right) $ is again
infinite dimensional. Arguing in the same way, we conclude that, for any
positive integer $m,$ $x_{k}$ is a sum of $m$ pair-wise disjoint elements $%
w_{j}$, $j=1,2,\dots ,m$. Without loss of generality, assume that $\Vert
w_{1}\Vert _{X}\geq \Vert w_{2}\Vert _{X}\geq \dots \geq \Vert w_{m}\Vert
_{X}$. By the assumption (see also Proposition \ref{Prop_dia}), $X $
satisfies a lower $p$-estimate for some $p<\infty .$ Consequently, we can
estimate 
\begin{equation*}
m^{1/p}\left\Vert w_{m}\right\Vert _{X}\leq \left( \sum_{i=1}^{m}\left\Vert
w_{j}\right\Vert _{X}^{p}\right) ^{1/p}\leq M_{\left[ p\right] }\left(
x\right) \left\Vert x_{k}\right\Vert _{X}=M_{\left[ p\right] }\left(
X\right) ,
\end{equation*}%
whence $\lim_{m\rightarrow \infty }\left\Vert w_{m}\right\Vert _{X}=0.$

Put 
\begin{equation*}
u_{n}=\left( x_{k}-w_{m}^{\left( m\right) }\right) /\left\Vert
x-w_{m}^{\left( m\right) }\right\Vert _{X},\;\;u_{n+1}=w_{m}^{\left(
m\right) }/\left\Vert w_{m}^{\left( m\right)}\right\Vert _{X}.
\end{equation*}
Since $u_{n}\wedge u_{n+1}=0$, we have $\left\{u_{i}\right\} _{i=1}^{n+1}\in 
\mathfrak{B}_{n+1}\left( X\right) $. Moreover, by construction, $%
x_{k}=\alpha u_{n}+\beta u_{n+1}$, with $\alpha =\Vert x_{k}-w_{m}^{\left(
m\right) }\Vert _{X}$, $\beta =\Vert w_{m}^{\left( m\right) }\Vert _{X}.$
Finally, since 
\begin{eqnarray*}
\left\Vert x_{k}-u_{n}\right\Vert _{X}&=&\left\Vert \frac{\left( \alpha
-1\right) x_{k}+w_{m}^{\left( m\right) }}{\alpha }\right\Vert _{X} \\
&\leq& \frac{\vert\Vert x_{k}-w_{m}^{\left( m\right) }\Vert _{X}-\Vert
x_{k}\Vert _{X}\vert }{\alpha }\left\Vert x_{k}\right\Vert _{X}+\frac{\Vert
w_{m}^{\left( m\right) }\Vert _{X}}{\alpha } \\
&\leq& \frac{2}{\alpha }\left\Vert w_{m}^{\left( m\right) }\right\Vert _{X},
\end{eqnarray*}%
we may select $m$ so that $\left\Vert x_{k}-u_{n}\right\Vert
_{X}<\varepsilon .$ Thus, all the conditions in $\left( ii\right) $ are
fulfilled.
\end{proof}

\begin{remark}
In the above proof we required that $l_{\infty }$ fails to be finitely
lattice representable in $X$. But, in fact, we need only a weaker property
that if $\left\{ x_{n}\right\} _{n=1}^{\infty }$ is an infinite sequence of
pair-wise disjoint elements with decreasing norms in $X$, then $\left\Vert
x_{n}\right\Vert _{X}\downarrow 0.$ According the terminology from the book 
\cite{AlBu03}, such a Banach lattice $X$ is said to have the Lebesgue
property (see \cite[Theorem 3.22]{AlBu03}).
\end{remark}

\begin{lemma}
\label{Lemma_xln_propeties} Each of the functionals $\left\Vert \cdot
\right\Vert _{X_{U}\left( n\right) },\Phi _{n}\left( \cdot \right) $ and $%
\left\Vert \cdot \right\Vert _{X_{L}\left( n\right) }$ defined on $\mathbb{R}%
^{n}$ is lattice monotone and symmetric.
\end{lemma}

\begin{proof}
Let $a=\left\{ a_{i}\right\} _{i=1}^{n}$ and $b=\left\{ b_{i}\right\}
_{i=1}^{n}$ be two sequences of scalars with $\left\vert b_{i}\right\vert
\leq \left\vert a_{i}\right\vert$, $1\leq i\leq n.$ Then, for every sequence 
$\left\{ x_{i}\right\} _{i=1}^{n}\in \mathfrak{B}_{n}\left( X\right) $ we
have 
\begin{equation}  \label{elem}
\left\vert \sum_{i=1}^{n}b_{i}x_{i}\right\vert =\sum_{i=1}^{n}\left\vert
b_{i}\right\vert \left\vert x_{i}\right\vert\leq \sum_{i=1}^{n}\left\vert
a_{i}\right\vert \left\vert x_{i}\right\vert =\left\vert
\sum_{i=1}^{n}a_{i}x_{i}\right\vert.
\end{equation}
Also, let $\sigma $ be a permutation of the set $\left\{ 1,\dots,n\right\} $
and the sequence $c$ be defined by $c=\sigma a:=\left\{ a_{\sigma \left(
i\right) }\right\} _{i=1}^{n}.$ Further, we prove the desired claims for
each functional separately.

$\left( i\right) :\left\Vert \cdot \right\Vert _{X_{U}\left( n\right) }.$
From \eqref{elem} it follows that 
\begin{equation*}
\left\Vert \sum_{i=1}^{n}b_{i}x_{i}\right\Vert _{X}\leq \left\Vert
\sum_{i=1}^{n}a_{i}x_{i}\right\Vert _{X},
\end{equation*}
which implies $\left\Vert b\right\Vert _{X_{U}\left( n\right) }\leq
\left\Vert a\right\Vert _{X_{U}\left( n\right) }$. Consequently, $\left\Vert
\cdot \right\Vert _{X_{U}\left( n\right) }$ is a lattice norm.

Next, since for every $\left\{ x_{i}\right\} _{i=1}^{n}\in \mathfrak{B}%
_{n}\left( X\right) $ and any permutation $\pi $ of $\left\{ 1,\dots
,n\right\} $ we have $\left\{ x_{\pi (i)}\right\} _{i=1}^{n}\in \mathfrak{B}%
_{n}\left( X\right) $, denoting by $\sigma ^{-1}$ the inverse permutation,
we obtain 
\begin{equation*}
\left\Vert \sum_{i=1}^{n}c_{i}x_{i}\right\Vert _{X}=\left\Vert
\sum_{i=1}^{n}a_{\sigma \left( i\right) }x_{i}\right\Vert _{X}=\left\Vert
\sum_{i=1}^{n}a_{i}x_{\sigma ^{-1}\left( i\right) }\right\Vert _{X}\leq
\left\Vert a\right\Vert _{X_{U}\left( n\right) }.
\end{equation*}%
Hence, $\left\Vert c\right\Vert _{X_{U}\left( n\right) }\leq \left\Vert
a\right\Vert _{X_{U}\left( n\right) }$, and by symmetry we obtain that the
norm $\left\Vert \cdot \right\Vert _{X_{U}\left( n\right) }$ is symmetric.

$\left( ii\right) :$ $\Phi _{n}\left( \cdot \right)$. In the same way, as
above, we have%
\begin{equation*}
\Phi _{n}\left( b\right) \leq \left\Vert \sum_{i=1}^{n}b_{i}x_{i}\right\Vert
_{X}\leq \left\Vert \sum_{i=1}^{n}a_{i}x_{i}\right\Vert _{X}.
\end{equation*}%
Thus, $\Phi _{n}\left( b\right) \leq \Phi _{n}\left( a\right)$, and so $\Phi
_{n}\left( \cdot \right)$ is a lattice functional. Also, arguing precisely
as in the case $\left( i\right) $, we obtain $\Phi _{n}\left( c\right) \leq
\Phi _{n}\left( a\right)$, and hence this functional is symmetric.

$\left( iii\right) :$ $\left\Vert \cdot \right\Vert _{X_{L}\left( n\right)
}. $ Let $a=\sum_{k\in F}a^{k}$, where $F\subseteq \mathbb{N}$ is finite and 
$a^{k}=\left( a^k_i\right) _{i=1}^{n}$, $k\in F$. For each $i\in
\left\{1,\dots,n\right\} $ we have 
\begin{equation*}
\left\vert b_{i}\right\vert \leq \left\vert a_{i}\right\vert \leq \sum_{k\in
F}\left\vert a^k_i\right\vert.
\end{equation*}%
One can readily select $b^k_i$ such that $b_{i}=\sum_{k\in F}b^k_i$, $1\le
i\le n$, and $\left\vert b^k_i\right\vert \leq \left\vert a^k_i\right\vert$
for all $k$ and $i$. Then, setting $b^{k}=\left\{ b^k_i\right\}_{i=1}^{n}$, $%
k\in F$, we have $b=\sum_{k\in F}b^{k}$ and $\Phi _{n}\left( b^{k}\right)
\leq \Phi _{n}\left( a^{k}\right)$, which implies 
\begin{equation*}
\left\Vert b\right\Vert _{X_{L}\left( n\right) }\leq \sum_{k\in F}\Phi
_{n}\left( b^{k}\right) \leq \sum_{k\in F}\Phi _{n}\left( a^{k}\right).
\end{equation*}%
In consequence, $\left\Vert b\right\Vert _{X_{L}\left( n\right) }\leq
\left\Vert a\right\Vert _{X_{L}\left( n\right) }$, that is, the norm $%
\left\Vert \cdot \right\Vert _{X_{L}\left( n\right) }$ is lattice.

Next, note that%
\begin{equation*}
\sigma a=\sum_{k\in F}\sigma a^{k},
\end{equation*}%
and hence, by $\left( ii\right)$, 
\begin{equation*}
\left\Vert \sigma a\right\Vert _{X_{L}\left( n\right) }\leq \sum_{k\in
F}\Phi _{n}\left( \sigma a^{k}\right) =\sum_{k\in F}\Phi _{n}\left(
a^{k}\right).
\end{equation*}%
Thus, $\left\Vert \sigma a\right\Vert _{X_{L}\left( n\right) }\leq
\left\Vert a\right\Vert _{X_{L}\left( n\right) }$ and so $\left\Vert \cdot
\right\Vert _{X_{L}\left( n\right) }$ is a symmetric norm.
\end{proof}

An immediate consequence of this lemma is the following

\begin{corollary}
\label{lattice} $\left\Vert \cdot \right\Vert _{X_{L}}$ and $\left\Vert
\cdot \right\Vert _{X_{U}}$ are lattice norms.
\end{corollary}

Our next two propositions state that the operators $I_{n+1}:\mathbb{R}%
^{n+1}\rightarrow \mathbb{R}^{n}$ are contractions with respect to each of
the three functionals considered in the latter lemma.

\begin{proposition}
\label{Lemma_xl_contractions} Let $n\in\mathbb{N}$. For each infinite
dimensional Banach lattice $X$ and all $a\in \mathbb{R}^{n+1}$ we have 
\begin{eqnarray*}
\Phi _{n}\left( I_{n+1}a\right) &\leq &\Phi _{n+1}\left( a\right)
\end{eqnarray*}
and 
\begin{eqnarray*}
\left\Vert I_{n+1}a\right\Vert _{X_{L}\left( n\right) } &\leq &\left\Vert
a\right\Vert _{X_{L}\left( n+1\right)}.
\end{eqnarray*}
\end{proposition}

\begin{proof}
Put $a=\left\{ a_{i}\right\} _{i=1}^{n+1}$ and $b=I_{n+1}a=\left\{
a_{i}\right\} _{i=1}^{n}$.

By Lemma \ref{Lemma_xln_propeties}, $\Phi _{n}$ is lattice monotone.
Consequently, it suffices to prove the result in the special case of $%
a_{n+1}=0.$ We set 
\begin{equation*}
\mathfrak{B}_{n}^{\ast }\left( X\right) :=\left\{ \left\{ x_{i}\right\}
_{i=1}^{n}:\,\mbox{there is}\;x_{n+1}\in X\;\mbox{such that}\;\left\{
x_{i}\right\} _{i=1}^{n+1}\in \mathfrak{B}_{n+1}\left( X\right) \right\} .
\end{equation*}%
Obviously, $\mathfrak{B}_{n}^{\ast }\left( X\right) \subseteq \mathfrak{B}%
_{n}\left( X\right) $, which implies, since $a_{n+1}=0,$ the following: 
\begin{eqnarray*}
\Phi _{n}\left( b\right) &=&\inf \left\{ \left\Vert
\sum_{i=1}^{n}a_{i}x_{i}\right\Vert _{X}:\left\{ x_{i}\right\} _{i=1}^{n}\in 
\mathfrak{B}_{n}\left( X\right) \right\} \\
&\leq &\inf \left\{ \left\Vert \sum_{i=1}^{n}a_{i}x_{i}\right\Vert
_{X}:\left\{ x_{i}\right\} _{i=1}^{n}\in \mathfrak{B}_{n}^{\ast }\left(
X\right) \right\} \\
&=&\inf \left\{ \left\Vert
\sum_{i=1}^{n}a_{i}x_{i}+a_{n+1}x_{n+1}\right\Vert _{X}:\left\{
x_{i}\right\} _{i=1}^{n+1}\in \mathfrak{B}_{n+1}\left( X\right) \right\} \\
&=&\Phi _{n+1}\left( a\right) ,
\end{eqnarray*}%
and the first inequality is proved.

To prove similar inequality for the norm $\left\Vert \cdot \right\Vert
_{X_{L}\left( n\right) }$, write $a=\sum_{k\in F}a^{k}$ for some finite
subset $F\subseteq \mathbb{N}$ and $a^{k}\in \mathbb{R}^{n+1}.$ Then, we
have 
\begin{equation*}
b=I_{n+1}Tr_{n+1}a=\sum_{k\in F}I_{n+1}Tr_{n+1}a^{k},
\end{equation*}%
and, by the first part, 
\begin{equation*}
\left\Vert b\right\Vert _{X_{L}\left( n\right) }\leq \sum_{k\in F}\Phi
_{n}\left( I_{n+1}Tr_{n+1}a^{k}\right) \leq \sum_{k\in F}\Phi _{n+1}\left(
Tr_{n+1}a^{k}\right) \leq \sum_{k\in F}\Phi _{n+1}\left( a^{k}\right) .
\end{equation*}%
Hence, 
\begin{equation*}
\left\Vert b\right\Vert _{X_{L}\left( n\right) }\leq \left\Vert a\right\Vert
_{X_{L}\left( n+1\right) },
\end{equation*}%
and the proof is completed.
\end{proof}

\begin{proposition}
\label{Prop_xu_contractions} Let $X$ be an infinite dimensional Banach
lattice such that $l_{\infty }$ is not finitely lattice representable in $X$%
. The following holds for every positive integer $n$ and $a\in \mathbb{R}%
^{n+1}$%
\begin{equation*}
\left\Vert I_{n+1}a\right\Vert _{X_{U}\left( n\right) }\leq \left\Vert
a\right\Vert _{X_{U}\left( n+1\right) }.
\end{equation*}
\end{proposition}

\begin{proof}
We put again $a=\left\{ a_{i}\right\} _{i=1}^{n+1}$ and $b=I_{n+1}a=\left\{
a_{i}\right\} _{i=1}^{n}$. As in the proof of Proposition \ref%
{Lemma_xl_contractions}, we may assume that $a_{n+1}=0$.

Let $\left\{ x_{i}\right\}_{i=1}^{n}\in \mathfrak{B}_{n}\left( X\right) $
and $\varepsilon >0.$ By Lemma \ref{Lemma_bn_extension}, we can select a
sequence $\left\{ u_{n}\right\}_{i=1}^{n+1}\in \mathfrak{B}_{n+1}\left(
X\right) $ that satisfies one of the conditions $\left(i\right) $ and $%
\left( ii\right) $ of that lemma.

In the case when $\left( i\right) $ is fulfilled, we have%
\begin{equation*}
\left\Vert \sum_{i=1}^{n}a_{i}x_{i}\right\Vert _{X}=\left\Vert
\sum_{i=1}^{n+1}a_{i}u_{i}\right\Vert _{X}\leq \left\Vert a\right\Vert
_{X_{U}\left( n+1\right) },
\end{equation*}%
and the desired result follows.

Assume now that we have $\left( ii\right) $ and let $k$, $\psi $ be as in
the statement of Lemma \ref{Lemma_bn_extension}. Define the vector $%
c=\left\{ c_{i}\right\} _{i=1}^{n+1}\in \mathbb{R}^{n+1}$ by 
\begin{equation*}
c_{i}=\left\{ 
\begin{array}{c}
a_{\psi {\left( i\right) }}:i\neq k,1\leq i\leq n-1 \\ 
c_{n}=a_{k} \\ 
c_{n+1}=0%
\end{array}%
\right\}
\end{equation*}%
Then, we have%
\begin{equation*}
\left\Vert \sum_{i=1}^{n}a_{i}x_{i}\right\Vert _{X}\leq \left\Vert
\sum_{i=1}^{n+1}c_{i}u_{i}\right\Vert _{X}+\left\Vert
\sum_{i=1}^{n+1}c_{i}u_{i}-\sum_{i=1}^{n}a_{i}x_{i}\right\Vert _{X}.
\end{equation*}%
Since 
\begin{eqnarray*}
\sum_{i=1}^{n+1}c_{i}u_{i}-\sum_{i=1}^{n}a_{i}x_{i}
&=&\sum_{i=1}^{n-1}a_{\psi \left( i\right) }x_{\psi \left( i\right)
}+a_{k}u_{n}+0\cdot u_{n+1}-\sum_{i=1}^{n}a_{i}x_{i} \\
&=&\left( \sum_{i=1,i\neq k}^{n}a_{i}x_{i}\right)
+a_{k}u_{n}-a_{k}x_{k}-\sum_{i=1,i\neq k}^{n}a_{i}x_{i} \\
&=&a_{k}\left( u_{n}-x_{k}\right) ,
\end{eqnarray*}%
$\left\vert a_{k}\right\vert \leq \left\Vert a\right\Vert _{X_{U}\left(
n+1\right) }$, $1\leq k\leq n$, and $\left\Vert \cdot \right\Vert
_{X_{U}\left( n+1\right) }$ is a symmetric norm, we conclude 
\begin{equation*}
\left\Vert \sum_{i=1}^{n}a_{i}x_{i}\right\Vert _{X}\leq \left\Vert
\sum_{i=1}^{n+1}c_{i}u_{i}\right\Vert _{X}+\left\Vert u_{n}-x_{k}\right\Vert
_{X}\leq \left\Vert c\right\Vert _{X_{U}\left( n+1\right) }+\varepsilon
\left\Vert a\right\Vert _{X_{U}\left( n+1\right) }=(1+\varepsilon
)\left\Vert a\right\Vert _{X_{U}\left( n+1\right) }.
\end{equation*}%
Thus, since $\varepsilon >0$ is arbitrary, $\left\Vert b\right\Vert
_{X_{U}\left( n\right) }\leq \left\Vert a\right\Vert _{X_{U}\left(
n+1\right) }$, what is required.
\end{proof}

\begin{proposition}
\label{Prop_xl_not_in_c0} Let $X$ be an infinite dimensional Banach lattice
such that $X_{L}$ is not contained in $c_{0}.$ Then, $l_{\infty }$ is
finitely lattice representable in $X$.
\end{proposition}

\begin{proof}
By the assumption, there exists a sequence $a=\left\{ a_{i}\right\}
_{i=1}^{\infty }\in X_{L}$ with $\lim \sup_{i\rightarrow \infty }\left\vert
a_{i}\right\vert >\delta$ for some $\delta>0$. By scaling we may assume that 
$\delta =1.$ Define the sequence $b=\left\{ b_{i}\right\} _{i=1}^{\infty }$
by 
\begin{equation*}
b_{i}=\left\{ 
\begin{array}{c}
1:\left\vert a_{i}\right\vert >1 \\ 
0:\left\vert a_{i}\right\vert \leq 1%
\end{array}%
\right\}
\end{equation*}%
Since $\left\vert b_{i}\right\vert \leq \left\vert a_{i}\right\vert $ for
all $i=1,2,\dots$, by Corollary \ref{lattice}, $b\in X_{L}$ and $%
\|b\|_{X_{L}}\le \|a\|_{X_{L}}$.

Further, for each positive integer $m$ select $k_{m}$ such that the set 
\begin{equation*}
U_{m}:=\left\{ i:\,1\leq i\leq k_{m},b_{i}=1\right\}
\end{equation*}%
has cardinality $m.$ Note that for any sequence $\left\{ y_{i}\right\}
_{i=1}^{k_{m}}\in \mathfrak{B}_{k_{m}}\left( X\right) $ it holds 
\begin{equation*}
\left\Vert \sum_{i\in U_{m}}y_{i}\right\Vert _{X}=\left\Vert
\sum_{i=1}^{k_{m}}b_{i}y_{i}\right\Vert _{X}\leq \Phi _{k_{m}}\left(
b\right) .
\end{equation*}%
Hence, if $b=\sum_{j\in F}b^{j}$ for some finite set $F$ and $b^{j}\in 
\mathbb{R}^{k_{m}}$, by the triangle inequality, we have%
\begin{equation*}
\left\Vert \sum_{i\in U_{m}}y_{i}\right\Vert =\left\Vert
\sum_{i=1}^{k_{m}}\sum_{j\in F}\left\langle b^{j},e_{i}\right\rangle
y_{i}\right\Vert _{X}\leq \sum_{j\in F}\left\Vert
\sum_{i=1}^{k_{m}}\left\langle b^{j},e_{i}\right\rangle y_{i}\right\Vert
_{X}\leq \sum_{j\in F}\Phi _{k_{m}}\left( b^{j}\right) .
\end{equation*}%
Consequently, for any sequence $\left\{ t_{i}\right\} _{i\in U_{m}}$ of
scalars we obtain 
\begin{equation*}
\sup_{i\in U_{m}}\left\vert t_{i}\right\vert \leq \left\Vert \sum_{i\in
U_{m}}t_{i}y_{i}\right\Vert _{X}\leq \sup_{i\in U_{m}}\left\vert
t_{i}\right\vert \left\Vert \sum_{i\in U_{m}}y_{i}\right\Vert _{X}\leq
\sup_{i\in U_{m}}\left\vert t_{i}\right\vert \left\Vert b\right\Vert
_{X_{L}(k_{m})}\leq \left\Vert a\right\Vert _{X_{L}}\sup_{i\in
U_{m}}\left\vert t_{i}\right\vert .
\end{equation*}%
Since the set $U_{m}$ has cardinality $m,$ which is arbitrary, and $%
\left\Vert a\right\Vert _{X_{L}}$ is a constant that does not depend on $m$,
the latter inequality means that $l_{\infty }$ is crudely finitely lattice
representable in $X$ (see Section \ref{estimates}). Since the latter is
equivalent to the finite lattice representability of $l_{\infty }$ in $X$ 
\cite[p.~288]{JMST}, the proof is completed.
\end{proof}

\begin{proposition}
\label{Prop_xl_a-oo} Assume that $l_{\infty }$ is finitely lattice
representable in a Banach lattice $X$. Then $X_{L}$ coincides with $%
l_{\infty }$ isometrically.
\end{proposition}

\begin{proof}
Fix a positive integer $n$ and $\varepsilon >0.$ By the assumption, there
exists a sequence $\left\{ x_{i}\right\} _{i=1}^{n}$ of pair-wise disjoint
elements in $X$ such that%
\begin{equation*}
\sup_{1\leq i\leq n}\left\vert a_{i}\right\vert \leq \left\Vert
\sum_{i=1}^{n}a_{i}x_{i}\right\Vert _{X}\leq \left( 1+\varepsilon \right)
\sup_{1\leq i\leq n}\left\vert a_{i}\right\vert
\end{equation*}%
for any scalar sequence $a=\left\{ a_{i}\right\} _{i=1}^{n}.$ Putting $%
z_{i}:=x_{i}/\left\Vert x_{i}\right\Vert _{X}$, we get $\left\{
z_{i}\right\} _{i=1}^{n}\in \mathfrak{B}_{n}\left( X\right) .$ Hence, in
view of embeddings \eqref{embeddings}, we have 
\begin{eqnarray*}
\left\Vert a\right\Vert _{l_{\infty }^{n}} &\leq &\left\Vert a\right\Vert
_{X_{L}\left( n\right) }\leq \Phi _{n}\left( a\right) \leq \left\Vert
\sum_{i=1}^{n}a_{i}z_{i}\right\Vert _{X}=\left\Vert \sum_{i=1}^{n}\frac{a_{i}%
}{\left\Vert x_{i}\right\Vert _{X}}x_{i}\right\Vert _{X} \\
&\leq &\left( 1+\varepsilon \right) \sup_{1\leq i\leq n}\frac{\left\vert
a_{i}\right\vert }{\left\Vert x_{i}\right\Vert _{X}}\leq \left(
1+\varepsilon \right) \left\Vert a\right\Vert _{l_{\infty }^{n}}
\end{eqnarray*}%
Since $\varepsilon >0$ is arbitrary, we conclude that $\left\Vert
a\right\Vert _{l_{\infty }^{n}}=\left\Vert a\right\Vert _{X_{L}\left(
n\right) }$ for all $n=1,2,\dots $, which implies that $X_{L}=l_{\infty }$
isometrically.
\end{proof}

Prove now the dual result.

\begin{proposition}
\label{Prop-XU-l1} If $l_{1}$ is finitely lattice representable in a Banach
lattice $X$, then $X_{U}$ coincides isometrically with $l_{1}.$
\end{proposition}

\begin{proof}
For every positive integer $n$ and $\varepsilon >0$ we can select a sequence 
$\left\{x_{i}\right\} _{i=1}^{n}$ of pair-wise disjoint elements such that%
\begin{equation*}
\sum_{i=1}^{n}\left\vert a_{i}\right\vert \leq \left\Vert
\sum_{i=1}^{n}a_{i}x_{i}\right\Vert _{X}\leq \left( 1+\varepsilon \right)
\sum_{i=1}^{n}\left\vert a_{i}\right\vert
\end{equation*}%
for any scalar sequence $a=\left\{ a_{i}\right\} _{i=1}^{n}.$ As in the
preceding proof, we put $z_{i}=x_{i}/\left\Vert x_{i}\right\Vert _{X}$. Then 
$\left\{z_{i}\right\} _{i=1}^{n}\in \mathfrak{B}_{n}\left( X\right)$, and
since $1\leq\left\Vert x_{i}\right\Vert _{X}\leq 1+\varepsilon $, $1\le i\le
n$, we have 
\begin{eqnarray*}
\left\Vert a\right\Vert _{l_{1}^{n}} &\geq &\left\Vert a\right\Vert
_{X_{U}\left( n\right) }\geq \left\Vert \sum_{i=1}^{n}a_{i}z_{i}\right\Vert
_{X}=\left\Vert \sum_{i=1}^{n}\frac{a_{i}}{\left\Vert x_{i}\right\Vert _{X}}%
x_{i}\right\Vert _{X} \\
&\geq &\sum_{i=1}^{n}\frac{\left\vert a_{i}\right\vert }{\left\Vert
x_{i}\right\Vert _{X}}\geq \left( 1+\varepsilon \right)
^{-1}\sum_{i=1}^{n}\left\vert a_{i}\right\vert =\left( 1+\varepsilon \right)
^{-1}\left\Vert a\right\Vert _{l_{1}^{n}}
\end{eqnarray*}%
This implies $\left\Vert a\right\Vert _{l_{1}^{n}}=\left\Vert
a\right\Vert_{X_{U}\left( n\right) }$ and thus the proof is complete.
\end{proof}

$\medskip$

As a result, all the pieces needed for the proof of Theorem \ref%
{Th_XL_XU_Prop} are in place.

\begin{proof}[Proof of Theorem \protect\ref{Th_XL_XU_Prop}]
Prove first the claim for $X_{L}.$ By Lemmas \ref{Lemma_xln_propeties} and %
\ref{Lemma_xl_contractions}, $\left\Vert \cdot \right\Vert _{X_{L}\left(
n\right) }$ is a lattice, symmetric norm for each positive integer $n$ and
the operators $I_{n}$ are contractions with respects to these norms. Hence,
if $X_{L}$ is embedded in $c_{0}$, then Proposition \ref%
{Prop_main_Prop_xl_xu} may be applied and we conclude that $X_{L}$ is a r.i.
Banach sequence lattice. In the case when $X_{L}$ is not embedded in $c_{0},$
from Propositions \ref{Prop_xl_not_in_c0} and \ref{Prop_xl_a-oo} it follows
that $X_{L}$ coincides isometrically with $l_{\infty }$, and hence it is a
r.i. Banach sequence lattice as well.

Proceeding with the case of $X_{U}$, observe that, by the assumption, $%
l_{\infty }$ fails to be finitely lattice representable in $X$, and so,
using Proposition \ref{Prop_xu_contractions}, we have that the maps $I_{n}$
are contractive with respect to these norms. Moreover, by Lemma \ref%
{Lemma_xln_propeties}, $\left\Vert \cdot \right\Vert _{X_{U}\left( n\right)
} $ is a lattice, symmetric norm for each positive integer $n.$ Finally,
from Proposition \ref{Prop_dia} it follows that $X$ satisfies a lower $p$%
-estimate for some $p<\infty .$ Hence, for every $n\in \mathbb{N}$ and any
sequences $\left\{ x_{i}\right\} _{i=1}^{n}\in \mathfrak{B}_{n}\left(
X\right) $ and $\left\{ a_{i}\right\} _{i=1}^{n}$ of scalars it follows 
\begin{equation*}
\left( \sum_{i=1}^{n}|a_{i}|^{p}\right) ^{1/p}\leq M_{\left[ p\right]
}\left( X\right) \left\Vert \sum_{i=1}^{n}a_{i}x_{i}\right\Vert _{X}\leq M_{%
\left[ p\right] }\left( X\right) \left\Vert
\sum_{i=1}^{n}a_{i}e_{i}\right\Vert _{X_{U}},
\end{equation*}%
i.e., $X_{U}$ is embedded into $l_{p}$ and thus also in $c_{0}.$ Thus,
applying Proposition \ref{Prop_main_Prop_xl_xu}, we conclude that $X_{U}$ is
a r.i. sequence Banach lattice.
\end{proof}

\vskip0.6cm

\section*{\label{Appendix}Appendix: A description of the optimal upper
sequence lattices for Orlicz spaces.}

Recall that, according to Example \ref{Lorentz}, optimal upper and lower
sequence lattices for the $L_{p,q}$-spaces are just some $l_{r}$-spaces. As
well known (see e.g. \cite{LT-73a,LT73}), comparing with the Lorentz spaces,
the structure of disjoint sequences in Orlicz spaces is much more
complicated. In particular, in general, an Orlicz space $L_{M}$ need not to
admit an upper $\delta (L_{M})$-estimate or a lower $\sigma (L_{M})$%
-estimate (as above, $\delta (X)$ and $\sigma (X)$ are the Grobler-Dodds
indices of a Banach lattice $X$). Therefore, we come to the problem of
identification of optimal sequence lattices for this class of r.i. spaces.
In this section, we present a description of optimal upper lattices for
separable Orlicz spaces as intersections of some special Musielak-Orlicz
sequence spaces. $\medskip $

We start with an assertion that reduces the consideration of issues related
to pairwise disjoint functions to that of a simpler case of multiples of
characteristic functions of pairwise disjoint sets.

\begin{proposition}
\label{Orlicz spaces} Let $M$ be an Orlicz function such that $M\in \Delta
_{2}^{\infty }$ with the constant $K$. For every $n\in \mathbb{N}$ and
arbitrary pairwise disjoint functions $y_{k}$, $k=1,\dots ,n$, there exist
two sequences $\{B_{k}\}_{k=1}^{n}$ and $\{B_{k}^{\prime }\}_{k=1}^{2n}$ of
pairwise disjoint subsets of $[0,1]$, $r_{k}\in \mathbb{R}$, $k=1,\dots ,n$,
and $r_{k}^{\prime }\in \mathbb{R}$, $k=1,\dots ,2n$, such that for the
functions $h_{k}:=r_{k}\chi _{B_{k}}$, $k=1,\dots ,n$, and $%
f_{k}:=r_{k}^{\prime }\chi _{B_{k}^{\prime }}$, $k=1,\dots ,2n$, we have 
\begin{equation}
\frac{1}{4}\Vert y_{k}\Vert _{L_{M}}\leq \Vert h_{k}\Vert _{L_{M}}\leq \Vert
y_{k}\Vert _{L_{M}},\;\;\frac{1}{2}\Vert y_{k}\Vert _{L_{M}}\leq \Vert
f_{k}\Vert _{L_{M}}\leq \frac{3}{2}\Vert y_{k}\Vert _{L_{M}},\;\;k=1,\dots
,n,  \label{suforlicz}
\end{equation}%
and 
\begin{equation}
\Big\|\sum_{k=1}^{n}h_{k}\Big\|_{L_{M}}\leq \Big\|\sum_{k=1}^{n}y_{k}\Big\|%
_{L_{M}}\leq (K+1)\Big\|\sum_{k=1}^{2n}f_{k}\Big\|_{L_{M}}.
\label{suffcondfororlicz}
\end{equation}
\end{proposition}

\begin{proof}
Clearly, without loss of generality, we can assume that given functions $%
y_k\in L_M$, $k=1,\dots,n$, are positive. Moreover, since $%
M\in\Delta_2^\infty$, the space $L_M$ is separable, and, consequently, it
can be assumed also that $y_k$ are bounded functions.

For each $1\leq k\leq n$ we set 
\begin{equation*}
c_{k}:=\frac{\Vert y_{k}\Vert _{L_{M}}}{2\varphi _{L_{M}}(m(\mathrm{supp}%
\,y_{k}))},u_{k}(t):=%
\begin{cases}
y_{k}(t), & \mathrm{if}~y_{k}(t)\geq c_{k}, \\ 
0, & \mathrm{if}~y_{k}(t)<c_{k}%
\end{cases}%
\;\;\mbox{and}\;\;g_{k}(t):=c_{k}\chi _{\mathrm{supp}y_{k}\setminus \mathrm{%
supp}u_{k}}(t).
\end{equation*}%
Then, it follows 
\begin{equation}
\sum_{k=1}^{n}u_{k}\leq \sum_{k=1}^{n}y_{k}\leq
\sum_{k=1}^{n}u_{k}+\sum_{k=1}^{n}g_{k}.  \label{ineq1}
\end{equation}%
Observe also that 
\begin{equation}
\Vert g_{k}\Vert _{L_{M}}=c_{k}\varphi _{L_{M}}(m(\mathrm{supp}%
y_{k}\setminus \mathrm{supp}u_{k}))\leq \frac{1}{2}\Vert y_{k}\Vert
_{L_{M}},\;\;k=1,\dots ,n  \label{ineq2}
\end{equation}%
($\varphi _{L_{M}}$ is the fundamental function of $L_{M}$; see formula %
\eqref{fundamfunc}), and 
\begin{equation}
\frac{1}{2}\Vert y_{k}\Vert _{L_{M}}=\Vert y_{k}\Vert _{L_{M}}-\Vert
c_{k}\chi _{\mathrm{supp}y_{k}}\Vert _{L_{M}}\leq \Vert u_{k}\Vert
_{L_{M}}\leq \Vert y_{k}\Vert _{L_{M}},\;\;k=1,\dots ,n.  \label{ineq3}
\end{equation}

Next, we estimate the norm $\Vert \sum_{k=1}^{n}u_{k}\Vert _{L_{M}}$. To
this end, we show that there is $r_{k}\in \lbrack c_{k},\sup_{t}u_{k}(t)]$
such that 
\begin{equation}
M(r_{k})=M\left( \frac{r_{k}}{\Vert u_{k}\Vert _{L_{M}}}\right)
\int_{0}^{1}M(u_{k}(t))\,dt.  \label{equationforri}
\end{equation}%
Indeed, let us consider the function 
\begin{equation*}
H_{k}(t):=\frac{M(u_{k}(t))}{M\left( \frac{u_{k}(t)}{\Vert u_{k}\Vert
_{L_{M}}}\right) },~t\in \mathrm{supp}\,u_{k}.
\end{equation*}%
From the equality $\int_{0}^{1}M(\frac{u_{k}(t)}{\Vert u_{k}\Vert _{L_{M}}}%
)dt=1$ it follows that 
\begin{equation*}
\inf_{t\in \mathrm{supp}\,u_{k}}H_{k}(t)\leq \int_{0}^{1}M(u_{k}(t))\,dt\leq
\sup\limits_{t\in \mathrm{supp}\,u_{k}}H_{k}(t).
\end{equation*}%
Thus, since $\inf\limits_{t\in \mathrm{supp}\,u_{k}}u_{k}(t)\geq c_{k}$, by
continuity of $M$, equality \eqref{equationforri} holds for some $r_{k}$
such that $r_k\in [c_{k},\sup_{t}u_{k}(t)]$.

Further, assuming as we can that the functions $M$ and $\varphi_{L_M}$ are
strictly increasing, define the real numbers $d_k\in [0, 1]$, $k = 1,2,
\ldots, n$, as follows: 
\begin{equation*}
d_k: = 
\begin{cases}
\varphi_{L_M}^{-1}\left(\frac{\|u_k\|_{L_M}}{r_k}\right), & \mathrm{if}~
\|u_k\|_{L_M} \leq r_k \varphi_{L_M}(m(\mathrm{supp}\, y_k)), \\ 
m(\mathrm{supp}\, y_k), & \mathrm{if} ~ \|u_k\|_{L_M} > r_k \varphi_{L_M}(m(%
\mathrm{supp}\, y_k)).%
\end{cases}%
\end{equation*}
Clearly, the definition of $d_k$ implies that 
\begin{equation}  \label{riphidi2}
r_k\varphi_{L_M} (d_k) \leq \|u_k\|_{L_M}.
\end{equation}
Conversely, $r_k \varphi_{L_M}(d_k)=\|u_k\|_{L_M}$ if $\|u_k\|_{L_M} \leq
r_k \varphi_{L_M}(m(\mathrm{supp}\, y_k))$. Otherwise, since $r_k \geq c_k,$
in view of \eqref{ineq3} and the definition of $c_k$, we obtain 
\begin{equation*}
r_k \varphi_{L_M}(d_k) \geq c_k \varphi_{L_M}(m(\mathrm{supp}\,
y_k))\ge\frac12 \|u_k\|_{L_M}.
\end{equation*}
Thus, summing up, we conclude that 
\begin{equation}  \label{riphidi}
r_k \varphi_{L_M}(d_k) \geq \frac{1}{2} \|u_k\|_{L_M}.
\end{equation}

Now, observe that from inequality \eqref{riphidi} and formula %
\eqref{fundamfunc} for the function $\varphi_{L_M}$ it follows that 
\begin{equation*}
d_k \geq \varphi_{L_M}^{-1}( \|u_k\|_{L_M} /(2 r_k)).
\end{equation*}
and 
\begin{equation*}
\varphi_{L_M}^{-1}(u)=\frac{1}{M(1/u)},\;\;0<t\le 1,
\end{equation*}
respectively. Hence, taking into account that $M\in \Delta_2^\infty$ with
constant $K$ and applying (\ref{equationforri}), we obtain 
\begin{eqnarray}
d_k M(r_k) &\geq& \varphi_{L_M}^{-1}( \|u_k\|_{L_M} /(2 r_k))M(r_k)= \frac{%
M(r_k)}{M\left(\frac{2r_k}{\|u_k\|_{L_M}}\right)}  \notag \\
&\ge& \frac{1}{K} \frac{M(r_k)}{M\left(\frac{r_k}{\|u_k\|_{L_M}}\right)} = 
\frac{1}{K} \int\limits_0^1 M(u_k(t))\,dt.  \label{riphidi3}
\end{eqnarray}

In the converse direction, from the equality $1/d_{k}=M\left(
1/\varphi_{L_{M}}(d_{k})\right) $ (see \eqref{fundamfunc}), combined with ( %
\ref{riphidi2}) and \eqref{equationforri}, it follows 
\begin{equation}
d_{k}M(r_{k})=\frac{M(r_{k})}{M(\frac{1}{\varphi _{L_{M}}(d_{k})})}\leq 
\frac{M(r_{k})}{M\left( \frac{r_{k}}{\Vert u_{k}\Vert _{L_{M}}}\right) }%
=\int\limits_{0}^{1}M(u_{k}(t))\,dt.  \label{estimate41}
\end{equation}

Furthermore, by the definition of $d_{k}$, we have $d_{k}\leq m(\mathrm{supp}%
\,y_{k}) $. Therefore, we can define the following functions $%
h_{k}(t):=r_{k}\chi _{B_{k}}(t),$ where $B_{k}\subset \mathrm{supp}\,y_{k}$
and $m(B_{k})=d_{k}.$ Since $\Vert h_{k}\Vert _{L_{M}}=r_{k}\varphi
_{L_{M}}(d_{k})$, according to \eqref{riphidi2} and \eqref{riphidi}, it
holds 
\begin{equation*}
\frac{1}{2}\Vert u_{k}\Vert _{L_{M}}\leq \Vert h_{k}\Vert _{L_{M}}\leq \Vert
u_{k}\Vert _{L_{M}},\;\;k=1,2,\ldots ,n.
\end{equation*}%
Hence, from \eqref{ineq3} it follows 
\begin{equation}
\frac{1}{4}\Vert y_{k}\Vert _{L_{M}}\leq \Vert h_{k}\Vert _{L_{M}}\leq \Vert
y_{k}\Vert _{L_{M}},\;\;k=1,2,\ldots ,n.  \label{estimate41a}
\end{equation}

Moreover, since the functions $h_{k}$ (respectively, $u_{k}$) are pairwise
disjoint, in view of estimate \eqref{estimate41}, we conclude that 
\begin{eqnarray*}
\int_{0}^{1}M\left( \sum_{k=1}^{n}h_{k}(t)\right) \,dt
&=&\sum_{k=1}^{n}d_{k}M(r_{k})\leq \sum_{k=1}^{n}\int_{0}^{1}M(u_{k}(t))\,dt
\\
&\leq &\int_{0}^{1}M\left( \sum_{k=1}^{n}y_{k}(t)\right) \,dt.
\end{eqnarray*}%
Conversely, by \eqref{riphidi3}, we have 
\begin{equation*}
\sum_{k=1}^{n}\int_{0}^{1}M(u_{k}(t))\,dt\leq K\int_{0}^{1}M\left(
\sum_{k=1}^{n}h_{k}(t)\right) \,dt.
\end{equation*}%
Therefore, since $M$ is convex and $K\geq 1$, it follows that 
\begin{equation*}
\Big\|\sum_{k=1}^{n}h_{k}\Big\|_{L_{M}}\leq \Big\|\sum_{k=1}^{n}u_{k}\Big\|%
_{L_{M}}\leq K\Big\|\sum_{k=1}^{n}h_{k}\Big\|_{L_{M}}.
\end{equation*}%
Noting that the collection $\{g_{k},h_{k}\}_{k=1}^{n}$ consists of $2n$
pairwise disjoint functions, we relabel them as $f_{k}$, $k=1,2,\dots ,2n$.
Then, by \eqref{ineq1} and the last inequality, we obtain 
\begin{eqnarray*}
\Big\|\sum_{k=1}^{n}h_{k}\Big\|_{L_{M}}\leq \Big\|\sum_{k=1}^{n}y_{k}\Big\|%
_{L_{M}} &\leq &\Big\|\sum_{k=1}^{n}u_{k}\Big\|_{L_{M}}+\Big\|%
\sum_{k=1}^{n}g_{k}\Big\|_{L_{M}} \\
&\leq &K\Big\|\sum_{k=1}^{n}h_{k}\Big\|_{L_{M}}+\Big\|\sum_{k=1}^{n}g_{k}%
\Big\|_{L_{M}} \\
&\leq &(K+1)\Big\|\sum_{k=1}^{2n}f_{k}\Big\|_{L_{M}},
\end{eqnarray*}%
and hence \eqref{suffcondfororlicz} is proved. Since inequalities %
\eqref{suforlicz} follow from \eqref{estimate41a} and \eqref{ineq2}, the
proof is completed.
\end{proof}

From Proposition \ref{Orlicz spaces} and its proof we obtain

\begin{corollary}
\label{cor4} Let $M$ be an Orlicz function such that $M\in \Delta
_{2}^{\infty }$. Then, 
\begin{equation*}
\Vert a\Vert _{(L_{M})_{U}}\asymp \sup \left\{ \Big\|\sum_{k=1}^{n}a_{k}%
\frac{\chi _{F_{k}}}{\varphi _{L_{M}}(m(F_{k}))}\Big\|_{L_{M}}:\,n\in 
\mathbb{N},F_{k}\subset \lbrack 0,1]\;\mbox{pairwise disjoint}\right\} ,
\end{equation*}%
with constants independent of $a=(a_{k})_{k=1}^{\infty }$.
\end{corollary}

Recalling that the space $(L_{M})_{U}$ is rearrangement invariant (see
Theorem \ref{Th_XL_XU_Prop}), denote by $\phi_U$ its fundamental function,
i.e., $\phi_{U}(n):=\|\sum_{k=1}^{n}e_{k}\|_{(L_{M})_{U}}$, $n\in\mathbb{N}$%
. Also, let ${\Phi}_{g}$ be the dilation function of a function $%
g:\,(0,\infty )\to (0,\infty )$ for large values of arguments defined by 
\begin{equation*}
{\Phi}_{g}(u):=\sup_{v\geq \max (1,1/u)}\frac{g(vu)}{g(v)},\;\;u>0.
\end{equation*}

\begin{corollary}
\label{cor5} Let $M$ be an Orlicz function such that $M\in \Delta
_{2}^{\infty }$. Then, 
\begin{equation*}
\phi _{U}(n)\asymp {\Phi }_{M^{-1}}(n),\;\;n\in \mathbb{N},
\end{equation*}%
where $M^{-1}$ is the inverse function for $M$.
\end{corollary}

\begin{proof}
By Corollary \ref{cor4}, we have 
\begin{equation}  \label{estimate42}
\phi_U(n)\asymp \sup\left\{\Big\|\sum_{k=1}^n \frac{\chi_{F_k}}{%
\varphi_{L_M}(m(F_k))}\Big\|_{L_M}:\,F_k\subset [0,1]\;%
\mbox{are pairwise
disjoint}\right\}.
\end{equation}

Let $n\in\mathbb{N}$ and let $F_k\subset [0,1]$, $k=1,\dots, n$, be pairwise
disjoint. Then, by formula \eqref{fundamfunc}, 
\begin{equation*}
\Big\|\sum_{k=1}^n \frac{\chi_{F_k}}{\varphi_{L_M}(m(F_k))}\Big\|%
_{L_M}=\inf\left\{\lambda>0:\,\sum_{k=1}^n M\left(\frac{M^{-1}(1/m(F_k))}{%
\lambda}\right)m(F_k)\le 1\right\}.
\end{equation*}
Next, we write 
\begin{eqnarray*}
\sum_{k=1}^n M\left(\frac{M^{-1}(1/m(F_k))}{{\Phi}_{M^{-1}}(n)}%
\right)m(F_k)&=&\sum_{k:\,m(F_k)\ge 1/n} M\left(\frac{M^{-1}(1/m(F_k))}{{\Phi%
}_{M^{-1}}(n)}\right)m(F_k) \\
&+&\sum_{k:\,m(F_k)<1/n} M\left(\frac{M^{-1}(1/m(F_k))}{{\Phi}_{M^{-1}}(n)}%
\right)m(F_k) \\
&=& (I)+(II).
\end{eqnarray*}
Observe that 
\begin{equation*}
(I)\le \sum_{k:\,m(F_k)\ge 1/n} M\left(\frac{M^{-1}(n)}{{\Phi}_{M^{-1}}(n)}%
\right)m(F_k)\le M(1)=1
\end{equation*}
and 
\begin{equation*}
(II)\le \sum_{k:\,m(F_k)<1/n} M\left(\frac{%
M^{-1}(1/(m(F_k)n))M^{-1}(1/m(F_k))}{M^{-1}(1/m(F_k))}\right)m(F_k)=\sum_{k:%
\,m(F_k)<1/n}\frac1n\le 1.
\end{equation*}
Summing up, we obtain 
\begin{equation*}
\Big\|\sum_{k=1}^n \frac{\chi_{F_k}}{\varphi_{L_M}(m(F_k))}\Big\|_{L_M}\le 2{%
\Phi}_{M^{-1}}(n),
\end{equation*}
for every $n\in\mathbb{N}$ and all pairwise disjoint $F_k\subset [0,1]$, $%
k=1,\dots,n$. Consequently, in view of \eqref{estimate42}, it follows 
\begin{equation*}
\phi_U(n)\preceq {\Phi}_{M^{-1}}(n),\;\;n\in\mathbb{N}.
\end{equation*}

Conversely, without loss of generality, assume that 
\begin{equation*}
{\Phi}_{M^{-1}}(n)=\frac{M^{-1}(nv_n)}{M^{-1}(v_n)}
\end{equation*}
for some $v_n\ge 1$. Let $F_k\subset [0,1]$, $k=1,\dots, n$, be arbitrary
pairwise disjoint subsets of $[0,1]$ such that $m(F_k)=(nv_n)^{-1}$. Then, 
\begin{eqnarray*}
\phi_U(n)\succeq \Big\|\sum_{k=1}^n \frac{\chi_{F_k}}{\varphi_{L_M}(m(F_k))}%
\Big\|_{L_M}&=&\inf\left\{\lambda>0:\,M\left(\frac{M^{-1}(nv_n)}{\lambda}%
\right)\frac{1}{v_n}\le 1\right\} \\
&=&\frac{M^{-1}(nv_n)}{M^{-1}(v_n)}={\Phi}_{M^{-1}}(n).
\end{eqnarray*}
\end{proof}

Recall that a family of Banach spaces $\{X_\alpha\}_{\alpha \in \mathcal{A}}$
forms a \textit{strongly compatible scale} if there exists a Banach space $%
\tilde X$ such that $X_\alpha\overset{1}{\hookrightarrow}\tilde X$, $\alpha
\in \mathcal{A}$.

Let $\{X_\alpha\}_{\alpha \in \mathcal{A}}$ be a strongly compatible scale.
We set 
\begin{equation*}
\Delta (X_\alpha)_{\alpha \in \mathcal{A}}:= \{x\in \cap _{\alpha \in 
\mathcal{A}}X_\alpha:\ \|x\|_{\Delta (X_\alpha)}:= \sup _{\alpha \in 
\mathcal{A}}\|x\|_{X_\alpha}<\infty\}.
\end{equation*}
Then, $(\Delta (X_\alpha)_{\alpha \in \mathcal{A}}, \|\cdot \|_{\Delta
(X_\alpha)})$ is a Banach space with the following properties:

(i) $\Delta (X_\alpha)_{\alpha \in \mathcal{A}}\overset{1}{\hookrightarrow}
X_\alpha$, $\forall \alpha \in \mathcal{A}$;

(ii) If $F$ is a Banach space such that $F\overset{1}{\hookrightarrow}
X_\alpha$, $\forall\alpha \in \mathcal{A}$, then $F\overset{1}{%
\hookrightarrow}\Delta (X_\alpha)_{\alpha \in \mathcal{A}}$. $\smallskip$

Let $M$ be an Orlicz function, $M_v (u):=M(uv)/M(v)$, $u\ge 0$, $v>0$. We
consider the family of the Musielak-Orlicz sequence spaces $\{l_{M_{\bar{%
\beta}}}\}_{\bar{\beta}\in \mathcal{B}}$, where $\mathcal{B}$ is the set of
all sequences $\bar{\beta}=(\beta_k)_{k=1}^\infty$ such that $%
\sum_{k=1}^\infty{1}/{M(\beta_k)}\le 1$. Recall that the norm of the
Musielak-Orlicz sequence space $l_{M_{\bar{\beta}}}$ is defined by 
\begin{equation*}
\|(a_k)\|_{l_{M_{\bar{\beta}}}}:=\Big\{\lambda>0:\,\sum_{k=1}^\infty
M_{\beta_k}\Big(\frac{a_k}{\lambda}\Big)\le 1\Big\}
\end{equation*}
(see e.g. \cite{Mus,Woo}). One can easily check that $l_{M_{\bar{\beta}}}%
\overset{1}{\hookrightarrow} l_\infty$, for each $\bar{\beta}\in \mathcal{B}%
, $ and hence this family is a strongly compatible scale.

\begin{theorem}
\label{Th. Oricz} For every Orlicz function $M$ such that $M\in \Delta
_{2}^{\infty }$ we have $(L_{M})_{U}=\Delta (l_{M_{\bar{\beta}}})_{{\bar{%
\beta}}\in \mathcal{B}}$ (with equivalence of norms). Moreover, the
following embeddings hold: 
\begin{equation}
l_{\Phi _{M}}\hookrightarrow (L_{M})_{U}\hookrightarrow l_{p_{M}},
\label{embedd}
\end{equation}%
where $\Phi _{M}$ is the dilation function of $M$ for large values of
arguments and 
\begin{equation*}
p_{M}:=\sup \{p\geq 1:\,M(uv)\leq Cu^{p}M(v)\;\;\mbox{for some}\;\;C>0\;\;%
\mbox{and
all}\;\;0<u\leq 1,uv\geq 1\}.
\end{equation*}
\end{theorem}

\begin{proof}
Without loss of generality, we will assume that $a_k\ge 0$, $k=1,2,\dots$.

Let $F_{k}\subset \lbrack 0,1]$, $k=1,2,\dots $, be arbitrary pairwise
disjoint sets of positive measure. Then, setting $\beta
_{k}:=M^{-1}(1/m(F_{k}))$, $k=1,2,\dots $, we see that 
\begin{equation*}
\sum_{k=1}^{\infty }\frac{1}{M(\beta _{k})}=\sum_{k=1}^{\infty }m(F_{k})\leq
1,
\end{equation*}%
and hence ${\bar{\beta}}=(\beta _{k})_{k=1}^{\infty }\in \mathcal{B}$.
Conversely, for each ${\bar{\beta}}=(\beta _{k})_{k=1}^{\infty }\in \mathcal{%
B}$ we can find pairwise disjoint sets $F_{k}\subset \lbrack 0,1]$, $%
m(F_{k})>0$, such that $\beta _{k}=M^{-1}(1/m(F_{k}))$, $k=1,2,\dots $ These
observations together with Corollary \ref{cor4} imply, for every $%
a=\{a_{k}\}_{k=1}^{\infty }$, the following: 
\begin{eqnarray*}
\Vert a\Vert _{(L_{M})_{U}} &\asymp &\sup \left\{ \Big\|\sum_{k=1}^{n}a_{k}%
\frac{\chi _{F_{k}}}{\varphi _{L_{M}}(m(F_{k}))}\Big\|_{L_{M}}:\,n\in 
\mathbb{N},F_{k}\subset \lbrack 0,1]\;\mbox{pairwise disjoint}\right\} \\
&=&\sup_{n\in \mathbb{N},F_{k}\mbox{\small{pairwise disjoint}}}\inf \left\{
\lambda >0:\,\sum_{k=1}^{n}M\left( \frac{a_{k}M^{-1}({1}/{m(F_{k})})}{%
\lambda }\right) m(F_{k})\leq 1\right\} \\
&=&\sup_{\bar{\beta} \in \mathcal{B},n\in \mathbb{N}}\inf \left\{ \lambda
>0:\,\sum_{k=1}^{n}\frac{M\left( \frac{a_{k}}{\lambda }\beta _{k}\right) }{%
M(\beta _{k})}\leq 1\right\} \\
&=&\sup_{\bar{\beta} \in \mathcal{B},n\in \mathbb{N}}\inf \left\{ \lambda
>0:\,\sum_{k=1}^{n}M_{\beta _{k}}(a_{k}/\lambda )\leq 1\right\} \\
&=&\Vert a\Vert _{\Delta (l_{M_{\bar{\beta} }})_{\bar{\beta} \in \mathcal{B}%
}}.
\end{eqnarray*}%
Thus, $(L_{M})_{U}=\Delta (l_{M_{\bar{\beta}}})_{\bar{\beta}\in \mathcal{B}}$%
, with equivalence of norms.

To prove the left-hand side embedding in \eqref{embedd} we assume that $%
\Vert (a_{k})\Vert _{l_{\Phi _{M}}}\leq 1$. This implies that $%
\sum_{k=1}^{n}\Phi _{M}(a_{k})\leq 1$ for every $n\in \mathbb{N}$. Then, for
each sequence $\bar{\beta}=(\beta _{k})_{k=1}^{\infty }\in \mathcal{B}$ we
have 
\begin{equation*}
\sum_{k=1}^{n}\frac{M\left( {a_{k}}\beta _{k}\right) }{M(\beta _{k})}%
=\sum_{k:\;a_{k}\beta _{k}\leq 1}\frac{M\left( {a_{k}}\beta _{k}\right) }{%
M(\beta _{k})}+\sum_{k:\;a_{k}\beta _{k}>1}\frac{M\left( {a_{k}}\beta
_{k}\right) }{M(\beta _{k})}\leq \sum_{k=1}^{n}\frac{1}{M(\beta _{k})}%
+\sum_{k=1}^{n}\Phi _{M}(a_{k})\leq 2.
\end{equation*}%
Thus, $\Vert (a_{k})\Vert _{l_{M_{\bar{\beta}}}}\leq 2$ for every $\bar{\beta%
}=(\beta _{k})_{k=1}^{\infty }\in \mathcal{B}$ and therefore, by the first
assertion of the theorem, it follows 
\begin{equation*}
\Vert (a_{k})\Vert _{(L_{M})_{U}}\leq C\Vert (a_{k})\Vert _{\Delta (l_{M_{%
\bar{\beta}}})_{\bar{\beta}\in \mathcal{B}}}\leq 2C.
\end{equation*}

Furthermore, by \cite{KMP97}, $p_M$ is the supremum of the set of all $p\ge
1 $ such that $M$ is equivalent to a $p$-convex function on the interval $%
[1,\infty)$, or equivalently $p_M$ is the supremum of the set of all $p\ge 1$
such that the Orlicz space $L_M[0,1]$ admits an upper $p$-estimate. Thus, by
using the notation of this paper, we have $p_M=\delta(L_M[0,1])$ and hence
the right-hand side embedding in \eqref{embedd} is a consequence of
Proposition \ref{coincidence with lp}(i). This completes the proof.
\end{proof}

\begin{remark}
Informally, the Orlicz space $l_{\Phi _{M}}$ is located rather "close" to
the space $(L_{M})_{U}$, because the fundamental functions of these spaces
are equivalent. Indeed, let $n\in \mathbb{N}$ and $\varepsilon >0$ be
arbitrary. Then, by definition, $\phi _{l_{\Phi _{M}}}(n)=1/u_{n}$, where $%
u_{n}$ satisfies the conditions: 
\begin{equation*}
\frac{M(u_{n}v_{n})}{M(v_{n})}\geq (1-\varepsilon )\frac{1}{n}\;\;%
\mbox{for
some}\;\;v_{n}\geq 1/u_{n}\;\;\mbox{and}\;\;\frac{M(u_{n}v)}{M(v)}\leq \frac{%
1}{n}\;\;\mbox{for
all}\;\;v\geq 1/u_{n}.
\end{equation*}%
In particular, from the last estimate it follows that $M(1/u_{n})\geq n$.
Therefore, since $M^{-1}$ is concave, we get 
\begin{equation*}
u_{n}\geq \frac{M^{-1}((1-\varepsilon )s_{n}/n)}{M^{-1}(s_{n})}\geq
(1-\varepsilon )\frac{M^{-1}(s_{n}/n)}{M^{-1}(s_{n})},\;\;\mbox{where}%
\;\;s_{n}=M(v_{n})\geq M(1/u_{n}),
\end{equation*}%
and 
\begin{equation*}
u_{n}\leq \frac{M^{-1}(s/n)}{M^{-1}(s)},\;\;\mbox{for all}\;\;s\geq
M(1/u_{n}).
\end{equation*}%
Thus, 
\begin{equation*}
\frac{1}{\phi _{l_{\Phi _{M}}}(n)}=\inf_{s\geq M(1/u_{n})}\frac{M^{-1}(s/n)}{%
M^{-1}(s)},
\end{equation*}%
whence 
\begin{equation*}
\phi _{l_{\Phi _{M}}}(n)=\sup_{s\geq M(1/u_{n})}\frac{M^{-1}(s)}{M^{-1}(s/n)}%
=\sup_{t\geq M(1/u_{n})/n}\frac{M^{-1}(tn)}{M^{-1}(t)}.
\end{equation*}%
Consequently, since $M(1/u_{n})/n\geq 1$, by Corollary \ref{cor5}, we have 
\begin{equation*}
\phi _{l_{\Phi _{M}}}(n)\leq \mathcal{M}_{M^{-1}}(n)\preceq \phi
_{(L_{M})_{U}}(n),\;\;n\in \mathbb{N}.
\end{equation*}%
It remains to note that the opposite inequality follows from the left-hand
side embedding \eqref{embedd}.
\end{remark}

\vskip0.5cm

\textbf{Data availability statement.} \textit{All data generating or
analysed during this study are included in this published article.}

\end{document}